\documentclass[letter,10pt]{article}
\usepackage{amssymb,amsmath,amsfonts,amsthm}
\usepackage{mathrsfs}
\usepackage[left=1 in, right=1 in,top=1 in, bottom=1 in]{geometry}

\usepackage{color}
\usepackage{authblk}
\usepackage{hyperref}

\newtheorem{theorem}{Theorem}[section]
\newtheorem{lemma}[theorem]{Lemma}

\newtheorem{corollary}[theorem]{Corollary}

\newtheorem{remark}[theorem]{Remark}
\newtheorem{definition}[theorem]{Definition}

\makeatletter
\renewcommand \theequation {%
\ifnum \c@section>\z@ \@arabic\c@section.%
\fi\@arabic\c@equation} \@addtoreset{equation}{section}
\makeatother

\providecommand{\abs}[1]{\left\vert#1\right\vert}
\providecommand{\nm}[1]{\left\Vert#1\right\Vert}
\providecommand{\br}[1]{\left\langle #1 \right\rangle}

\providecommand{\vnm}[2]{\left\Vert#1\right\Vert_{L^{\nu}_{#2}}}
\providecommand{\vnnm}[1]{\left\vert\kern-0.25ex\left\vert\kern-0.25ex\left\vert#1\right\vert\kern-0.25ex\right\vert\kern-0.25ex\right\vert_{\nu}}

\providecommand{\vnnmz}[1]{\left\vert\kern-0.25ex\left\vert\kern-0.25ex\left\vert#1\right\vert\kern-0.25ex\right\vert\kern-0.25ex\right\vert_{\nu,0}}

\providecommand{\tnm}[2]{\left\Vert#1\right\Vert_{L^2_{#2}}}

\newcommand{\vertiii}[1]{{\left\vert\kern-0.25ex\left\vert\kern-0.25ex\left\vert #1 \right\vert\kern-0.25ex\right\vert\kern-0.25ex\right\vert}}

\newcommand{\nnm}[1]{{\left\vert\kern-0.25ex\left\vert\kern-0.25ex\left\vert #1 \right\vert\kern-0.25ex\right\vert\kern-0.25ex\right\vert}}

\def\ud{\mathrm{d}}
\def\dt{\partial_t}
\def\p{\partial}
\def\ls{\lesssim}
\def\gs{\gtrsim}

\def\rt{\rightarrow}
\def\r{\mathbb{R}}
\def\no{\nonumber}
\def\ue{\mathrm{e}}

\def\ds{\displaystyle}

\def\s{\mathbb{S}}
\def\e{\varepsilon}
\def\o{\omega}
\def\th{\theta}
\def\rp{\r_+}
\def\t{\mathbb{T}}
\def\m{\mu}
\def\mm{\mathcal{M}}
\def\vh{\varrho}
\def\q{\mathscr{Q}}
\def\qq{\mathcal{Q}}
\def\mh{\mm^{\frac{1}{2}}}
\def\mhh{\mm^{-\frac{1}{2}}}
\def\l{L}
\def\g{\Gamma}
\def\d{\delta}
\def\k{\kappa}

\def\nx{\nabla_x}
\def\nv{\nabla_v}
\def\pk{\mathbf{P}}
\def\ik{\mathbf{I}}
\def\npk{\ik-\pk}
\def\nnpk{(\ik-\pk)}
\def\id{\mathbf{1}}

\def\ee{\mathcal{E}}

\def\fs{F^{\#}}

\def\qs{Q^{\#}}
\def\rs{R^{\#}}
\def\lls{\ell^{\#}}
\def\uus{u^{\#}}

\begin{document}

\title{On the Quantum Boltzmann Equation\\ near Maxwellian and Vacuum}
\date{}
\author[1]{Zhimeng Ouyang\thanks{zhimeng\_ouyang@brown.edu}}
\affil[1]{Department of Mathematics, Brown University}
\author[2]{Lei Wu\thanks{lew218@lehigh.edu}}
\affil[2]{Department of Mathematics, Lehigh University}

\maketitle

\vspace{-5pt}
\begin{abstract}
We consider the non-relativistic quantum Boltzmann equation for fermions and bosons. Using the nonlinear energy method and mild formulation, we justify the global well-posedness when the density function is near the global Maxwellian and vacuum. This work is a generalization and adaptation of the classical Boltzmann theory. Our main contribution is a detailed analysis of the nonlinear operator $Q$ in the quantum context. This is the first piece of a long-term project on the quantum kinetic equations. 

\smallskip
\textbf{Keywords: fermions; bosons; energy method; stability}
\end{abstract}

\setcounter{tocdepth}{2}
\tableofcontents

\newpage


\section{Introduction}

\subsection{Problem Setup}

We consider the quantum Boltzmann equation in three dimensions with hard-sphere collisions:
    \begin{align}\label{equation: boltzmann}
        \dt F+v\cdot\nx F=Q[F,F;F],
    \end{align}
where the collision operator
    \begin{align} \label{collision operator}
        Q[F,F;F]:=&\int_{\r^3}\int_{\s^2}q(\omega,\abs{v-u})\Big[F(u')F(v')\big(1+\th F(u)\big)\big(1+\th F(v)\big)\\
        &-F(u)F(v)\big(1+\th F(u')\big)\big(1+\th F(v')\big)\Big]\ud\o\ud u.\no
    \end{align}
for $q(\omega,\abs{v-u})=\big|\o\cdot(v-u)\big|$ ($\o\in\s^2$). 
Here the unknown $F(t,x,v)$ is the density function (for quantum particles),
at time $t\in\rp$, at the position $x\in\t^3\ \text{or}\ \r^3$, and having the velocity $v\in\r^3$. 

In the integral of \eqref{collision operator}, 
$(u,v)$ denotes the pre-collision velocity, and $(u',v')$ the post-collision velocity with $u'=u+\omega\big(\omega\cdot(v-u)\big)$, $v'=v-\omega\big(\omega\cdot(v-u)\big)$. They satisfy the conservation of momentum $u+v=u'+v'$ and energy $\abs{u}^2+\abs{v}^2=\abs{u'}^2+\abs{v'}^2$. $\th=\pm 1$ corresponds to the fermions ($-$) or bosons ($+$), respectively. The equation is equipped with initial data
\begin{align}
    F(0,x,v)=F_0(x,v).
\end{align}
The collision operator $Q$ is essentially cubic, since the cancellation reveals that
\begin{align}
    Q[F,F;F]=&\int_{\r^3}\int_{\s^2}q(\omega,\abs{v-u})\Big[F(u' )F(v')\big(1+\th F(u)+\th F(v)\big)\\
    &-F(u)F(v)\big(1+\th F(u')+\th F(v')\big)\Big]\ud\o\ud u.\no
\end{align}
For fermions ($\th=-1$), we require that $0\leq F_0\leq 1$. For bosons ($\th=1$), we require that $F_0\geq 0$. 

In this paper, we intend to study the global well-posedness and decay of the solution $F$ when it is close to the global Maxwellian or the vacuum.

\begin{remark}
When $\th=0$, the cubic terms in $Q$ vanish and the equation \eqref{equation: boltzmann} reduces to the classical Boltzmann equation.
\end{remark}


\smallskip
\subsection{Background and Modelling}

In this section, we briefly discuss various quantum kinetic equations. We refer to \cite{Spohn1994, Benedetto.Castella.Esposito.Pulvirenti2007} for more detailed discussion.

Quantum kinetic theory concerns the dynamics of a large number of quantum particles, including fermions and bosons. The equation arising from first principles to describe $N$ identical interacting particles is the $N$-body Schr\"odinger equation. Under proper scaling and in suitable regime, asymptotically the overall behavior of this particle system can be characterized by the quantum kinetic equations.
\ \\

\noindent\textbf{Model 1: Low density scaling:} $\text{time}\sim\e^{-1}$, $\text{volume}\sim \e^{-3}$ and $N\sim\e^{-2}$.
    \begin{itemize}
        \item 
        The particles are typically far apart and the difference between classical, Fermi, or Bose gas is irrelevant. All three statistics should be described by the same equation.
        \item
        The interacting potential $\phi(x)$ is short-ranged.
        \item
        The mean free path and mean free time is of $O(\e^{-1})$.
        \item
        In the classical regime, this model corresponds to dilute gas dynamics, which is described by the classical Boltzmann equation.
        \item
        In the quantum regime, this model is described by the quantum Boltzmann equation
        \begin{align}
            \dt F+v\cdot\nx F=Q_L[F,F],
        \end{align}
        where
        \begin{align}
            Q_L[F,F]=\int_{\r^3}\int_{\s^2}B_L(\o,v-u)\big[F(u')F(v')-F(u)F(v)\big]\ud\o\ud u.
        \end{align}
        Here the collision kernel
        \begin{align}
            B_L(\o,v-u)\sim q(\omega,\abs{v-u})\abs{\hat\phi\Big(\o\big(\o\cdot(v-u)\big)\Big)}^2+\sum_{n\geq3}B_L^{(n)}(\o,v-u),
        \end{align}
        for $\hat\phi(k)$ as the Fourier transform of $\phi(x)$, and higher-order Born approximation $B_L^{(n)}$ (see \cite{ Benedetto.Castella.Esposito.Pulvirenti2007}). 
        \item
        The quantum Boltzmann equation has the same structure as the classical Boltzmann equation, but the collision kernel may be different.
    \end{itemize}

\noindent\textbf{Model 2: Weak coupling scaling:} $\text{time}\sim\e^{-1}$, $\text{volume}\sim \e^{-3}$ and $N\sim\e^{-3}$.
    \begin{itemize}
        \item
        The gas is dense, but the coupling of neighboring particles is weak.
        \item
        There are two possible scaling for the 
        potential $\phi(x)$. \begin{itemize}
            \item 
            Scaling $\e^{\frac{1}{2}}\phi(x)$.
            \item
            Scaling $\phi(\e^{-\frac{1}{2}}x)$.
        \end{itemize}
        \item
        For classical particles, the first scaling gives rise to the classical Landau equation and the second scaling is equivalent to the low-density limit (which means that it will lead to the classical Boltzmann equation). 
        \item
        However, for quantum particles, both scalings are genuine weak coupling. This is described by the quantum Boltzmann equation (or the so-called Uehling-Uhlenbeck equation \cite{Uehling.Uhlenbeck1933})
        \begin{align}
            \dt F+v\cdot\nx F=Q_W[F,F;F],
        \end{align}
        where
        \begin{align}
            Q_W[F,F;F]=&\int_{\r^3}\int_{\s^2}B_W(\o,v-u)\Big[F(u')F(v')\big(1+\th F(u)\big)\big(1+\th F(v)\big)\\
            &-F(u)F(v)\big(1+\th F(u')\big)\big(1+\th F(v')\big)\Big]\ud\o\ud u.\no
        \end{align}
        Here $\th=1$ corresponds to Bose-Einstein statistics (for bosons), $\th=-1$ corresponds to Fermi-Dirac statistics (for fermions), and $\th=0$ corresponds to Maxwell-Boltzmann statistics (for classical particles). The collision kernel
        \begin{align}
            B_W(\o,v-u)\sim \big|\o\cdot(v-u)\big|\abs{\hat\phi(v'-v)+\th\hat\phi(v'-u)}^2.
        \end{align}
    \end{itemize}

\noindent\textbf{Model 3: Mean field scaling:} $\text{time}\sim\e^{-1}$, $\text{volume}\sim \e^{-3}$ and $N\sim\e^{-3}$.
    \begin{itemize}
        \item
        Particles are affected not only by its neighboring particles, but all the others.
        \item
        The potential $\phi(x)$ can be short- or long-ranged. The scaling is $\e^3\phi(\e x)$.
        \item
        In classical mechanics, this corresponds to the Vlasov equation, which is applicable to the long-ranged potential. 
        \item
        In quantum mechanics, this is described by the quantum Vlasov equation
        \begin{align}
            \dt F+v\cdot\nx F+\int_{\r^3}\int_{\r^3}F(y,u)\big[\nx\phi(y-x)\cdot \nv F(x,v)\big]\ud u\ud y=0.
        \end{align}
    \end{itemize}
    
\noindent In this paper, we focus on the weak coupling model with hard-sphere potential $\phi(x)=\delta_0(x)$, which yields $\hat\phi(k)=1$. Hence, the collision kernel reduces to $q(\omega,\abs{v-u})$, which is exactly the same as the classical one. 

\begin{remark}
Some comments regarding the models above:
\begin{enumerate}
    \item 
    We will discuss the more general models (e.g. inverse power laws) in the subsequent work. Note that a direct computation reveals that for $\phi(x)\sim \abs{x}^{-p}$ with $p\geq1$, unlike the hard-sphere case, the quantum collision kernel is not the same as the classical hard/soft potential one. Hence, some of the work in the related literature using classical hard/soft potential may need reexamination.
    \item
    The low density model with hard-sphere potential will roughly reduce to the classical Boltzmann equation (if we ignore the higher-order Born approximation terms), which is less appealing in the quantum context. Note that for general potential $\phi$, the low density model is also of interest. There are much less work in this direction.
    \item
    For classical particles, the Boltzmann equation describes dilute gas and the Landau equation describes dense gas. However, for quantum particles, both cases give rise to the quantum Boltzmann equations. In other words, in the quantum regime, Boltzmann equation can be applied to more general scenarios.
    \item
    Similar to the classical Landau equation, quantum Landau equation can be derived by taking grazing collision limit in quantum Boltzmann equation. The typical quantum Landau equation is 
    \begin{align}
        \dt F+v\cdot\nx F=Q_{\text{Lan}}[F,F;F],
    \end{align}
where the collision operator
    \begin{align} 
        Q_{\text{Lan}}[F,F;F]:=&\nv\cdot\int_{\r^3}\Phi(v-u)\Big[F(u)\big(1+\th F(u)\big)\nv F(v)\\
        &-F(v)\big(1+\th F(v)\big)\nabla_uF(u)\Big]\ud u,\no
    \end{align}
    for a semi-positive definite matrix
    \begin{equation}
        \Phi(v)=\abs{v}^{\gamma+2}\left(I-\frac{v\otimes v}{\abs{v}^2}\right)
        \;\text{ with } -3\leq\gamma\leq 1.
    \end{equation}
\end{enumerate}
\end{remark}





\smallskip
\subsection{Motivation and Previous Results}

The study of the quantum Boltzmann equation dates back to Uehling-Uhlenbeck \cite{Uehling.Uhlenbeck1933}. Since then, there are many papers devoted to its theory and applications in physics, chemistry and engineering. Here, we briefly summarize the relevant literature regarding its mathematical theory.

So far, the study mainly focuses on two types of problems: the derivation of quantum Boltzmann equations and homogeneous equations.

The rigorous derivation of classical/quantum equation is an outstanding problem in kinetic theory (see Vallani \cite{Villani2002}). The formal derivation of quantum Boltzmann equation from $N$-body Schr\"odinger equation can be found in Erd\H{o}s-Salmhofer-Yau \cite{Erdos.Salmhofer.Yau2004} and Spohn \cite{Spohn1994}. In a series of papers \cite{Benedetto.Castella.Esposito.Pulvirenti2004, Benedetto.Castella.Esposito.Pulvirenti2005, Benedetto.Castella.Esposito.Pulvirenti2006, Benedetto.Castella.Esposito.Pulvirenti2008}, Benedetto-Castella-Esposito-Pulvirenti partially derived the quantum Boltzmann equation both in the low-density limit and weak-coupling limit. The brief surveys of their results can be found in \cite{ Benedetto.Castella.Esposito.Pulvirenti2007} and \cite{Pulvirenti2006}, and the references there are also very informative. Some recent progress can be found in Colangeli-Pezzotti-Pulvirenti \cite{Colangeli.Pezzotti.Pulvirenti2015} and Chen-Guo \cite{Chen.Guo2015}. In summary, this problem is still largely open so far. 

Just like in the classical Boltzmann theory, the homogeneous quantum equation 
\begin{align}
        \dt F=Q[F,F;F],
    \end{align}
is a good starting point to develop the whole theory. There are quite a few results in this direction. We refer to Lu \cite{Lu2001}, Lu-Wennberg \cite{Lu.Wennberg2003} for fermions, and Lu \cite{Lu2000, Lu2004, Lu2005}, Lu-Zhang \cite{Lu.Zhang2011}, Briant-Einav \cite{Briant.Einav2016} for bosons. Their results basically follow from the moment-entropy approach and are based on $L^1$ theory. The theory has been developed to treat both the isotropic and anisotropic cases.

One of the most interesting facts about the quantum Boltzmann equation is for bosons. Physicists have long predicted the existence of the so-called Bose-Einstein condensation and it was observed in 1995. Since then, the mathematical justification of such phenomenon has attracted a lot of attention. Mathematically, it means that for bosons, the solution to the quantum Boltzmann equation may have singularity (e.g. $\delta$-function) at certain spatial point. The existence of equilibria with $\delta$-function has been justified by minimizing the entropy functional by Escobedo-Mischler-Valle \cite{Escobedo.Mischler.Valle2005}.

On one hand, it will be thrilling to justify the generation of such singular solutions when the temperature is sufficiently low (physical requirement). Escobedo-Vel\'azquez \cite{Escobedo.Velazquez2015} justifies that for some well-prepared initial data (e.g. smooth, radial, and sufficiently localized), the solution will become unbounded (i.e. blow up) within finite time. The work is done for the homogeneous equation with isotropic data. Such localization restriction for the initial data has been removed in Cai-Lu \cite{Cai.Lu2019}. We refer to Escobedo-Mischler-Vel\'azquez \cite{Escobedo.Mischler.Velazquez2007, Escobedo.Mischler.Velazquez2008}, Escobedo-Vel\'azquez \cite{Escobedo.Velazquez2008}, Spohn \cite{Spohn2010}, Lu \cite{Lu2013, Lu2014}, Bandyopadhyay-Vel\'azquez \cite{Bandyopadhyay.Velazquez2015} for more information. 

On the other hand, studying the general well-posedness and regularity becomes much harder. We have to develop measure solutions for the Boltzmann equation. There are a lot of progress in this direction for the homogeneous equation (see Lu \cite{Lu2016, Lu2018}, Li-Lu \cite{Li.Lu2019}). In particular, the stability of Bose-Einstein distribution has also been studied in the sense of measure.

There are also some results regarding the dynamics of the excited states and interactions with the condensed states, see Arkeryd-Nouri \cite{Arkeryd.Nouri2012, Arkeryd.Nouri2013, Arkeryd.Nouri2015} and Nguyen-Tran \cite{Nguyen.Tran2019}.

Compared with the homogeneous equation, there are much fewer results on the non-homogeneous equation, see Dolbeault \cite{Dolbeault1994}, Zhang-Lu \cite{Zhang.Lu2002,Zhang.Lu2002(=)}, Lu \cite{Lu2006, Lu2008} and Arkeryd-Nouri \cite{Arkeryd.Nouri2017}. In particular, all of the known results so far regarding the Bose-Einstein condensation concerns the homogeneous equation. Also, we refer to the nice introduction to the quantum Landau equation by Lemou \cite{Lemou2000}. The semi-classical limit from quantum Boltzmann equation to quantum Landau equation can be found in He-Lu-Pulvirenti \cite{He.Lu.Pulvirenti2020}.

As far as we are aware of, so far there are very limited study on the non-homogeneous quantum Boltzmann equation, and the Bose-Einstein condensation. In this paper, we plan to utilize the well-known nonlinear energy method (see Guo \cite{Guo2002, Guo2002(=), Guo2010, Guo2003, Guo2003(=), Guo2004, Guo2012}, Strain-Guo \cite{Strain.Guo2004, Strain.Guo2006, Strain.Guo2008}, Kim-Guo-Hwang \cite{Kim.Guo.Hwang2020}) and mild formulation (see Guo \cite{Guo2001}, Illner-Shinbrot \cite{Illner.Shinbrot1984}, Duan-Yang-Zhu \cite{Duan.Yang.Zhu2005} and the references in \cite{Glassey1996}) to investigate the global well-posedness of solutions near the Maxwellian and the vacuum.

When we are preparing this work, we are aware of the recent preprint by Bae-Jang-Yun \cite{Bae.Jang.Yun2021} on the relativistic quantum Boltzmann equation. They used the nonlinear energy method similar to ours. However, there are some key differences:
\begin{itemize}
    \item 
    They only consider the system in the periodic setting near the Maxwellian. On the other hand, we consider both the periodic and whole space settings near the Maxwellian and the whole space setting near the vacuum. This provides a more comprehensive picture of the solutions in these classical frameworks.
    \item
    More importantly, the nonlinear estimates are different. While in \cite{Bae.Jang.Yun2021} they can bound the nonlinear term in the relativistic case without using $L^{\infty}$ bounds in velocity, such an estimate is absent in the non-relativistic scenario. Therefore, we must take velocity derivatives and use Sobolev embedding, which introduces a lot of technical difficulties.
\end{itemize}


\smallskip
\subsection{Main Results}

\subsubsection{Near the Maxwellian}

For the case $\Omega=\t^3$, let the multi-indices $\gamma$ and $\beta$ be
\begin{align}
    \gamma=(\gamma_1,\gamma_2,\gamma_3),\quad \beta=(\beta_1,\beta_2,\beta_3).
\end{align}
Denote the differential operator by
\begin{align}
    \p^{\gamma}_{\beta}=\p_{x_1}^{\gamma_1}\p_{x_2}^{\gamma_2}\p_{x_3}^{\gamma_3}\p_{v_1}^{\beta_1}\p_{v_2}^{\beta_2}\p_{v_3}^{\beta_3}.
\end{align}
If each component of $\th$ is not greater than $\bar\th$, we denote by $\th\leq \bar\th$.

Let $N\geq 8$.
Denote
\begin{align}
    \vertiii{f(t)}&=\sum_{\abs{\gamma}+\abs{\beta}\leq N}\tnm{\p_{\beta}^{\gamma}f(t)}{x,v},\\
    \vertiii{f(t)}_{\nu}&=\sum_{\abs{\gamma}+\abs{\beta}\leq N}\vnm{\p_{\beta}^{\gamma}f(t)}{x,v},
\end{align}
where $\vnm{f}{x,v}:\sim \|f\|_{2,\frac{1}{2}}$ is defined in Definition \ref{nu-norm}.
Denote also
\begin{align}
    \ee[f(t)]=\vertiii{f(t)}^2+\int_0^t\vertiii{f(s)}_{\nu}^2 \,\ud s,
\end{align}
and
\begin{align}
    \ee[f_0]=\vertiii{f_0}^2.
\end{align}

\begin{theorem}[Periodic Case]
Let $\Omega=\t^3$. 
Assume that the initial data $F_0(x,v)=\m(v)+\mh(v)f_0(x,v) \geq 0$ with $\mu(v)$ and $\mm(v)$ defined in \eqref{Maxwellian} and \eqref{Maxwellian-M}, and $f_0(x,v)$ satisfying
\begin{align}
    \ee[f_0]\leq \frac{M_0}{2}
\end{align}
for some $M_0>0$, as well as the conservation laws \eqref{CL-Mass}\,--\,\eqref{CL-Energy}.
Then there exists a unique global solution $F(t,x,v)=\m(v)+\mh(v)f(t,x,v) \geq 0$ to the quantum Boltzmann equation \eqref{equation: boltzmann} such that
\begin{align}
    \ee[f(t)]\leq M_0
\end{align}
for any $t\in[0,\infty)$. Also, the perturbation $f(t,x,v)$ satisfies
\begin{align}
    \nnm{f(t)}\leq C\ue^{-kt}\nnm{f_0}
\end{align}
for some constant $C,K>0$. 
\end{theorem}

For the case $\Omega=\r^3$, let the multi-indices $\gamma$ and $\beta$ be
\begin{align}
    \gamma=(\gamma_0,\gamma_1,\gamma_2,\gamma_3),\quad \beta=(\beta_1,\beta_2,\beta_3).
\end{align}
Denote the differential operator by
\begin{align}
    \p^{\gamma}_{\beta}=\dt^{\gamma_0}\p_{x_1}^{\gamma_1}\p_{x_2}^{\gamma_2}\p_{x_3}^{\gamma_3}\p_{v_1}^{\beta_1}\p_{v_2}^{\beta_2}\p_{v_3}^{\beta_3}.
\end{align}
If each component of $\th$ is not greater than $\bar\th$, we denote by $\th\leq \bar\th$.

Let $N\geq 8$.
Denote
\begin{align}
    \vertiii{f(t)}&=\sum_{\abs{\gamma}+\abs{\beta}\leq N}\tnm{\p_{\beta}^{\gamma}f(t)}{x,v},\\
    \vertiii{f(t)}_{\nu}&=\sum_{\abs{\gamma}+\abs{\beta}\leq N}\vnm{\p_{\beta}^{\gamma}f(t)}{x,v}.
\end{align}
Denote also
\begin{align}
    \ee[f(t)]=\vertiii{f(t)}^2+\int_0^t\vertiii{f(s)}_{\nu}^2 \,\ud s,
\end{align}
and
\begin{align}
    \ee[f_0]=\vertiii{f_0}^2.
\end{align}

\begin{theorem}[Whole Space Case]
Let $\Omega=\r^3$. 
Assume that the initial data $F_0(x,v)=\m(v)+\mh(v)f_0(x,v) \geq 0$ with $\mu(v)$ and $\mm(v)$ defined in \eqref{Maxwellian} and \eqref{Maxwellian-M}, and $f_0(x,v)$ satisfying
\begin{align}
    \ee[f_0]\leq \frac{M_0}{2}
\end{align}
for some $M_0>0$, as well as the conservation laws \eqref{CL-Mass}\,--\,\eqref{CL-Energy}.
Then there exists a unique global solution $F(t,x,v)=\m(v)+\mh(v)f(t,x,v) \geq 0$ to the quantum Boltzmann equation \eqref{equation: boltzmann} such that
\begin{align}
    \ee[f(t)]\leq M_0
\end{align}
for any $t\in[0,\infty)$. 
\end{theorem}

\begin{remark}
In the whole space case, since the dissipation lacks the lowest-order term, we cannot easily obtain decay estimates. This loss of lowest-order terms is purely due to the absence of the Poincar\'e-type inequality. The optimal $t^{-\frac{3}{4}}$ decay may be obtained under other norms based on a different framework (see Glassey \cite{Glassey1996}, Duan-Strain \cite{Duan.Strain2011, Duan.Strain2011(=)}). However, it is beyond the scope of our paper.
\end{remark}

\begin{remark}\label{main remark 1}
For both $\Omega=\t^3$ and $\Omega=\r^3$, the global solution satisfies the following positivity bounds:
\begin{itemize}
    \item 
    for fermions $\th=-1$, if $0\leq F_0(x,v)\leq 1$, then $0\leq F(t,x,v)\leq 1$;
    \item
    for bosons $\th=1$, if $F_0(x,v)\geq 0$, then $F(t,x,v)\geq 0$.
\end{itemize}
See Theorem \ref{local-theorem}. This is a byproduct of the standard iteration argument.
\end{remark}

\subsubsection{Near the Vacuum}

Given $\beta>0$, define 
\begin{align}
    S=\Big\{F\in C^0(\rp\times\r^3\times\r^3): \abs{F(t,x,v)}\leq c\ue^{-\beta(\abs{x}^2+\abs{v}^2)} \text{ for some } c>0 \Big\},
\end{align} 
equipped with the norm
\begin{align}
    \nnm{F}:=\sup_{t,x,v}\Big(\ue^{\beta(\abs{x}^2+\abs{v}^2)}\abs{F(t,x,v)}\Big).
\end{align}
Define also the solution set
\begin{align}
    S_R=\big\{F\in S: \nnm{F}\leq R\big\}.
\end{align}

\begin{theorem}[Whole Space Case]
There exists an $R_0>0$ sufficiently small such that if $\nnm{F_0}\leq \frac{R_0}{2}$, then the equation \eqref{equation: boltzmann} has a unique mild solution $F\in S_{R_0}$.
\end{theorem}

\begin{remark}
In the near vacuum case, the global solution also satisfies the positivity bounds as in Remark \ref{main remark 1}. However, the proof is highly non-trivial. See Theorem \ref{vacuum-theorem: positivity} for bosons and Theorem \ref{vacuum-theorem: positivity.} for fermions.
\end{remark}

\begin{remark}
We will skip the periodic case near the vacuum due to the following:
\begin{itemize}
    \item 
    The framework developed here highly relies on the dispersion properties of the transport operator $\dt +v\cdot\nx$. However, such dispersive decay in time is absent in the periodic case for mild solution (smoothness may improve the decay). It is far beyond our methods discussed here.
    \item
    In the periodic case for the classical Boltzmann equation, as Mouhot \cite{Mouhot2005} and Briant \cite{Briant2015} proved, the solution will instantaneously fill the vacuum and be above a Maxwellian. This indicates that the near Maxwellian framework is more suitable for this case. We anticipate that the similar result will hold for the quantum Boltzmann equation.
    \item
    The general theory for solutions merely satisfying $\abs{F(t,x,v)}\ls \ue^{-\beta\abs{v}^2}$ is far from mature and there are very limited results on the global well-posedness.
\end{itemize}
\end{remark}

\begin{remark}
All of these results reveal that if the initial data is sufficiently close to the Maxwellian or vacuum, then the solution will remain finite and not blow up. Hence, for bosons, the Bose-Einstein condensation will not occur. In other words, the occurrence of Bose-Einstein condensation requires more delicate analysis. This is a sharp constrast with the result in Escobedo-Vel\'azquez \cite{Escobedo.Velazquez2015}.
\end{remark}

\subsubsection{Notations}

Throughout this paper, $C>0$ denotes a constant that only depends on
the domain $\Omega$, but does not depend on the data. It is
referred as universal and can change from one inequality to another.
When we write $C(z)$, it means a certain positive constant depending
on the quantity $z$. We write $a\ls b$ to denote $a\leq Cb$ for some universal constant $C>0$; we will use $\gtrsim$ and $\simeq$ in a similar standard way.


\smallskip
\subsection{Difficulties and Ideas}

The framework of studying the non-quantum Boltzmann equation near the Maxwellian is classical (see \cite{Glassey1996, Guo2002, Guo2002(=), Guo2004}). It is customary to linearize the solution $F(t,x,v)$ around the global Maxwellian $\bar\mu(v)=\ue^{-\abs{v}^2}$ as $F=\bar\mu+\bar\mu^{\frac{1}{2}}f$, where the perturbation $f(t,x,v)$ satisfies 
\begin{align}
    \dt f+v\cdot\nx f+\l[f]=\g[f,f],
\end{align}
for the linearized Boltzmann operator $\l$ and quadratic operator $\g$. 

For the quantum Boltzmann equation, the first difficulty is that the naive perturbation $F=\mu+\mu^{\frac{1}{2}}f$ around the quantum Maxwellian 
\begin{align}
\m(v)=\dfrac{1}{\rho\ue^{\abs{v}^2}-\th}
\end{align}
with $\rho>1$ will result in a linearized Boltzmann operator without coercivity, which is devastrating for the energy method. To get around this, a crucial observation is that we need to redesign the perturbation as $F=\mu+\mh f$ with 
\begin{align}
    \mathcal{M}(v)=\m(1+\th \m)=\frac{\rho\ue^{\abs{v}^2}}{(\rho\ue^{\abs{v}^2}-\th)^2}.
\end{align}
Then as Theorem \ref{prelim-theorem: L-properties} reveals, we have that $\l$ is a self-adjoint operator on $L^2_v(\r^3)$ and
\begin{align}
    \int_{\r^3}f\cdot\l[f]\,\ud v\gs \vnm{\nnpk[f]}{v},
\end{align}
where $\pk[f]$ is the orthogonal projection of $f$ onto the five-dimensional null space of $\l$, i.e.
\begin{align}
    \text{Null space of }\l=\Big\{\mh(a+b\cdot v+c\abs{v}^2): a,c\in\r,b\in\r^3\Big\}.
\end{align}
Then this recovers the basic energy structure as in the classical Boltzmann equation.

The most important distinction between the classical and quantum Boltzmann equation framework lies in the nonlinear estimates. Compared with \cite[Lemma 2.3]{Guo2002} using only $L^2_v$ norms, Lemma \ref{prelim-lemma: nonlinear-no-derivative} depends on $\sup_vf$. Thus in order to obtain $L^{\infty}_v$ estimate of $f$ in the Sobolev-space framework, we resort to the high-regularity framework with derivatives in both space and velocity. 

Intuitively, for the classical Boltzmann equation, the nonlinear operator $\g$ only involves integral for $\ud u\ud v$ of the form (we ignore the collision kernel)
\begin{align}
    \int f(u)f(v)\ \text{or}\ \int f(u')f(v').
\end{align}
These can be handled with the pre-post change-of-variable $(u,v)\leftrightarrow(u',v')$. However, for the quantum Boltzmann equation, due to the cubic nonlinearity, we must estimate the mixed-type integral
\begin{align}
    \int f(u)f(v')\ \text{or}\ \int f(u)f(u')\ \text{or}\ \int f(v)f(u')\ \text{or}\ \int f(v)f(v').
\end{align}
Then the pre-post change-of-variable cannot resolve the difficulty. This kind of integrals have long been captured as an important ingredient to tackle long-range interactions and non-cutoff Boltzmann equations. So far, there are mainly two approaches to deal with them:
\begin{itemize}
    \item 
    The proof of Alexandre-Desvillettes-Villani-Wennberg \cite[Lemma 1]{Alexandre.Desvillettes.Villani.Wennberg2000} justifies that under proper sense, $\frac{\ud u'}{\ud u}$ and $\frac{\ud v'}{\ud v}$ have positive lower bounds. With this in hand, we may bound $\int f(u)f(v')$ and $\int f(v)f(u')$. However, the difficulty lies in the remaining two integrals $\int f(u)f(u')$ and $\int f(v)f(v')$. It has been shown that $\frac{\ud u'}{\ud v}$ and $\frac{\ud v'}{\ud u}$ do not have positive lower bounds. In the literature, some authors claimed that it is possible to use the well-known cancellation lemma (see Alexandre-Desvillettes-Villani-Wennberg \cite[Lemma 1]{Alexandre.Desvillettes.Villani.Wennberg2000}) to finish the job, but we cannot see the viability and it does not look hopeful. Actually, cancellation lemma is used to bound $\int f(u)(f(v)-f(v'))$ and $\int f(v)(f(u)-f(u'))$, which can help handle $\int f(u)f(v')$ and $\int f(v)f(u')$, but not $\int f(u)f(u')$ and $\int f(v)f(v')$, so the difficulty remains.
    \item
    In the analysis of non-cutoff Boltzmann equation, Gressman-Strain \cite{Gressman.Strain2011} introduces two Carleman-type representations. Roughly speaking, \cite[Proposition A.1]{Gressman.Strain2011} makes the change-of-variables $v\rt v'$ and $u\rt u'$ with the help of $\sigma$ integral, which is controllable (see \cite[Section 3]{Gressman.Strain2011}). This results in the same good bounds as the above Alexandre-Desvillettes-Villani-Wennberg argument. On the other hand, \cite[Proposition A.2]{Gressman.Strain2011} forcefully makes the change-of-variables $u\rt v'$ and $v\rt u'$ with the help of $\sigma$ integral. However, as \cite[Section 7]{Gressman.Strain2011} reveals, this representation will generate singularity $\sim{\abs{v-v'}}^{-1}$ which cannot fit into our needs. The similar result was also proved by Silvestre \cite[Lemma A.1]{Silvestre2016}. All in all, we still cannot handle the remaining two integrals $\int f(u)f(u')$ and $\int f(v)f(v')$.
\end{itemize}
Note that \cite[Lemma 3.4, Lemma 3.7]{Bae.Jang.Yun2021} introduced another method, i.e. rewriting $\sigma\in\s^2$ integral with $\delta$ function due to the conservation of momentum and energy. While it works for the relativistic models there, we may easily check that in the non-relativistic case, the resulting formula is equivalent to the Carleman-type representation in \cite[Proposition A.2]{Gressman.Strain2011}. Hence, it does not work as expected.

At the end of the day, we arrive at the conclusion that we have to resort to $L^{\infty}_v$ estimate of $f$, which depends on high-regularity framework in the velocity variable and Sobolev embedding. This in turn creates a lot of technical difficulties. For example, we need to estimate the velocity derivatives of $\l$ and $\g$ operators. Such estimates have been done for classical hard-sphere case in \cite[Lemma 2.1, Lemma 2.2]{Guo2002}. However, the proof there highly depends on the explicit formula of $\l$ (see \cite[Section 3.2, Section 3.3]{Glassey1996}), which we do not have. Then we have to use the fact that $\m(v)\simeq\bar\mu(v)$ and $\mathcal{M}(v)\simeq\bar\mu(v)$ and the conservation laws to bound the derivatives term by term. In particular, we need to find the partially explicit formula as in \cite[Section 3.2, Section 3.3]{Glassey1996} to complete the estimates.

For the quantum Boltzmann equation near the vacuum, we utilize the robust framework introduced in Illner-Shinbrot \cite{Illner.Shinbrot1984} (we refer to Glassey \cite[Section 2]{Glassey1996} for clarity). While the global well-posedness is not too difficult to adapt, the positivity proof needs more thinking. The proof for the classical Boltzmann equation relies on a monotonicity argument to construct two approximate sequences from above and below. In particular, it highly depends on the monotonicity of the gain term $\int F(u')F(v')$. While this still holds for bosons with $\th=1$, such a naive adaptation does not work for fermions $\th=-1$ since $\int F(u')F(v')\big(1+\th F(u)+\th F(v)\big)$ does not have the desired monotonicity.
Our strategy is to redesign the artificial gain and loss terms to enforce the monotonicity. In particular, we also need to redesign the approximate sequences such that the convergence is preserved.


\smallskip
\subsection{Organization of the Paper}


Our paper is organized as follows: 
In Section \ref{Sec:Collision-Opt}, we record some preliminary results on the quantum Boltzmann collision operator $Q$, including the conservation laws, the entropy and $H$-Theorem, and the quantum Maxwellian. 
In Section \ref{Sec:Maxwellian-GlbSol}, we study the near Maxwellian case through the nonlinear energy method, which comes with a careful derivation of the perturbation form, linear and nonlinear estimates, and the proof of local and global well-posedness. 
In Section \ref{Sec:Vacuum-GlbSol}, we study the near vacuum case via the mild formulation, proving the global well-posedness and positivity of solutions.


\bigskip
\section{Properties of the Collision Operator} \label{Sec:Collision-Opt}

In this section, we present some basic results regarding the collision operator $Q$. They are mostly well-known for the classical Boltzmann equation and we derive them in the quantum context. We mainly adapt from Glassey \cite{Glassey1996}.

\subsection{Conservation Laws}

\begin{lemma}\label{basic-lemma: invariants}
For all smooth functions $F(v)$ and $\phi(v)$, small at infinity,
\begin{align}
    &\int_{\r^3}Q[F,F;F](v)\phi(v)\,\ud v\\
    =&\int_{\r^3}\int_{\r^3}\int_{\s^2}q(\o,\abs{v-u})\Big[F(u')F(v')\big(1+\th F(u)\big)\big(1+\th F(v)\big)\no\\
    &-F(u)F(v)\big(1+\th F(u')\big)\big(1+\th F(v')\big)\Big]\phi(v)\,\ud\o\ud u\ud v\no\\
    =&\int_{\r^3}\int_{\r^3}\int_{\s^2}q(\o,\abs{v-u})\Big[F(u')F(v')\big(1+\th F(u)\big)\big(1+\th F(v)\big)\no\\
    &-F(u)F(v)\big(1+\th F(u')\big)\big(1+\th F(v')\big)\Big]\phi(u)\,\ud\o\ud u\ud v\no\\
    =&-\int_{\r^3}\int_{\r^3}\int_{\s^2}q(\o,\abs{v-u})\Big[F(u')F(v')\big(1+\th F(u)\big)\big(1+\th F(v)\big)\no\\
    &-F(u)F(v)\big(1+\th F(u')\big)\big(1+\th F(v')\big)\Big]\phi(v')\,\ud\o\ud u\ud v\no\\
    =&-\int_{\r^3}\int_{\r^3}\int_{\s^2}q(\o,\abs{v-u})\Big[F(u')F(v')\big(1+\th F(u)\big)\big(1+\th F(v)\big)\no\\
    &-F(u)F(v)\big(1+\th F(u')\big)\big(1+\th F(v')\big)\Big]\phi(u')\,\ud\o\ud u\ud v.\no
\end{align}
\end{lemma}

\begin{proof}
The first equality is the definition of $Q$. Then switching the status of $u$ and $v$, we find that the integral does not change, so the second equality naturally follows.

The third equality comes from the pre-post collision substitution $(u,v)\mapsto (u',v')$. Based on \cite[Lemma 1.4.1, Lemma 1.6.1]{Glassey1996}, we may directly verify that $\o\cdot(v-u)=-\o\cdot(v'-u')$, and the Jacobian $J$ satisfies $\abs{J}=1$. Also, under this substitution, we know $(u',v')\mapsto (u,v)$. Hence, after renaming the variables, we get the third equality. Similarly, switching the status of $u$ and $v$, we get the fourth equality. 
\end{proof}

\begin{corollary}\label{basic-corollary: invariants 1}
For all smooth functions $F(v)$ and $\phi(v)$, small at infinity,
\begin{align}
    &\int_{\r^3}Q[F,F;F](v)\phi(v)\,\ud v\\
    =&\;\frac{1}{4}\int_{\r^3}\int_{\r^3}\int_{\s^2}q(\o,\abs{v-u})\Big[F(u')F(v')\big(1+\th F(u)\big)\big(1+\th F(v)\big)\no\\
    &\,-F(u)F(v)\big(1+\th F(u')\big)\big(1+\th F(v')\big)\Big]\Big\{\phi(v)+\phi(u)-\phi(v')-\phi(u')\Big\}\,\ud\o\ud u\ud v.\no
\end{align}

\begin{proof}
    Adding the four equalities in Lemma \ref{basic-lemma: invariants}, this result naturally follows.
\end{proof}
    
\end{corollary}

\begin{corollary}\label{basic-corollary: invariants 2}
    For any smooth function $F(v)$, small at infinity,
    \begin{align}
        \int_{\r^3}Q[F,F;F](v)\,\ud v=0,\quad \int_{\r^3}Q[F,F;F](v)v_i\,\ud v=0,\quad \int_{\r^3}Q[F,F;F](v)\abs{v}^2\ud v=0,
    \end{align}
    for $i=1,2,3$.
\end{corollary}

\begin{proof}
Taking $\phi(v)=1,v_i,\abs{v}^2$ in Corollary\;\ref{basic-corollary: invariants 1}, which satisfy $\phi(v)+\phi(u)-\phi(v')-\phi(u')=0$, we immediately obtain the results.
\end{proof}

\begin{remark}
The test function $\phi(v)$ satisfying \,$\phi(v)+\phi(u)-\phi(v')-\phi(u')=0$\, is called a collisional invariant. Corollary \ref{basic-corollary: invariants 2} justifies that \,$1,v_i,\abs{v}^2$ are collisional invariants.
\end{remark}

\begin{theorem}\label{basic-theorem: conservation law}
If $F(t,x,v)$ is a solution to the equation \eqref{equation: boltzmann} which is suitable small at infinity, then we have for any $t>0$ and for $\Omega=\t^3\ \text{or}\ \r^3$,
\begin{align}
    \int_{\Omega}\int_{\r^3}F(t,x,v)\,\ud v\ud x&=\int_{\Omega}\int_{\r^3}F_0(x,v)\,\ud v\ud x,\quad (\text{Conservation of Mass})\\
    \int_{\Omega}\int_{\r^3}F(t,x,v)v_i\,\ud v\ud x&=\int_{\Omega}\int_{\r^3}F_0(x,v)v_i\,\ud v\ud x,\quad (\text{Conservation of Momentum})\\
    \int_{\Omega}\int_{\r^3}F(t,x,v)\abs{v}^2\ud v\ud x&=\int_{\Omega}\int_{\r^3}F_0(x,v)\abs{v}^2\ud v\ud x.\quad (\text{Conservation of Energy})
\end{align}
\end{theorem}

\begin{proof}
We multiply $1,v_i,\abs{v}^2$ on both sides of the equation \eqref{equation: boltzmann}, and then integrate the resulting equation over $(x,v)\in\Omega\times\r^3$. 
Using Corollary \ref{basic-corollary: invariants 2} and integration by parts, 
we get 
$$\frac{\ud}{\ud t}\iint_{\Omega\times\r^3}\! F(t,x,v) \phi(v)\,\ud v\ud x \,\equiv\, 0$$
for $\phi(v)=1,v_i,\abs{v}^2$,
which implies the desired result.
\end{proof}

\begin{remark}
The above is the basic conservation laws for the quantum Boltzmann equation.
\end{remark}


\smallskip
\subsection{Entropy and the $H$-Theorem}

\begin{theorem}[Entropy: $H$-Theorem]\label{basic-theorem: h-theorem}
If $F(t,x,v)>0$ is a solution to the equation \eqref{equation: boltzmann} which is suitable small at infinity (for $\th=-1$, we further require $F(t,x,v)<1$ to ensure $1+\th F(t,x,v)>0$), then we have for any $t>0$ and for $\Omega=\t^3\ \text{or}\ \r^3$,
\begin{align}
    \frac{\ud}{\ud t}\int_{\Omega}\int_{\r^3}\bigg[F(t,x,v)\ln\left(\frac{1}{F(t,x,v)}\right)-\th\big(1+\th F(t,x,v)\big)\ln\left(\frac{1}{1+\th F(t,x,v)}\right)\bigg]\ud v\ud x\geq 0.
\end{align}
\end{theorem}

\begin{proof}
Denote
\begin{align}
    S[F](t,x,v):=F(t,x,v)\ln\left(\frac{1}{F(t,x,v)}\right)-\th\big(1+\th F(t,x,v)\big)\ln\left(\frac{1}{1+\th F(t,x,v)}\right).
\end{align}
We may directly compute
\begin{align}
    \frac{\p S[F]}{\p t}=\dt F\ln\left(\frac{F}{1+\theta F}\right).
\end{align}
Hence, multiplying $\ln\left(\dfrac{F}{1+\theta F}\right)$ on both sides of the equation \eqref{equation: boltzmann} and integrating over $(x,v)\in\Omega\times\r^3$, we obtain
\begin{align}
    \frac{\ud}{\ud t}\int_{\Omega}\int_{\r^3}S[F]\,\ud v\ud x&=\int_{\Omega}\int_{\r^3}\dt F\ln\left(\frac{F}{1+\theta F}\right)\ud v\ud x\\
    &=-\int_{\Omega}\int_{\r^3}\Big\{v\cdot\nx F+Q[F,F;F]\Big\}\ln\left(\frac{F}{1+\theta F}\right)\ud v\ud x\no\\
    &=-\int_{\Omega}\int_{\r^3}Q[F,F;F]\ln\left(\frac{F}{1+\theta F}\right)\ud v\ud x.\no
\end{align}
Based on Corollary \ref{basic-corollary: invariants 1} with $\phi=\ln\left(\dfrac{F}{1+\theta F}\right)$, we have
\begin{align}\label{basic-estimate: h-theorem}
    &\int_{\r^3}Q[F,F;F]\ln\left(\frac{F}{1+\theta F}\right)\ud v\\
    =&\;\frac{1}{4}\int_{\r^3}\int_{\r^3}\int_{\s^2}q(\o,\abs{v-u})\Big[F(u')F(v')\big(1+\th F(u)\big)\big(1+\th F(v)\big)\no\\
    &\,-F(u)F(v)\big(1+\th F(u')\big)\big(1+\th F(v')\big)\Big]\ln\left(\frac{F(u)F(v)\big(1+\th F(u')\big)\big(1+\th F(v')}{F(u')F(v')\big(1+\th F(u)\big)\big(1+\th F(v)\big)}\right)\ud\o\ud u\ud v\no\\
    =&\;\frac{1}{4}\int_{\r^3}\int_{\r^3}\int_{\s^2}q(\o,\abs{v-u})F(u')F(v')\big(1+\th F(u)\big)\big(1+\th F(v)\big)(1-A)\ln(A)\,\ud\o\ud u\ud v,\no
\end{align}
where $A:=\dfrac{F(u)F(v)\big(1+\th F(u')\big)\big(1+\th F(v')\big)}{F(u')F(v')\big(1+\th F(u)\big)\big(1+\th F(v)\big)}$. Since the function $(1-A)\ln(A)\leq 0$ for any $A>0$, we thus have
\begin{align}
    \int_{\Omega}\int_{\r^3}Q[F,F;F]\ln\left(\frac{F}{1+\theta F}\right)\ud v\ud x\leq 0.
\end{align}
Hence, we know
\begin{align}
    \frac{\ud}{\ud t}\int_{\Omega}\int_{\r^3}S[F]\,\ud v\ud x\geq0.
\end{align}
\end{proof}

\begin{remark}
We usually call $S[F]$ the entropy density.
\end{remark}

\smallskip
\subsection{Equilibrium: the Quantum Maxwellian}

In this section, we study the equilibrium, which is independent of time and space.
\begin{theorem}[Equilibrium]\label{basic-theorem: equilibrium}
The equilibrium (a.k.a. global Maxwellian) of the equation \eqref{equation: boltzmann} is
\begin{align}
    \m(v)=\frac{1}{\ue^{-(a+b\cdot v+c\abs{v}^2)}-\th},
\end{align}
for some $a,c\in\r$, $b\in\r^3$ with $c<0$. For bosons $\th=1$, we also require $a<0$.
\end{theorem}

\begin{proof}
An equilibrium means that it is invariant when time evolves. 
Suppose that the solution $F(t,x,v)$ to the equation \eqref{equation: boltzmann} is an equilibrium, then we must have 
\begin{align}
    \frac{\ud}{\ud t}\int_{\Omega}\int_{\r^3}S[F]\,\ud v\ud x \equiv 0,
\end{align}
which, from the proof of Theorem \ref{basic-theorem: h-theorem}, yields
\begin{align}
    \int_{\Omega}\int_{\r^3}Q[F,F;F]\ln\left(\frac{F}{1+\theta F}\right)\ud v\ud x\equiv 0.
\end{align}
We use the notation from the proof of Theorem \ref{basic-theorem: h-theorem}. In the last line of \eqref{basic-estimate: h-theorem}, since the integrand is of one sign, we must have $A=1$, i.e.
\begin{align}
    \dfrac{F(u)F(v)}{\big(1+\th F(u)\big)\big(1+\th F(v)\big)}=\frac{F(u')F(v')}{\big(1+\th F(u')\big)\big(1+\th F(v')\big)}.
\end{align}
Hence, for $\phi=\ln\left(\dfrac{F}{1+\theta F}\right)$, we must have
\begin{align}
    \phi(v)+\phi(u)=\phi(v')+\phi(u').
\end{align}
Then based on \cite[Lemma 1.7.2]{Glassey1996}, for continuous $\phi$, we must have
\begin{align}
    \phi(v)=a+b\cdot v+c\abs{v}^2,
\end{align}
for some $a,c\in\r$, $b\in\r^3$ with $c<0$ (to guarantee integrability). For bosons $\th=1$, we also require $a<0$ to avoid singularity of the denominator. Therefore, we know that if $F$ is an equilibrium, then it must have the form
\begin{align}\label{basic-formula: equilibrium}
    F(v)=\frac{1}{\ue^{-(a+b\cdot v+c\abs{v}^2)}-\th}.
\end{align}
On the other hand, it is easy to verify that such $F$ indeed satisfies the equation \eqref{equation: boltzmann}.

\end{proof}

\begin{remark}
Actually, there exist solutions to the equation \eqref{equation: boltzmann} in the form of \eqref{basic-formula: equilibrium} where $a,b,c$ depend on time and space. We call such a solution the local Maxwellian.
\end{remark}

\begin{remark}
Based on Escobedo-Mischler-Valle \cite{Escobedo.Mischler.Valle2005} on miminization of the entropy functional, for bosons we actually allow the presence of $\d$-function. 
For any given mass, momentum and energy, there exists an equilibrium in the sense of distribution of the form
\begin{align}
    \frac{1}{\ue^{-(a+b\cdot v+c\abs{v}^2)}-1}+d\delta_{-\frac{b}{2c}},
\end{align}
where $a,c,d\in\r$ and $b\in\r^3$.
The presence of $\d$-function at finite time indicates the Bose-Einstein condensation.
\end{remark}

For now on, we will consider the simplest equilibrium with $b=\mathbf{0}$ and $c=-1$, which is
\begin{align}
    \m(v)=\frac{1}{\ue^{\abs{v}^2-a}-\th}.
\end{align}
When $\th=0$ and $a=0$ for the classical particles, this reduces to the standard Maxwellian $\m(v)=\ue^{-\abs{v}^2}$. Actually, we allow any $a\in\r$. However, for Fermion gas ($\th=-1$) or Boson gas ($\th=1$), this Maxwellian is highly non-trivial. 

For fermions, the Maxwellian is well-defined for any $a\in\r$. 
However, for bosons, we must require $a<0$ to guarantee the positivity of $F$. If $a=0$, the Maxwellian might contain a singularity at $v=0$. 


To handle all cases in a uniform fashion, we require $a<0$ and denote
\begin{align}
    \m(v):=\frac{1}{\vh\ue^{\abs{v}^2}-\th}
\end{align}
with $\vh=\ue^{-a} >1$.



\begin{remark}
Although it is easy to justify that for $\vh>1$,
\begin{align}
    \int_{\r^3}\frac{1}{\vh\ue^{\abs{v}^2}-\th}\,\ud v<\infty,
\end{align}
it is almost impossible to evaluate this integral explicitly. This has serious consequences. For the classical Boltzmann equation with hard-sphere collision, the analysis heavily depends on the explicit computation of such type of integrals (see the beautiful arguments in \cite[Section 3.2]{Glassey1996} for Gaussian functions). However, now we have to find other approaches to get around this difficulty.

\end{remark}


\bigskip
\section{Global Stability of the Maxwellian} \label{Sec:Maxwellian-GlbSol}

\subsection{Perturbation Formulation}

\subsubsection{The Quantum Boltzmann Operator}

Recall the quantum Boltzmann collision operator 
\begin{align}\label{temp 1}
        Q[F,F;F]:=&\int_{\r^3}\int_{\s^2}q(\o,\abs{v-u})\Big[F(u')F(v')\big(1+\th F(u)\big)\big(1+\th F(v)\big)\\
        &-F(u)F(v)\big(1+\th F(u')\big)\big(1+\th F(v')\big)\Big]\ud\o\ud u\no\\
        =&\int_{\r^3}\int_{\s^2}q(\o,\abs{v-u})\Big[F(u' )F(v')\big(1+\th F(u)+\th F(v)\big)\no\\
    &-F(u)F(v)\big(1+\th F(u')+\th F(v')\big)\Big]\ud\o\ud u.\no
\end{align}




We may further decompose
\begin{align}
    Q[F,F;F]&=Q_1[F,F]+\th Q_2[F,F;F],
\end{align}
where
\begin{align}
    Q_1[F,F]&=\int_{\r^3}\int_{\s^2}q(\o,\abs{v-u})\Big[F(u' )F(v')-F(u)F(v)\Big]\ud\o\ud u,\\
    Q_2[F,F;F]&=\int_{\r^3}\int_{\s^2}q(\o,\abs{v-u})\Big[F(u' )F(v')\big(F(u)+F(v)\big)-F(u)F(v)\big(F(u')+F(v')\big)\Big]\ud\o\ud u.
\end{align}
Here $Q_1$ is identical to the classical Boltzmann collision operator, and $Q_2$ is a trilinear form which contains the quantum effects. Note that $Q_2$ is not necessarily smaller than $Q_1$ (we should simply regard it as a correction), so the quantum effects will play a significant role.

Denote the symmetrized operators
\begin{align}
    \qq[F,G]&:=\frac{1}{2}\int_{\r^3}\int_{\s^2}q(\o,\abs{v-u})\Big[F(u' )G(v')+G(u' )F(v')-F(u)G(v)-G(u)F(v)\Big]\ud\o\ud u,
\end{align}
and
\begin{align}
    \q[F,G;H]:=\;&\frac{1}{2}\int_{\r^3}\int_{\s^2}q(\o,\abs{v-u})\Big[F(u' )G(v')\big(H(u)+H(v)\big)+G(u' )F(v')\big(H(u)+H(v)\big)\\
    &-F(u)G(v)\big(H(u')+H(v')\big)-G(u)F(v)\big(H(u')+H(v')\big)\Big]\ud\o\ud u.\no
\end{align}
Obviously, we know
\begin{align}
    \qq[F,F]=Q_1[F,F],\quad \q[F,F;F]=Q_2[F,F;F].
\end{align}



Recall that the quantum Maxwellian is
\begin{align} \label{Maxwellian}
    \m(v):=\frac{1}{\vh\ue^{\abs{v}^2}-\th}
\end{align}
with $\vh>1$.
Define also
\begin{align} \label{Maxwellian-M}
    \mm(v):=\m(1+\th\m)=\m+\th\m^2=\frac{\vh\ue^{\abs{v}^2}}{(\vh\ue^{\abs{v}^2}-\th)^2}.
\end{align}

We define the perturbation $f(t,x,v)$ via
\begin{align} \label{perturbation}
    F(t,x,v)=\m(v)+\mh(v)f(t,x,v),
\end{align}
with
\begin{align} \label{perturbation_0}
    F_0(x,v)=\m(v)+\mh(v)f_0(x,v),
\end{align}
satisfying the conservation laws
\begin{align}
    \int_{\Omega}\int_{\r^3} f(t,x,v)\mh(v)\,\ud v\ud x&=\int_{\Omega}\int_{\r^3}f_0(x,v)\mh(v)\,\ud v\ud x=0,\quad (\text{Mass}) \label{CL-Mass}\\
    \int_{\Omega}\int_{\r^3}f(t,x,v)\mh(v)v_i\,\ud v\ud x&=\int_{\Omega}\int_{\r^3}f_0(x,v)\mh(v)v_i\,\ud v\ud x=0,\quad (\text{Momentum}) \label{CL-Momentum}\\
    \int_{\Omega}\int_{\r^3}f(t,x,v)\mh(v)\abs{v}^2\ud v\ud x&=\int_{\Omega}\int_{\r^3}f_0(x,v)\mh(v)\abs{v}^2\ud v\ud x=0.\quad (\text{Energy}) \label{CL-Energy}
\end{align}
Then we may write the equation \eqref{equation: boltzmann} as
\begin{align}\label{equation: prelim-perturbation}
    \mh\dt f+\mh\Big(v\cdot\nx f\Big)=\qq\left[\m+\mh f,\m+\mh f\right]+\th\q\left[\m+\mh f,\m+\mh f;\m+\mh f\right].
\end{align}
Note that 
\begin{align}
    \qq[\m,\m]+\th\q[\m,\m;\m]=Q[\m,\m;\m]=0.
\end{align}
Hence, we have
\begin{align}
    &\qq\left[\m+\mh f,\m+\mh f\right]+\th\q\left[\m+\mh f,\m+\mh f;\m+\mh f\right]\\
    =\;&2\qq\left[\m,\mh f\right]+\qq\left[\mh f,\mh f\right]\no\\
    &+\th\bigg(2\q\left[\m,\mh f;\m\right]+\q\left[\m,\m;\mh f\right]\bigg)+\th\bigg(\q\left[\mh f,\mh f;\m\right]+2\q\left[\m,\mh f;\mh f\right]\bigg)\no\\
    &+\th\q\left[\mh f,\mh f;\mh f\right].\no
\end{align}
Hence, we may rewrite the equation \eqref{equation: prelim-perturbation} as
\begin{align} \label{equation: prelim-perturbation-2}
    \dt f+v\cdot\nx f+\l[f]=\g[f,f;f],
\end{align}
where
\begin{align}
    \l[f]:=&-2\mhh\qq\left[\m,\mh f\right]-2\th\mhh\q\left[\m,\mh f;\m\right]-\th\mhh\q\left[\m,\m;\mh f\right],
\end{align}
and
\begin{align}
    \g[f,f;f]:=\;&\mhh\qq\left[\mh f,\mh f\right]+\th\mhh\q\left[\mh f,\mh f;\m\right]+2\th\mhh\q\left[\m,\mh f;\mh f\right]\\
    &+\th\mhh\q\left[\mh f,\mh f;\mh f\right].\no
\end{align}
Here $\l[f]$ is a linear operator on $f$ and $\g[f,f;f]$ is a nonlinear operator on $f$.

\begin{remark}
$\mm(v)$ is chosen in such a way that after ``linearization'' it is convenient to express the null space of the linear operator $L$ in terms of $\mm(v)$
(see Lemma\;\ref{prelim-theorem: L-properties} and its proof).
Note that unlike the standard perturbation formulation for the classical Boltzmann and Landau equations (see Glassey \cite[Section 3]{Glassey1996}), here $\mm$ does not necessarily coincide with $\m$.
\end{remark}



\subsubsection{Linear Estimates}

Recall that
\begin{align} \label{L-operator}
    \l[f]&=-2\mhh\qq\left[\m,\mh f\right]-2\th\mhh\q\left[\m,\mh f;\m\right]-\th\mhh\q\left[\m,\m;\mh f\right].
\end{align}
We use $\langle \cdot,\cdot\rangle$ to denote the standard $L^2$ inner product on $\r_v^3$ for a pair of functions $f(v)$ and $g(v)$:
$$\langle f,g\rangle := \int_{\r^3} f(v)g(v) \,\ud v \,.$$

% ...

\begin{lemma}[Properties of $\l$]\label{prelim-theorem: L-properties}
(1) Non-negativity: 
For any $f(v)$ small at infinity, we have
$\langle \l[f],f\rangle \geq0$,
and the equality holds if and only if \,$f(v)=\mh\left(a+b\cdot v+c\abs{v}^2\right)$ with $a,c\in\r$ and $b\in\r^3$.

(2) Self-adjointness: 
$\l$ is a self-adjoint (symmetric) operator, i.e. for any $f(v),g(v)$ small at infinity, we have
$\langle \l[f],g\rangle = \langle f,\l[g]\rangle$.

(3) Null space: 
$\l[f]=0$\, if and only if \,$f(v)=\mh\Big(a+b\cdot v+c\abs{v}^2\Big)$ with $a,c\in\r$ and $b\in\r^3$.

\end{lemma}

\begin{proof}
(1)
We may write each term explicitly:
\begin{align}
    &-2\mhh\qq\left[\m,\mh f\right]\\
    =&-\mhh(v)\int_{\r^3}\int_{\s^2}q(\o,\abs{v-u})\Big[\m(u' )\mh(v')f(v')+\mh(u' )\m(v')f(u')
    -\m(u)\mh(v)f(v)-\mh(u)\m(v)f(u)\Big]\ud\o\ud u,\no
\end{align}
\begin{align}
    &-2\th\mhh\q\left[\m,\mh f;\m\right]\\
    =&-\th\mhh(v)\int_{\r^3}\int_{\s^2}q(\o,\abs{v-u})\Big[\m(u' )\mh(v')f(v')\big(\m(u)+\m(v)\big)+\mh(u' )\m(v')f(u')\big(\m(u)+\m(v)\big)\no\\
    &-\m(u)\mh(v)f(v)\big(\m(u')+\m(v')\big)-\mh(u)\m(v)f(u)\big(\m(u')+\m(v')\big)\Big]\ud\o\ud u,\no
\end{align}
and
\begin{align}
    &-\th\mhh\q\left[\m,\m;\mh f\right]\\
    =&-\frac{1}{2}\th\mhh(v)\int_{\r^3}\int_{\s^2}q(\o,\abs{v-u})\Big[\m(u' )\m(v')\big(\mh(u)f(u)+\mh(v)f(v)\big)+\m(u' )\m(v')\big(\mh(u)f(u)+\mh(v)f(v)\big)\no\\
    &-\m(u)\m(v)\big(\mh(u')f(u')+\mh(v')f(v')\big)-\m(u)\m(v)\big(\mh(u')f(u')+\mh(v')f(v')\big)\Big]\ud\o\ud u\no\\
    =&-\th\mhh(v)\int_{\r^3}\int_{\s^2}q(\o,\abs{v-u})\Big[\m(u' )\m(v')\big(\mh(u)f(u)+\mh(v)f(v)\big)\no\\
    &-\m(u)\m(v)\big(\mh(u')f(u')+\mh(v')f(v')\big)\Big]\ud\o\ud u.\no
\end{align}
Hence, summarizing all above, we have
\begin{align}
    \l[f]=\;&\mhh(v)\int_{\r^3}\int_{\s^2}q(\o,\abs{v-u})\bigg[\Big(\m(v)+\th\m(v)\m(u')+\th\m(v)\m(v')-\th\m(u')\m(v')\Big)\mh(u)f(u)\\
    &+\Big(\m(u)+\th\m(u)\m(u')+\th\m(u)\m(v')-\th\m(u')\m(v')\Big)\mh(v)f(v)\no\\
    &-\Big(\m(v')+\th\m(v')\m(u)+\th\m(v')\m(v)-\th\m(u)\m(v)\Big)\mh(u')f(u')\no\\
    &-\Big(\m(u')+\th\m(u')\m(u)+\th\m(u')\m(v)-\th\m(u)\m(v)\Big)\mh(v')f(v')\bigg]\ud\o\ud u\no.
\end{align}
Direct computation reveals that
\begin{align}
    \m(v)+\th\m(v)\m(u')+\th\m(v)\m(v')-\th\m(u')\m(v')&=\m(v)\m(u')\m(v')\vh\ue^{\abs{v}^2}\Big(\vh\ue^{\abs{u}^2}-\th\Big),\\
    \m(u)+\th\m(u)\m(u')+\th\m(u)\m(v')-\th\m(u')\m(v')&=\m(u)\m(u')\m(v')\vh\ue^{\abs{u}^2}\Big(\vh\ue^{\abs{v}^2}-\th\Big),\\
    \m(v')+\th\m(v')\m(u)+\th\m(v')\m(v)-\th\m(u)\m(v)&=\m(v')\m(u)\m(v)\vh\ue^{\abs{v'}^2}\Big(\vh\ue^{\abs{u'}^2}-\th\Big),\\
    \m(u')+\th\m(u')\m(u)+\th\m(u')\m(v)-\th\m(u)\m(v)&=
    \m(u')\m(u)\m(v)\vh\ue^{\abs{u'}^2}\Big(\vh\ue^{\abs{v'}^2}-\th\Big).
\end{align}
Hence, we have
\begin{align}\label{prelim-equality: temp 1}
    \l[f]=\;&\mhh(v)\int_{\r^3}\int_{\s^2}q(\o,\abs{v-u})\m(u)\m(v)\m(u')\m(v')\vh^2\ue^{\abs{u}^2+\abs{v}^2}\\
    &\bigg[\m^{-1}(u)\vh\ue^{-\abs{u}^2}\Big(\vh\ue^{\abs{u}^2}-\th\Big)\mh(u)f(u)+\m^{-1}(v)\vh\ue^{-\abs{v}^2}\Big(\vh\ue^{\abs{v}^2}-\th\Big)\mh(v)f(v)\no\\
    &-\m^{-1}(u')\vh\ue^{-\abs{u'}^2}\Big(\vh\ue^{\abs{u'}^2}-\th\Big)\mh(u')f(u')-\m^{-1}(v')\vh\ue^{-\abs{v'}^2}\Big(\vh\ue^{\abs{v'}^2}-\th\Big)\mh(v')f(v')\bigg]\ud\o\ud u\no\\
    =\;&\mhh(v)\int_{\r^3}\int_{\s^2}q(\o,\abs{v-u})\m(u)\m(v)\m(u')\m(v')\vh^2\ue^{\abs{u}^2+\abs{v}^2}\no\\
    &\Big[\mhh(u)f(u)+\mhh(v)f(v)-\mhh(u')f(u')-\mhh(v')f(v')\Big]\ud\o\ud u,\no
\end{align}
due to $\abs{u}^2+\abs{v}^2=\abs{u'}^2+\abs{v'}^2$ and $\m^{-1}(v)\vh\ue^{-\abs{v}^2}\Big(\vh\ue^{\abs{v}^2}-\th\Big)=\mm^{-1}(v)$.

Hence, using the similar techniques as in the proof of Theorem \ref{basic-theorem: h-theorem} (by symmetry of the ``kernel'' and change of variables), we have that
\begin{align}\label{prelim-equality: temp 2}
    \br{\l[f],f}=\;&\frac{1}{4}\int_{\r^3}\int_{\r^3}\int_{\s^2}q(\o,\abs{v-u})\m(u)\m(v)\m(u')\m(v')\vh^2\ue^{\abs{u}^2+\abs{v}^2}\\
    &\Big[\mhh(u)f(u)+\mhh(v)f(v)-\mhh(u')f(u')-\mhh(v')f(v')\Big]^2\ud\o\ud u\ud v.\no
\end{align}
Since the integrand is always nonnegative, we know 
\begin{align}
    \br{\l[f],f}\geq0.
\end{align}
In particular, if the equality holds, then we have
\begin{align}
    \mhh(u)f(u)+\mhh(v)f(v)-\mhh(u')f(u')-\mhh(v')f(v')=0,
\end{align}
which implies that
\begin{align}
    f(v)=\mh\Big(a+b\cdot v+c\abs{v}^2\Big),
\end{align}
for $a,c\in\r$ and $b\in\r^3$.

(2)
Based on \eqref{prelim-equality: temp 1} and \eqref{prelim-equality: temp 2} in the proof of (1), it is clear that $\l$ is self-adjoint.

(3)
If $f(v)=\mh\Big(a+b\cdot v+c\abs{v}^2\Big)$, inserting it into \eqref{prelim-equality: temp 1}, we know $\l[f]=0$. On the other hand, if $\l[f]=0$, then using \eqref{prelim-equality: temp 2} (being zero and noting the non-negativity of the integrand), we have $f(v)=\mh\Big(a+b\cdot v+c\abs{v}^2\Big)$ for some $a,c\in\r$ and $b\in\r^3$.
\end{proof}



Now we know that the linear operator $\l$ given by \eqref{L-operator} is a self-adjoint non-negative operator on $L^2_v(\r^3)$.
Its null space (kernel) is a ﬁve-dimensional subspace of $L^2_v(\r^3)$ spanned by \,$\left\{\mh,\, v\mh,\, |v|^2\mh \right\}$\,.
We introduce the following notations:

\begin{definition}[Projection $\pk$ onto the null space of $\l$]
Let \,$\mathbf{N}(\l) := \big\{\,g\in L^2_v(\r^3) :\, \l [g]=0 \,\big\}$\, denote the null space of the linear operator $\l$\, with 
a set of basis
\,$e_0:=\mh\,,\, e_i:=v_i\mh\;\,(i=1,2,3)\,,\, e_4:=|v|^2\mh $\,,
and write 
\begin{equation} \label{N(L)-basis}
\mathbf{N}(\l) \,:=\,  {\rm span} \left\{\mh\,,\, v_i\mh\;\,(i=1,2,3)\,,\, |v|^2\mh\, \right\} \,.
\end{equation}
For the function \,$f(t,x,v)$\, with fixed \,$(t,x)$\,, we define the projection of \,$f$\, in \,$L^2_v(\r^3)$\, onto \,$\mathbf{N}(\l)$\, as
$$\pk f\,(t,x,v) \,:=\, \sum_{i=0}^4 \big\langle f(t,x,\cdot\,), e_i \big\rangle\, e_i 
\,=:\, \left\{\mathfrak{a}_f(t,x) + v\cdot \mathfrak{b}_f(t,x) + |v|^2\, \mathfrak{c}_f(t,x) \right\}\mh \,,$$
where 
\begin{align*}
\mathfrak{a}_f(t,x) &\,:=\, \big\langle f(t,x,\cdot\,), e_0 \big\rangle 
\,=\, \int_{\r^3}\! f(t,x,\cdot\,)\mh\, \ud v \,, \\
\mathfrak{b}_f^{i}(t,x) &\,:=\, \big\langle f(t,x,\cdot\,), e_i \big\rangle 
\,=\, \int_{\r^3}\! f(t,x,\cdot\,) v_i\mh\, \ud v \,, \\
\mathfrak{c}_f(t,x) &\,:=\, \big\langle f(t,x,\cdot\,), e_4 \big\rangle 
\,=\, \int_{\r^3}\! f(t,x,\cdot\,) |v|^2\mh\, \ud v \,.
\end{align*}
\end{definition}

From \eqref{prelim-equality: temp 1}, we may rewrite
\begin{align}
    \l[f]=\nu f-K[f],
\end{align}
where
\begin{align}
    \nu(v)&:=\int_{\r^3}\int_{\s^2}q(\o,\abs{v-u})\m(u)\m(v)\m(u')\m(v')\vh^2\ue^{\abs{u}^2+\abs{v}^2}\mm^{-1}(v)\,\ud\o\ud u,
\end{align}
and $K=K_2-K_1$ for
\begin{align}
    K_1[f]:=\;&\mhh(v)\int_{\r^3}\int_{\s^2}q(\o,\abs{v-u})\m(u)\m(v)\m(u')\m(v')\vh^2\ue^{\abs{u}^2+\abs{v}^2}\mhh(u)f(u)\,\ud\o\ud u,\\
    \\
    K_2[f]:=\;&\mhh(v)\int_{\r^3}\int_{\s^2}q(\o,\abs{v-u})\m(u)\m(v)\m(u')\m(v')\vh^2\ue^{\abs{u}^2+\abs{v}^2}\Big(\mhh(u')f(u')+\mhh(v')f(v')\Big)\ud\o\ud u.\no
\end{align}

\begin{lemma}[$\nu$ Estimate]\label{prelim-lemma: nu-estimate}
$\nu(v)$ satisfies that 
\begin{align}
    \nu_1\big(1+\abs{v}\big)\leq\nu(v)\leq\nu_2\big(1+\abs{v}\big),
\end{align}
for some $\nu_1,\nu_2>0$.
\end{lemma}

\begin{proof}
We have the naive bounds for $\m$ and $\mm$:
\begin{align}\label{temp 2}
    \ue^{-\abs{v}^2}\ls \m(v)\ls \ue^{-\abs{v}^2},\quad \ue^{-\abs{v}^2}\ls \mm(v)\ls \ue^{-\abs{v}^2}.
\end{align}
Hence, considering $\abs{u}^2+\abs{v}^2=\abs{u'}^2+\abs{v'}^2$, we may bound
\begin{align}
    \int_{\r^3}\int_{\s^2}q(\o,\abs{v-u})\ue^{-\abs{u}^2}\ud\o\ud u\ls \nu(v)\ls \int_{\r^3}\int_{\s^2}q(\o,\abs{v-u})\ue^{-\abs{u}^2}\ud\o\ud u,
\end{align}
which implies
\begin{align}
    \int_{\r^3}\abs{v-u}\ue^{-\abs{u}^2}\ud\o\ud u\ls \nu(v)\ls \int_{\r^3}\abs{v-u}\ue^{-\abs{u}^2}\ud\o\ud u.
\end{align}
Then following the same proof as \cite[Section\;3.3.1]{Glassey1996}, we obtain the desired result.
\end{proof}

\begin{lemma}\label{prelim-lemma: nu-derivative-estimate}
We have $\abs{\p_{\beta}\nu}\ls 1$ for $\abs{\beta}\geq 1$.
\end{lemma}

\begin{proof}
We first rearrange $\nu$ as
\begin{align}
    \nu(v)&=\int_{\r^3}\int_{\s^2}\Big\{q(\o,\abs{v-u})\m(u)\Big\}\Big\{\m(v)\mm^{-1}(v)\Big\}\Big\{\m(u')\m(v')\vh^2\ue^{\abs{u}^2+\abs{v}^2}\Big\}\,\ud\o\ud u\\
    &=\int_{\r^3}\int_{\s^2}A(u,v,\o)B(v)C(u,v,\o)\,\ud\o\ud u.\no
\end{align}
We directly take $v$ derivatives in $\nu$, so it might hit one of $A,B,C$. Then we have
\begin{align}
    \nv A\sim\;&\nv\big(\abs{v-u}\m(u)\big)=\frac{v-u}{\abs{v-u}}\m(u),\\
    \nv B=\;&\nv\left(1-\vh\th\ue^{-\abs{v}^2}\right)=2v\vh\th \ue^{-\abs{v}^2},\\
    \nv C=\;&\nv\left(\frac{1}{\vh^2-\vh\th(\ue^{-\abs{u'}^2}+\ue^{-\abs{v'}^2})+\th^2\ue^{-\abs{u}^2-\abs{v}^2}}\right)\\
    =\;&\frac{2\vh\th\Big(u'\cdot\nv u'\ue^{-\abs{u'}^2}+v'\cdot\nv v'\ue^{-\abs{v'}^2}\Big)-2v\th^2\ue^{-\abs{u}^2-\abs{v}^2}}{\Big(\vh^2-\vh\th(\ue^{-\abs{u'}^2}+\ue^{-\abs{v'}^2})+\th^2\ue^{-\abs{u}^2-\abs{v}^2}\Big)^2},\no
\end{align}
where $\nv u'=\o\otimes\o$ and $\nv v'=I-\o\otimes\o$. Then we have
\begin{align}
    \abs{\nv A(u,v,\o)}\ls\;& \m(u),\\
    \abs{\nv B(v)}\ls\;&\ue^{-\frac{1}{2}\abs{v}^2},\\
    \abs{\nv C(u,v,\o)}\ls\;&\ue^{-\frac{1}{2}\abs{u'}^2}+\ue^{-\frac{1}{2}\abs{v'}^2}+\ue^{-\frac{1}{2}\abs{u}^2-\frac{1}{2}\abs{v}^2}.
\end{align}
Then using the substitution $v-u\rt w$ (the integral is transformed to be with respect to $\ud w$), we know
\begin{align}
    A(w,v,\o)&=\frac{w}{\abs{w}}\m(v-w).
\end{align}
Obviously, we know $\abs{\p_{\beta}\m(v-w)}\ls \ue^{-\frac{1}{2}\abs{v-w}^2}$. $B(v)$ does not change and any $v$ derivative will be controlled by $\ue^{-\abs{v}^2}$. The structure of $C(w,v,\o)$ will be preserved. In particular,
\begin{align}
    u'=v-w+\o(\o\cdot w),\quad v'=v-\o(\o\cdot w).
\end{align}
Hence, any $v$ derivative will be controlled by $\ue^{-\abs{u'}^2}$, $\ue^{-\abs{v'}^2}$ and $\ue^{-\abs{u}^2-\abs{v}^2}$, 
and thus our result follows.
\end{proof}

\begin{definition}[$\nu$ Norms]\label{nu-norm}
We define 
\begin{align}
    \vnm{f}{v}:=\bigg(\int_{\r^3}\nu(v) \abs{f(v)}^2\ud v\bigg)^{\frac{1}{2}}
    \sim |f|_{2,\frac{1}{2}},
\end{align}
and
\begin{align}
    \vnm{f}{x,v}:=\bigg(\int_{\Omega}\int_{\r^3}\nu(v) \abs{f(v)}^2\ud v\ud x\bigg)^{\frac{1}{2}}
    \sim \|f\|_{2,\frac{1}{2}},
\end{align}
\end{definition}

\begin{lemma}[Compactness of $K$]\label{prelim-lemma: K-compactness}
$K$ is a compact operator on $L^{\nu}_v(\r^3)$ and on $L^{2}_v(\r^3)$. 
\end{lemma}

\begin{proof}
We may denote
\begin{align}
    K_i[f]=\int_{\r^3}k_i(u,v)f(u)\,\ud u
\end{align}
for some kernel functions $k_i$, $i=1,2$.

We first consider $K_1$. Obviously,
\begin{align}
    k_1(u,v)=\vh^2\ue^{\abs{u}^2+\abs{v}^2}\mhh(u)\mhh(v)\m(u)\m(v)\int_{\s^2}q(\o,\abs{v-u})\m(u')\m(v')\,\ud\o.
\end{align}
Based on a \eqref{temp 2}, we know 
\begin{align}
    \ue^{-\frac{\abs{u}^2}{2}-\frac{\abs{v}^2}{2}}\ls\mhh(u)\mhh(v)\m(u)\m(v)\ls \ue^{-\frac{\abs{u}^2}{2}-\frac{\abs{v}^2}{2}},
\end{align}
and
\begin{align}
    1\ls\ue^{\abs{u}^2+\abs{v}^2}\m(u')\m(v')\ls 1.
\end{align}
Thus, though we cannot derive an explicit formula for $k_1$, we know it can be well-controlled by the corresponding $\tilde k_1$ for the classical Boltzmann equation, where
\begin{align}
    \tilde k_1(u,v)= \pi\abs{v-u}\exp\left(-\frac{\abs{u}^2}{2}-\frac{\abs{v}^2}{2}\right).
\end{align}
For $K_2$, $k_2$ is not easy to obtain. We first split
\begin{align}
    K_2[f]=&\int_{\r^3}\int_{\s^2}q(\o,\abs{v-u})\m(u)\m(v)\m(u')\m(v')\vh^2\ue^{\abs{u}^2+\abs{v}^2}\mhh(v)\mhh(u')f(u')\,\ud\o\ud u\\
    &+\int_{\r^3}\int_{\s^2}q(\o,\abs{v-u})\m(u)\m(v)\m(u')\m(v')\vh^2\ue^{\abs{u}^2+\abs{v}^2}\mhh(v)\mhh(v')f(v')\,\ud\o\ud u.\no
\end{align}
Following the argument as in \cite[(3.34),(3.51)]{Glassey1996}, we obtain that
\begin{align}
    K_2[f]=&\int_{\r}\bigg\{\int_{\r^2}\frac{2}{\abs{\eta-v}}\Big[\m(\eta+V_{\perp})\m(v)\m(v+V_{\perp})\m(\eta)\vh^2\ue^{\abs{\eta+V_{\perp}}^2+\abs{v}^2}\mhh(v)\mhh(\eta)\Big]\ud{V_{\perp}}\bigg\}f(\eta)\,\ud\eta
\end{align}
where
\begin{align}
    &V=u-v,\quad V_{\parallel}=(V\cdot\o)\o,\quad V_{\perp}=V-(V\cdot\o)\o,\quad\eta=v+V_{\parallel},
\end{align}
or equivalently
\begin{align}
    u=\eta+V_{\perp},\quad u'=v+V_{\perp},\quad v'=\eta.
\end{align}
Therefore, we know
\begin{align}
    k_2(v,\eta)=\int_{\r^2}\frac{2}{\abs{\eta-v}}\Big[\m(\eta+V_{\perp})\m(v)\m(v+V_{\perp})\m(\eta)\vh^2\ue^{\abs{\eta+V_{\perp}}^2+\abs{v}^2}\mhh(v)\mhh(\eta)\Big]\ud{V_{\perp}}.
\end{align}
Unfortunately, due to the complexity of $\m$ and $\mm$ in the quantum Boltzmann equation, we can hardly further simplify $k_2$ as in \cite[(3.45),(3.52)]{Glassey1996} to get an explicit formula. Hence, we turn to direct bounds. 

Note that in the original variables
\begin{align}
    \ue^{-\frac{1}{2}\abs{u}^2-\frac{1}{2}\abs{u'}^2}\ls\m(u)\m(v)\m(u')\m(v')\vh^2\ue^{\abs{u}^2+\abs{v}^2}\mhh(v)\mhh(v')\ls \ue^{-\frac{1}{2}\abs{u}^2-\frac{1}{2}\abs{u'}^2}.
\end{align}
Thus in the new variables, we know
\begin{align}
\\
    \ue^{-\frac{1}{2}\abs{\eta+V_{\perp}}^2-\frac{1}{2}\abs{v+V_{\perp}}^2}\ls\m(\eta+V_{\perp})\m(v)\m(v+V_{\perp})\m(\eta)\vh^2\ue^{\abs{\eta+V_{\perp}}^2+\abs{v}^2}\mhh(v)\mhh(\eta)\ls \ue^{-\frac{1}{2}\abs{\eta+V_{\perp}}^2-\frac{1}{2}\abs{v+V_{\perp}}^2}.\no
\end{align}
Compared with \cite[(3.35)]{Glassey1996}, we know the upper bound and lower bound can be computed explicitly. In other words, though we cannot obtain explicitly formula of $k_2$, we know that it can be well-controlled by the corresponding $\tilde k_2$ for the classical Boltzmann equation, where based on \cite[(3.52)]{Glassey1996}
\begin{align}
    \tilde k_2(u,v)=\frac{2\pi}{\abs{u-v}}\exp\left(-\frac{1}{4}\abs{u-v}^2-\frac{1}{4}\frac{(\abs{u}^2-\abs{v}^2)^2}{\abs{u-v}^2}\right).
\end{align}

Then similarly to the argument in \cite[Section 3.5]{Glassey1996}, we know $K$ is compact.
\end{proof}

\begin{corollary}
We have 
\begin{align}
    \sup_{v\in\r^3}\int_{\r^3}k_1(u,v)\,\ud u&<\infty,\quad \sup_{v\in\r^3}\int_{\r^3}\big|k_1(u,v)\big|^2\ud u<\infty,\\
    \sup_{v\in\r^3}\int_{\r^3}k_2(u,v)\,\ud u&<\infty,\quad \sup_{v\in\r^3}\int_{\r^3}\big|k_2(u,v)\big|^2\ud u<\infty.
\end{align}
\end{corollary}

\begin{proof}
See \cite[Section 3.3.2]{Glassey1996}.
\end{proof}

\begin{corollary}
$K$ is bounded on $L^2_v(\r^3)$ and on $L^{\nu}_v(\r^3)$. 
\end{corollary}

\begin{proof}
See \cite[Section 3.5]{Glassey1996}.
\end{proof}

\begin{corollary}
Let $\alpha>0$ and $k=k_2-k_1$. Then we have 
\begin{align}
\int_{\r^3}\abs{k(u,v)}(1+\abs{u}^2)^{-\frac{\alpha}{2}}\ud u\ls (1+\abs{u}^2)^{-\frac{\alpha+1}{2}}
\end{align}
\end{corollary}

\begin{proof}
See \cite[Lemma 3.3.1]{Glassey1996}.
\end{proof}

\begin{lemma}\label{prelim-lemma: K-derivative}
Let $\abs{\beta}=k$. Then we have
\begin{align}
    \tnm{\p_{\beta}K[g]}{v}^2\ls \tnm{g}{v}^2+\sum_{\abs{\alpha}=k}\tnm{\p_{\alpha}g}{v}^2.
\end{align}
Also, for any small $\eta>0$, there exists $C_{k,\eta}>0$ such that for any $g(v)\in H^k_v(\r^3)$ and $\beta'\leq \beta$, we have
\begin{align}
    \tnm{\p_{\beta'}K[g]}{v}^2\leq C_{k,\eta}\tnm{g}{v}^2+\eta\sum_{\abs{\alpha}=k}\tnm{\p_{\alpha}g}{v}^2.
\end{align}
\end{lemma}

\begin{proof}
Due to the standard interpolation estimate 
\begin{align}
    \tnm{\nv^jg}{v}\ls \tnm{g}{v}\tnm{\nv^kg}{v}
\end{align}
for $0\leq j\leq k$, it suffices to justify the $L^2$ boundedness of $\p_{\beta}K$. Furthermore, it suffices to show that $\p_{\beta}K$ is compact.

We introduce substitution $w=v-u$. Then we regroup $K_1$ to obtain
\begin{align}
    K_1[g]=\;&\mhh(v)\int_{\r^3}\int_{\s^2}\abs{\o\cdot w}\m(v-w)\m(v)\m(u')\m(v')\vh^2\ue^{\abs{v-w}^2+\abs{v}^2}\mhh(u)g(u)\,\ud\o\ud w\\
    =&\int_{\r^3}\int_{\s^2}\Big\{\abs{\o\cdot w}\m(v-w)\mhh(v-w)\Big\}\vh^2\Big\{\mhh(v)\m(v)\Big\}\Big\{\m(u')\m(v')\ue^{\abs{v-w}^2+\abs{v}^2}\Big\}g(u)\,\ud\o\ud w\no\\
    =&\int_{\r^3}\int_{\s^2}A(w,v,\o)B(v)C(w,v,\o)g(v-w)\,\ud\o\ud w.\no
\end{align}
Then similar to the proof of Lemma \ref{prelim-lemma: nu-derivative-estimate}, after taking $v$ derivatives, we know 
\begin{align}
    \abs{\p_{\beta_1}A}\ls\;& \ue^{-\frac{1}{2}\abs{v-w}^2},\\
    \abs{\p_{\beta_2}B}\ls\;&\ue^{-\frac{1}{2}\abs{v}^2},\\
    \abs{\p_{\beta_3}C}\ls\;&\ue^{-\frac{1}{2}\abs{u'}^2-\frac{1}{2}\abs{v'}^2}.
\end{align}
Hence, following the proof of Lemma \ref{prelim-lemma: K-compactness}, we know $\p_{\beta}K_1$ is compact.

Similarly, for $K_2$, using the substitution $w=v-u$ and regrouping, we obtain
\begin{align}
\\
    K_2[g]=\;&\mhh(v)\int_{\r^3}\int_{\s^2}\abs{\o\cdot w}\m(v-w)\m(v)\m(u')\m(v')\vh^2\ue^{\abs{v-w}^2+\abs{v}^2}\Big(\mhh(u')g(u')+\mhh(v')g(v')\Big)\ud\o\ud w\no\\
    =&\int_{\r^3}\int_{\s^2}\Big\{\abs{\o\cdot w}\m^{\frac{1}{2}}(v-w)\Big\}\vh^2\Big\{\mhh(v)\m^{\frac{1}{2}}(v)\Big\}\Big\{\m(u')\m(v')\ue^{\abs{v-w}^2+\abs{v}^2}\Big\}\no\\
    &\Big\{\m^{\frac{1}{2}}(v-w)\m^{\frac{1}{2}}(v)\Big(\mhh(u')g(u')+\mhh(v')g(v')\Big)\Big\}\,\ud\o\ud w\no\\
    =&\int_{\r^3}\int_{\s^2}A(w,v,\o)B(v)C(w,v,\o)D(w,v,\o)\,\ud\o\ud w.\no
\end{align}
Here, $A,B,C$ can be handled as in $K_1$ case, so we focus on $D$. In particular,
\begin{align}
    &\nv\left(\m^{\frac{1}{2}}(v-w)\m^{\frac{1}{2}}(v)\mhh(u')\right)\\
    =&\left(\m^{-\frac{1}{2}}(v-w)\m^{-\frac{1}{2}}(v)\mh(u')\right)\Big(\nv\big(\mm^{-1}(u')\big)\big(\m(v-w)\m(v)\big)+\nv\big(\m(v-w)\m(v)\big)\mm^{-1}(u')\Big).\no
\end{align}
Then direct computation reveals that
\begin{align}
    \abs{\nv\big(\mm^{-1}(u')\big)}\ls\;& \abs{u'}\ue^{|u'|^2},\\
    \big|\nv\big(\m(v-w)\m(v)\big)\big|\ls\;&\big(\abs{v}+\abs{v-w}\big)\ue^{-\abs{v}^2-\abs{v-w}^2}.
\end{align}
Hence, we know
\begin{align}
    \abs{\nv\left(\m^{\frac{1}{2}}(v-w)\m^{\frac{1}{2}}(v)\mhh(u')\right)}\ls\;&\big(\abs{u'}+\abs{v}+\abs{v-w}\big) \ue^{-\frac{1}{2}|v'|^2}\\
    \ls\;&\big(\abs{v'}+\abs{v}+\abs{v-w}\big) \ue^{-\frac{1}{2}|v'|^2}\no\\
    \ls\;&\big(1+\abs{v}+\abs{v-w}\big) \ue^{-\frac{1}{4}|v'|^2}.\no
\end{align}
Similar technique justifies that
\begin{align}
    \abs{\nv\left(\m^{\frac{1}{2}}(v-w)\m^{\frac{1}{2}}(v)\mhh(v')\right)}\ls\;&\big(1+\abs{v}+\abs{v-w}\big) \ue^{-\frac{1}{4}|u'|^2}.
\end{align}
The similar structure will be preserved when taking higher-order $v$ derivatives. Therefore, we know
\begin{align}
    \abs{\p_{\beta_4}D}\ls \Big(1+\abs{v}^{\abs{\beta_4}}+\abs{v-w}^{\abs{\beta_4}}\Big) \Big(\ue^{-\frac{1}{4}\abs{u'}^2}+\ue^{-\frac{1}{4}\abs{v'}^2}\Big)\Big(\abs{\p_{\beta_4}g(u')}+\abs{\p_{\beta_4}g(v')}\Big).
\end{align}
Here, $\Big(1+\abs{v}^{\abs{\beta_4}}+\abs{v-w}^{\abs{\beta_4}}\Big)$ can be handled by $A$ and $B$. Summarizing all above, we know that $v$ derivatives of $K_2$ will not change its fundamental structure, so we may follow the proof of Lemma \ref{prelim-lemma: K-compactness} to show that $\p_{\beta}K_2$ is compact.
\end{proof}




\begin{theorem}[Semi-Positivity of $\l$]\label{prelim-theorem: semi-positivity}
There exists a $\d>0$ such that
\begin{align}
    \br{\l[g],g}\geq \d\vnm{\nnpk[g]}{v}^2
\end{align}
for a general function $g(v)$.
\end{theorem}

\begin{proof}
We prove by contradiction. Assume that there exists a sequence of functions $\{g_n\}_{n=1}^{\infty}$ satisfying $\pk[g_n]=0$, $\vnm{g_n}{v}=1$ and 
\begin{align}\label{prelim-equality: temp 4}
    \br{\l[g_n],g_n}\leq \frac{1}{n}.
\end{align}
Since $L^2_{\nu}$ is a Hilbert space, based on the Eberlain-Shmulyan theorem, we have the weakly convergent sequence (up to extracting a subsequence with an abuse of notation) $g_n\rightharpoonup g$ in $L^2_{\nu}$. Therefore, by the weak semi-continuity, we have
\begin{align}
    \vnm{g}{v}\leq 1.
\end{align}
Notice that
\begin{align}\label{prelim-equality: temp 3}
    \br{\l[g_n],g_n}=\vnm{g_n}{v}^2-\br{K[g_n],g_n}=1-\br{K[g_n],g_n}.
\end{align}
Since $K$ is a compact operator on $L^{\nu}_v$, we know it maps weakly convergent sequence into strongly convergent sequence, i.e.
\begin{align}
    \lim_{n\rt\infty}\vnm{K[g_n]-K[g]}{v}=0.
\end{align}
Hence, we naturally have
\begin{align}
    \lim_{n\rt\infty}\big(\br{K[g_n],g_n}-\br{K[g],g}\big)=0.
\end{align}
Therefore, using \eqref{prelim-equality: temp 4}, we may direct take limit $n\rt\infty$ in \eqref{prelim-equality: temp 3} to get 
\begin{align}
    \br{\l[g],g}=1-\br{K[g],g}=0.
\end{align}
On the other hand, the above equality may be written as
\begin{align}
    \br{\l[g],g}=\big(1-\vnm{g}{v}^2\big)+\big(\vnm{g}{v}^2-\br{K[g],g}\big)=0.
\end{align}
Based on the weak semi-continuity, we just proved that the first term is non-negative. Also, the second term is actually $\br{\l[g],g}$ which is also non-negative due to Lemma \ref{prelim-theorem: L-properties} (1). Hence, both of them must be zero, i.e. $\l[g]=0$ and $\vnm{g}{v}^2=1$. Then based on Lemma \ref{prelim-theorem: L-properties} (3), we know 
\begin{align}
    g=\mh\Big(a+b\cdot v+c\abs{v}^2\Big).
\end{align}
Then our assumption $\pk[g_n]=0$ implies that the limit $\pk[g]=0$, which means $a=c=0$ and $b=0$. Therefore, we must have $g=0$, which contradicts with $\vnm{g}{v}=1$.
\end{proof}


\subsubsection{Nonlinear Estimates}

Recall that
\begin{align} \label{Gamma-opt}
    \g[f_1,f_2;f_3]=\;&\mhh\qq\left[\mh f_1,\mh f_2\right]+\th\mhh\q\left[\mh f_1,\mh f_2;\m\right]+\th\mhh\q\left[\mh f_1,\m;\mh f_3\right]\\
    &+\th\mhh\q\left[\m,\mh f_2;\mh f_3\right]+\th\mhh\q\left[\mh f_1,\mh f_2;\mh f_3\right].\no
\end{align}

\begin{lemma}
For any $f(v)$, we have
\begin{align}
    \pk\big[\g[f,f;f]\big]=0.
\end{align}
\end{lemma}

\begin{proof}
Since 
\begin{align}
    Q\left[\m+\mh f,\m+\mh f;\m+\mh f\right]=\;&Q[\m,\m;\m]+\mh\l[f]+\mh\g[f,f;f]\\
    =\;&\mh\l[f]+\mh\g[f,f;f],\no
\end{align}
and by direct computation in the same spirit as the proof of Corollary\;\ref{basic-corollary: invariants 2},
\begin{align}
    \pk\Big[\mhh Q[\m+\mh f,\m+\mh f;\m+\mh f]\Big]=\pk\big[\l[f]\big]=0,
\end{align}
we must have
\begin{align}
    \pk\big[\g[f,f;f]\big]=0.
\end{align}
\end{proof}

\begin{lemma}\label{prelim-lemma:derivative}
For fixed $\o$ and $v$, the Jacobian of the transformation $u\rightarrow u'$ satisfies 
\begin{align}
    \abs{\frac{\ud u'}{\ud u}}\geq \frac{1}{8},
\end{align}
and for fixed $\o$ and $u$, the Jacobian of the transformation $v\rightarrow v'$ satisfies
\begin{align}
    \abs{\frac{\ud v'}{\ud v}}\geq \frac{1}{8}.
\end{align}
\end{lemma}

\begin{proof}
This is based on the proof of Alexandre-Desvillettes-Villani-Wennberg \cite[Lemma 1]{Alexandre.Desvillettes.Villani.Wennberg2000}. 
\end{proof}

\begin{lemma}\label{prelim-lemma: nonlinear-no-derivative}
Let $f_i$ for $i=1,2,3$, and $g$ be smooth functions. Then we have
\begin{align}\label{temp 3}
    &\abs{\int_{\r^3}\g[f_1,f_2;f_3]g\,\ud v}\\
    \ls&\tnm{g}{v}\tnm{f_1}{v}\tnm{f_2}{v}\tnm{f_3}{v}+\vnm{g}{v}\tnm{f_a}{v}\vnm{f_b}{v}\no\\ &+\vnm{g}{v}\min\bigg\{\Big(\vnm{f_1}{v}\tnm{f_2}{v}+\tnm{f_1}{v}\vnm{f_2}{v}\Big)\sup_v\abs{f_3},\,\sup_v\abs{f_1}\sup_v\abs{f_2}\vnm{f_3}{v}\bigg\}\no\\
    &+\vnm{g}{v}\min\bigg\{\Big(\vnm{f_1}{v}\tnm{f_2}{v}+\tnm{f_1}{v}\vnm{f_2}{v}\Big)\sup_v\abs{f_3},\Big(\sup_v\abs{\nu^{\frac{1}{2}}f_1}\tnm{f_2}{v}+\tnm{f_1}{v}\sup_v\abs{\nu^{\frac{1}{2}}f_2}\Big)\tnm{f_3}{v}\bigg\}\no\\
    &+\vnm{g}{v}\min\bigg\{\vnm{f_1}{v}\sup_v\abs{f_3},\,\sup_v\abs{\nu^{\frac{1}{2}}f_1}\tnm{f_3}{v}\bigg\}+\vnm{g}{v}\min\bigg\{\vnm{f_2}{v}\sup_v\abs{f_3},\,\sup_v\abs{\nu^{\frac{1}{2}}f_2}\tnm{f_3}{v}\bigg\}.\no
\end{align}
Also, we have
\begin{align}\label{temp 4}
    \abs{\int_{\r^3}\g[f_1,f_2;f_3]g\,\ud v}\ls&\;\sup_{v}\abs{\nu^3g}\tnm{f_a}{v}\tnm{f_b}{v}+\Big(\vnm{g}{v}+\sup_{v}\abs{\nu g}\Big)\tnm{f_1}{v}\tnm{f_2}{v}\tnm{f_3}{v},
\end{align}
and
\begin{align}\label{temp 5}
    \big\|\g[f_1,f_2;f_3]g\big\|_{L^2_v}\ls&\,\sup_{v}\abs{\nu g}\bigg(\tnm{f_1}{v}\tnm{f_2}{v}\tnm{f_3}{v}+\tnm{f_a}{v}\tnm{f_b}{v}\\
    &+\min\Big\{\sup_{v}\abs{f_3}\tnm{f_1}{v}\tnm{f_2}{v}, \sup_{v}\abs{f_1}\sup_{v}\abs{f_2}\tnm{f_3}{v}\Big\}\no\\
    &+\min\Big\{ \sup_{v}\abs{f_3}\tnm{f_1}{v}\tnm{f_2}{v},\sup_{v}\abs{f_2}\tnm{f_1}{v}\tnm{f_3}{v} \Big\}\no\\
    &+\min\Big\{\sup_{v}\abs{f_3}\tnm{f_1}{v}\tnm{f_2}{v},\sup_{v}\abs{f_1}\tnm{f_2}{v}\tnm{f_3}{v}\Big\}\bigg)\no.
\end{align}
Here $(a,b)$ runs all combinations of $\{1,2,3\}$.
\end{lemma}

\begin{proof}
We look at formula \eqref{Gamma-opt} for the $\Gamma$ operator and estimate term by term.

{\it \underline{Step\;1}. Estimate of quadratic terms}.
For $\mhh\qq\left[\mh f_1,\mh f_2\right]$ and $\th\mhh\q\left[\mh f_1,\mh f_2;\m\right]$, based on \cite[Lemma 2.3]{Guo2002}, we have
\begin{align}
    \abs{\int_{\r^3}\mhh\qq\left[\mh f_1,\mh f_2\right]g\,\ud v}&\ls\vnm{f_1}{v}\tnm{f_2}{v}\vnm{g}{v},\\
    \abs{\int_{\r^3}\mhh\qq\left[\mh f_1,\mh f_2\right]g\,\ud v}&\ls \sup_{v}\abs{\nu^3g}\tnm{f_1}{v}\tnm{f_2}{v},\\
    \tnm{\mhh\qq\left[\mh f_1,\mh f_2\right]g}{v}&\ls\sup_{v}\abs{\nu g}\tnm{f_1}{v}\tnm{f_2}{v},
\end{align}
and
\begin{align}
    \abs{\int_{\r^3}\th\mhh\q\left[\mh f_1,\mh f_2;\m\right]g\,\ud v}&\ls\vnm{f_1}{v}\tnm{f_2}{v}\vnm{g}{v},\\
    \abs{\int_{\r^3}\th\mhh\q\left[\mh f_1,\mh f_2;\m\right]g\,\ud v}&\ls \sup_{v}\abs{\nu^3g}\tnm{f_1}{v}\tnm{f_2}{v},\\
    \tnm{\th\mhh\q\left[\mh f_1,\mh f_2;\m\right]g}{v}&\ls\sup_{v}\abs{\nu g}\tnm{f_1}{v}\tnm{f_2}{v}.
\end{align}

{\it \underline{Step\;2-1}. Estimate of cubic terms for \eqref{temp 3}}.
We then focus on $\th\mhh\q\left[\mh f_1,\mh f_2;\mh f_3\right]$. Recalling \eqref{temp 2}, we have
\begin{align}
    &\abs{\th\mhh\q\left[\mh f_1,\mh f_2;\mh f_3\right]}\\
    \ls\,&\,\ue^{\frac{1}{2}\abs{v}^2}\int_{\r^3}\int_{\s^2}q(\o,\abs{v-u})\ue^{-\frac{1}{2}\abs{u'}^2-\frac{1}{2}\abs{v'}^2}\Big(\abs{f_1(u')f_2(v')}+\abs{f_1(v' )f_2(u')}\Big)\Big(\ue^{-\frac{1}{2}\abs{u}^2}\abs{f_3(u )}+\ue^{-\frac{1}{2}\abs{v}^2}\abs{f_3(v)}\Big)\ud\o\ud u\no\\
    &+\ue^{\frac{1}{2}\abs{v}^2}\int_{\r^3}\int_{\s^2}q(\o,\abs{v-u})\ue^{-\frac{1}{2}\abs{u}^2-\frac{1}{2}\abs{v}^2}\Big(\abs{f_1(u)f_2(v)}+\abs{f_1(v )f_2(u)}\Big)\Big(\ue^{-\frac{1}{2}\abs{u'}^2}\abs{f_3(u' )}+\ue^{-\frac{1}{2}\abs{v'}^2}\abs{f_3(v')}\Big)\ud\o\ud u\no\\
    \ls\,&\int_{\r^3}\int_{\s^2}q(\o,\abs{v-u})\ue^{-\frac{1}{2}\abs{u}^2}\Big(\abs{f_1(u')f_2(v')}+\abs{f_1(v' )f_2(u')}\Big)\Big(\ue^{-\frac{1}{2}\abs{u}^2}\abs{f_3(u )}+\ue^{-\frac{1}{2}\abs{v}^2}\abs{f_3(v)}\Big)\ud\o\ud u\no\\
    &+\int_{\r^3}\int_{\s^2}q(\o,\abs{v-u})\ue^{-\frac{1}{2}\abs{u}^2}\Big(\abs{f_1(u)f_2(v)}+\abs{f_1(v )f_2(u)}\Big)\Big(\ue^{-\frac{1}{2}\abs{u'}^2}\abs{f_3(u' )}+\ue^{-\frac{1}{2}\abs{v'}^2}\abs{f_3(v')}\Big)\ud\o\ud u\no\\
    =\,&\, J_1+J_2.\no
\end{align}
For both parts, we have the naive bound
\begin{align}
    q(\o,\abs{v-u})\ue^{-\frac{1}{2}\abs{u}^2}\ls \nu(v).
\end{align}
Also, noticing that $q(\o,\abs{v-u})=\abs{\o\big(\o\cdot(v-u)\big)}=\abs{u-u'}$, we have
\begin{align}
    q(\o,\abs{v-u})\ue^{-\frac{1}{2}\abs{u}^2}\ls \nu(u').
\end{align}
For $J_1$, we split
\begin{align}
    &\abs{\int_{\r^3}J_1(v)g(v)\,\ud v}\\
    \ls\,&\int_{\r^3}\int_{\r^3}\int_{\s^2} q(\o,\abs{v-u})\ue^{-\frac{1}{2}\abs{u}^2}\Big(\abs{f_1(u')f_2(v')}+\abs{f_1(v' )f_2(u')}\Big)\Big(\ue^{-\frac{1}{2}\abs{u}^2}\abs{f_3(u )}\Big)|g(v)|\,\ud\o\ud u\ud v\no\\
    &+\int_{\r^3}\int_{\r^3}\int_{\s^2}
    q(\o,\abs{v-u})\ue^{-\frac{1}{2}\abs{u}^2}\Big(\abs{f_1(u')f_2(v')}+\abs{f_1(v' )f_2(u')}\Big)\Big(\ue^{-\frac{1}{2}\abs{v}^2}\abs{f_3(v)}\Big)|g(v)|\,\ud\o\ud u\ud v\no\\
    =\,&\, I_{11}+I_{12}.\no
\end{align}
We may directly use Cauchy's inequality to bound $I_{11}$, 
\begin{align}
    I_{11}\ls\,&\bigg(\int_{\r^3}\int_{\r^3}\int_{\s^2}\Big(f_1^2(u')f_2^2(v')+f_1^2(v')f_2^2(u')\Big)\ud\o\ud u\ud v\bigg)^{\frac{1}{2}}\\
    &\times\bigg(\int_{\r^3}\int_{\r^3}\int_{\s^2}\nu^2(v)\ue^{-\abs{u}^2}f_3^2(u)g^2(v)\,\ud\o\ud u\ud v\bigg)^{\frac{1}{2}}\no\\
    \ls\,&\tnm{g}{v}\tnm{f_1}{v}\tnm{f_2}{v}\tnm{f_3}{v}.\no
\end{align}
The estimate of $I_{12}$ is a bit complicated due to $g(v)f_3(v)$ term. 
We may bound it in two different ways
\begin{align}
    I_{12}\ls\,&\bigg(\int_{\r^3}\int_{\r^3}\int_{\s^2}\nu(u')\Big(f_1^2(u')f_2^2(v')+f_1^2(v')f_2^2(u')\Big)\ud\o\ud u\ud v\bigg)^{\frac{1}{2}}\\
    &\times\bigg(\int_{\r^3}\int_{\r^3}\int_{\s^2}\abs{v-u}\ue^{-\frac{\abs{u}^2}{2}}\ue^{-|v|^2}f_3^2(v) g^2(v)\,\ud\o\ud u\ud v\bigg)^{\frac{1}{2}}\no\\
    \ls\,&\vnm{g}{v}\Big(\vnm{f_1}{v}\tnm{f_2}{v}+\tnm{f_1}{v}\vnm{f_2}{v}\Big)\sup_v\abs{f_3}.\no
\end{align}
and
\begin{align}
    I_{12}\ls\,&\sup_v\abs{f_1}\sup_v\abs{f_2}\bigg(\int_{\r^3}\int_{\r^3}\int_{\s^2}\abs{v-u}g^2(v)\ue^{-\frac{1}{2}\abs{u}^2}\ue^{-\frac{1}{2}\abs{v}^2}\ud\o\ud u\ud v\bigg)^{\frac{1}{2}}\\
    &\times\bigg(\int_{\r^3}\int_{\r^3}\int_{\s^2}\abs{v-u}f_3^2(v)\ue^{-\frac{1}{2}\abs{u}^2}\ue^{-\frac{1}{2}\abs{v}^2}\ud\o\ud u\ud v\bigg)^{\frac{1}{2}}\no\\
    \ls\,&\tnm{g}{v}\sup_v\abs{f_1}\sup_v\abs{f_2}\tnm{f_3}{v}.\no
\end{align}
Hence, we know
\begin{align}
    I_{12}\ls \vnm{g}{v}\min\bigg\{\Big(\vnm{f_1}{v}\tnm{f_2}{v}+\tnm{f_1}{v}\vnm{f_2}{v}\Big)\sup_v\abs{f_3},\sup_v\abs{f_1}\sup_v\abs{f_2}\tnm{f_3}{v}\bigg\}.
\end{align}
In total, we have
\begin{align}
    \abs{\int_{\r^3}J_1(v)g(v)\,\ud v}\ls\,&\tnm{g}{v}\tnm{f_1}{v}\tnm{f_2}{v}\tnm{f_3}{v}\\ &+\vnm{g}{v}\min\bigg\{\Big(\vnm{f_1}{v}\tnm{f_2}{v}+\tnm{f_1}{v}\vnm{f_2}{v}\Big)\sup_v\abs{f_3},\sup_v\abs{f_1}\sup_v\abs{f_2}\tnm{f_3}{v}\bigg\}.\no
\end{align}
For $J_2$, 
we may use two different ways to bound it:
\begin{align}
    \abs{\int_{\r^3}J_2(v)g(v)\,\ud v}\ls\,&\sup_v\abs{ f_3}\bigg(\int_{\r^3}\int_{\r^3}\int_{\s^2}\abs{v-u}g^2(v)\ue^{-\frac{1}{2}\abs{u}^2}\ud\o\ud u\ud v\bigg)^{\frac{1}{2}}\\
    &\times\bigg(\int_{\r^3}\int_{\r^3}\int_{\s^2}\nu(v)f_i^2(v)f_j^2(u)\,\ud\o\ud u\ud v\bigg)^{\frac{1}{2}}\no\\
    \ls\,&\vnm{g}{v}\Big(\vnm{f_1}{v}\tnm{f_2}{v}+\tnm{f_1}{v}\vnm{f_2}{v}\Big)\sup_v\abs{f_3},\no
\end{align}
and using the fact that $\abs{q(\o,\abs{v-u})}=\abs{v-v'}$,
\begin{align}
    \abs{\int_{\r^3}J_2(v)g(v)\,\ud v}\ls\,&\bigg(\int_{\r^3}\int_{\r^3}\int_{\s^2}\nu(v)g^2(v)f_i^2(u)\,\ud\o\ud u\ud v\bigg)^{\frac{1}{2}}\\
    &\times\bigg(\int_{\r^3}\int_{\r^3}\int_{\s^2}\nu(u')f_j^2(v)\ue^{-\abs{u'}^2}f_3^2(u')\,\ud\o\ud u\ud v\bigg)^{\frac{1}{2}}\no\\
    &+\bigg(\int_{\r^3}\int_{\r^3}\int_{\s^2}\nu(v)g^2(v)f_i^2(u)\,\ud\o\ud u\ud v\bigg)^{\frac{1}{2}}\no\\
    &\times\bigg(\int_{\r^3}\int_{\r^3}\int_{\s^2}\abs{v-v'}\ue^{-\frac{1}{2}\abs{u}^2}f_j^2(v)\ue^{-\abs{v'}^2}f^2_3(v')\,\ud\o\ud u\ud v\bigg)^{\frac{1}{2}}\no\\
    =&\;I_{21}+I_{22}.\no
\end{align}
Using Lemma \ref{prelim-lemma:derivative} with substitution $u\rt u'$, we may obtain
\begin{align}
    I_{21}\ls \vnm{g}{v}\tnm{f_i}{v}\tnm{f_j}{v}\tnm{f_3}{v}.
\end{align}
Using Lemma \ref{prelim-lemma:derivative} with substitution $v\rt v'$, we may obtain
\begin{align}
    I_{22}
    \ls \sup_{v}\abs{\nu^{\frac{1}{2}}f_j}\vnm{g}{v}\tnm{f_i}{v}\tnm{f_3}{v}.
\end{align}
In total, we know
\begin{align}
    \abs{\int_{\r^3}J_2(v)g(v)\,\ud v}
    \ls\,&\vnm{g}{v}\Big(\tnm{f_1}{v}\tnm{f_2}{v}+\sup_v\abs{\nu^{\frac{1}{2}}f_1}\tnm{f_2}{v}+\tnm{f_1}{v}\sup_v\abs{\nu^{\frac{1}{2}}f_2}\Big)\tnm{f_3}{v}.\no
\end{align}
Hence, we know
\begin{align}
    \abs{\int_{\r^3}J_2(v)g(v)\,\ud v}\ls&\vnm{g}{v}\min\bigg\{\Big(\vnm{f_1}{v}\tnm{f_2}{v}+\tnm{f_1}{v}\vnm{f_2}{v}\Big)\sup_v\abs{f_3},\\
    &\Big(\tnm{f_1}{v}\tnm{f_2}{v}+\sup_v\abs{\nu^{\frac{1}{2}}f_1}\tnm{f_2}{v}+\tnm{f_1}{v}\sup_v\abs{\nu^{\frac{1}{2}}f_2}\Big)\tnm{f_3}{v}\bigg\}.\no
\end{align}
Summarizing all above, we know
\begin{align}
    &\abs{\int_{\r^3}\th\mhh\q\left[\mh f_1,\mh f_2;\mh f_3\right]g\,\ud v}\\
    \ls\,&\tnm{g}{v}\tnm{f_1}{v}\tnm{f_2}{v}\tnm{f_3}{v}\no\\ \,&+\vnm{g}{v}\min\bigg\{\Big(\vnm{f_1}{v}\tnm{f_2}{v}+\tnm{f_1}{v}\vnm{f_2}{v}\Big)\sup_v\abs{f_3},\sup_v\abs{f_1}\sup_v\abs{f_2}\vnm{f_3}{v}\bigg\}\no\\
    \,&+\vnm{g}{v}\min\bigg\{\Big(\vnm{f_1}{v}\tnm{f_2}{v}+\tnm{f_1}{v}\vnm{f_2}{v}\Big)\sup_v\abs{f_3},\Big(\sup_v\abs{\nu^{\frac{1}{2}}f_1}\tnm{f_2}{v}+\tnm{f_1}{v}\sup_v\abs{\nu^{\frac{1}{2}}f_2}\Big)\tnm{f_3}{v}\bigg\}.\no
\end{align}

{\it \underline{Step\;2-2}. Estimate of cubic terms for \eqref{temp 4}}.
On the other hand, similar to the estimates in Step 2-1, if we take supremum over $v$ on $g$, we have
\begin{align}
    \abs{\int_{\r^3}J_1(v)g(v)\,\ud v}\ls\,&\,\Big(\vnm{g}{v}+\sup_{v}\abs{\nu g}\Big)\tnm{f_1}{v}\tnm{f_2}{v}\tnm{f_3}{v},
\end{align}
and
\begin{align}
    \abs{\int_{\r^3}J_2(v)g(v)\,\ud v}\ls\,\Big(\vnm{g}{v}+\sup_{v}\abs{\nu g}\Big)\tnm{f_1}{v}\tnm{f_2}{v}\tnm{f_3}{v},\no
\end{align}
Hence, we have
\begin{align}
    \int_{\r^3}\th\mhh\q\left[\mh f_1,\mh f_2;\mh g\right]\ud v\ls\, \Big(\vnm{g}{v}+\sup_{v}\abs{\nu g}\Big)\tnm{f_1}{v}\tnm{f_2}{v}\tnm{f_3}{v}.
\end{align}

{\it \underline{Step\;2-3}. Estimate of cubic terms for \eqref{temp 5}}.
Also, in a similar fashion, for any $h\in L^2_{v}$, we have
\begin{align}
    &\abs{\int_{\r^3}J_1(v)g(v)h(v)\,\ud v}\\
    \ls&\,\sup_{v}\abs{\nu g}\tnm{h}{v}\bigg(\tnm{f_1}{v}\tnm{f_2}{v}\tnm{f_3}{v}+\min\Big\{\sup_{v}\abs{f_3}\tnm{f_1}{v}\tnm{f_2}{v}, \sup_{v}\abs{f_1}\sup_{v}\abs{f_2}\tnm{f_3}{v}\Big\}\bigg),\no
\end{align}
and
\begin{align}
    &\abs{\int_{\r^3}J_2(v)g(v)h(v)\,\ud v}\\
    \ls\,&\sup_{v}\abs{\nu g}\tnm{h}{v}\bigg(\min\Big\{\sup_{v}\abs{f_2}\tnm{f_1}{v}\tnm{f_3}{v}, \sup_{v}\abs{f_3}\tnm{f_1}{v}\tnm{f_2}{v} \Big\}\no\\
    \,&+\min\Big\{\sup_{v}\abs{f_1}\tnm{f_2}{v}\tnm{f_3}{v}, \sup_{v}\abs{f_3}\tnm{f_1}{v}\tnm{f_2}{v} \Big\}\bigg)\no.
\end{align}
Therefore, due to duality of $L^2$, we have
\begin{align}
    &\tnm{\th\mhh\q\left[\mh f_1,\mh f_2;\mh g\right]}{v}\\
    \ls\,& \sup_{v}\abs{\nu g}\bigg(\tnm{f_1}{v}\tnm{f_2}{v}\tnm{f_3}{x,v}+\min\Big\{\sup_{v}\abs{f_3}\tnm{f_1}{v}\tnm{f_2}{v}, \sup_{v}\abs{f_1}\sup_{v}\abs{f_2}\tnm{f_3}{v}\Big\}\no\\
    &\,+\min\Big\{ \sup_{v}\abs{f_3}\tnm{f_1}{v}\tnm{f_2}{v},\sup_{v}\abs{f_2}\tnm{f_1}{v}\tnm{f_3}{v} \Big\}\no\\
    &\,+\min\Big\{\sup_{v}\abs{f_3}\tnm{f_1}{v}\tnm{f_2}{v},\sup_{v}\abs{f_1}\tnm{f_2}{v}\tnm{f_3}{v}\Big\}\bigg)\no.
\end{align}

{\it \underline{Step\;3}. More estimate of quadratic terms}.
By a similar argument, we can handle \,$\th\mhh\q\left[\mh f_1,\m;\mh f_3\right]$ and \,$\th\mhh\q\left[\m,\mh f_2;\mh f_3\right]$. For \eqref{temp 3}, we obtain
\begin{align}
    \,&\abs{\int_{\r^3}\th\mhh\q\left[\mh f_1,\m;\mh f_3\right]g\,\ud v}\\
    \ls\,&\vnm{g}{v}\tnm{f_1}{v}\vnm{f_3}{v}+\vnm{g}{v}\min\Big\{\vnm{f_1}{v}\sup_v\abs{f_3},\sup_v\abs{\nu^{\frac{1}{2}}f_1}\tnm{f_3}{v}\Big\},\no
\end{align}
and
\begin{align}
    \,&\abs{\int_{\r^3}\th\mhh\q\left[\m,\mh f_2;\mh f_3\right]g\,\ud v}\\
    \ls\,&\vnm{g}{v}\tnm{f_2}{v}\vnm{f_3}{v}+\vnm{g}{v}\min\Big\{\vnm{f_2}{v}\sup_v\abs{f_3},\sup_v\abs{\nu^{\frac{1}{2}}f_2}\tnm{f_3}{v}\Big\}.\no
\end{align}
In addition, for \eqref{temp 4},
\begin{align}
    \abs{\int_{\r^3}\th\mhh\q\left[\mh f_1,\m;\mh f_3\right]g\,\ud v}\ls\, \Big(\vnm{g}{v}+\sup_{v}\abs{\nu g}\Big)\tnm{f_1}{v}\tnm{f_3}{v},
\end{align}
and
\begin{align}
    \abs{\int_{\r^3}\th\mhh\q\left[\m,\mh f_2;\mh f_3\right]g\,\ud v}\ls\, \Big(\vnm{g}{v}+\sup_{v}\abs{\nu g}\Big)\tnm{f_2}{v}\tnm{f_3}{v}.
\end{align}
Also, for \eqref{temp 5},
\begin{align}
    &\tnm{\th\mhh\q\left[\mh f_1,\m;\mh f_3\right]g}{v}\\
    \ls&\, \sup_{v}\abs{\nu g}\tnm{f_1}{v}\tnm{f_3}{v}+\sup_{v}\abs{\nu g}\min\Big\{\tnm{f_1}{v}\sup_v\abs{f_3},\sup_v\abs{f_1}\tnm{f_3}{v}\Big\},\no
\end{align}
and
\begin{align}
    &\tnm{\th\mhh\q\left[\m,\mh f_2;\mh f_3\right]g}{v}\\
    \ls&\, \sup_{v}\abs{\nu g}\tnm{f_2}{v}\tnm{f_3}{v}+\sup_{v}\abs{\nu g}\min\Big\{\tnm{f_2}{v}\sup_v\abs{f_3},\sup_v\abs{f_2}\tnm{f_3}{v}\Big\}.
\end{align}
\end{proof}

\begin{remark}
The nonlinear estimate for the quantum case is much more complicated than the classical Boltzmann version. In particular, unlike the classical Boltzmann equation, here we need \,$\sup_v$ estimate via Sobolev embedding, which implies that we have to consider velocity derivatives.
\end{remark}

\begin{lemma}\label{prelim-lemma: nonlinear-derivative}
Let $f_i$ for $i=1,2,3$, and $g$ be smooth functions. Then we have
\begin{align}
    &\abs{\int_{\r^3}\p_{\beta}^{\gamma}\g[f_1,f_2;f_3]g\,\ud v}\\
    \ls\,&\tnm{g}{v}\tnm{f_1'}{v}\tnm{f_2'}{v}\tnm{f_3'}{v}+\vnm{g}{v}\tnm{f_a'}{v}\vnm{f_b'}{v}\no\\ \,&+\vnm{g}{v}\min\bigg\{\Big(\vnm{f_1'}{v}\tnm{f_2'}{v}+\tnm{f_1'}{v}\vnm{f_2'}{v}\Big)\sup_v\abs{f_3'},\sup_v\abs{f_1'}\sup_v\abs{f_2'}\vnm{f_3'}{v}\bigg\}\no\\
    \,&+\vnm{g}{v}\min\bigg\{\Big(\vnm{f_1'}{v}\tnm{f_2'}{v}+\tnm{f_1'}{v}\vnm{f_2'}{v}\Big)\sup_v\abs{f_3'},\Big(\sup_v\abs{\nu^{\frac{1}{2}}f_1'}\tnm{f_2'}{v}+\tnm{f_1'}{v}\sup_v\abs{\nu^{\frac{1}{2}}f_2'}\Big)\tnm{f_3'}{v}\bigg\}\no\\
    \,&+\vnm{g}{v}\min\Big\{\vnm{f_1'}{v}\sup_v\abs{f_3'},\sup_v\abs{\nu^{\frac{1}{2}}f_1'}\tnm{f_3'}{v}\Big\}+\vnm{g}{v}\min\Big\{\vnm{f_2'}{v}\sup_v\abs{f_3'},\sup_v\abs{\nu^{\frac{1}{2}}f_2'}\tnm{f_3'}{v}\Big\}.\no
\end{align}
Thus, we have
\begin{align}
    \abs{\int_{\r^3}\p_{\beta}^{\gamma}\g[f_1,f_2;f_3]g\,\ud v}\ls\sup_{v}\abs{\nu^3g}\tnm{f_a'}{v}\tnm{f_b'}{v}+\Big(\vnm{g}{v}+\sup_{v}\abs{\nu g}\Big)\tnm{f_i'}{v}\tnm{f_j'}{v}\tnm{f_k'}{v},\no
\end{align}
and
\begin{align}
    \tnm{\p_{\beta}^{\gamma}\g[f_1,f_2;f_3]g}{v}\ls&\,\sup_{v}\abs{\nu g}\bigg(\tnm{f_1'}{v}\tnm{f_2'}{v}\tnm{f_3'}{v}+\tnm{f_a'}{v}\tnm{f_b'}{v}\\
    &+\min\Big\{\sup_{v}\abs{f_3'}\tnm{f_1'}{v}\tnm{f_2'}{v}, \sup_{v}\abs{f_1'}\sup_{v}\abs{f_2'}\tnm{f_3'}{v}\Big\}\no\\
    &+\min\Big\{ \sup_{v}\abs{f_3'}\tnm{f_1'}{v}\tnm{f_2'}{v},\sup_{v}\abs{f_2'}\tnm{f_1'}{v}\tnm{f_3'}{v} \Big\}\no\\
    &+\min\Big\{\sup_{v}\abs{f_3'}\tnm{f_1'}{v}\tnm{f_2'}{v},\sup_{v}\abs{f_1'}\tnm{f_2'}{v}\tnm{f_3'}{v}\Big\}\bigg)\no.
\end{align}
Here $(a,b)$ runs all combinations of $\{1,2,3\}$. For $s=1,2,3$, $f_s':=\p_{\beta_s}^{\gamma_s}f_s$ where $\abs{\beta_1}+\abs{\beta_2}+\abs{\beta_3}=\abs{\beta}$ and $\abs{\gamma_1}+\abs{\gamma_2}+\abs{\gamma_3}=\abs{\gamma}$.
\end{lemma}

\begin{proof}
The proof is similar to that of Lemma \ref{prelim-lemma: nonlinear-no-derivative} and Lemma \ref{prelim-lemma: nu-derivative-estimate}. Note that the spatial and velocity derivatives will be distributed among the three arguments in $\Gamma$, but it will not change the fundamental structure here.
\end{proof}


\smallskip
\subsection{Local Solutions}




\begin{theorem}[Local Well-posedness] \label{local-theorem}
There exists $M>0$ and $T(M)>0$ such that if 
\begin{align}
    \ee[f_0]\leq \frac{M}{2},
\end{align}
then there exists a unique solution $f(t,x,v)$ to the quantum Boltzmann equation \eqref{equation: prelim-perturbation-2} such that
\begin{align}
    \ee[f(t)]\leq M,
\end{align}
for any $t\in[0,T]$. Moreover, the energy $\ee[f(t)]$ is continuous over $t\in[0,T]$. Furthermore, 
\begin{itemize}
    \item 
    for fermions $\th=-1$, if $0\leq F_0(x,v)\leq 1$, then $0\leq F(t,x,v)\leq 1$;
    \item
    for bosons $\th=1$, if $F_0(x,v)\geq 0$, then $F(t,x,v)\geq 0$.
\end{itemize}
\end{theorem}

\begin{proof}

In the following, we mainly study $Q[F,F;F]$, so we rearrange the terms as
\begin{align}\label{prelim-decompostion: Q}
    Q[F,F;F]=Q_p[F,F;F]-Q_q[F,F;F],
\end{align}
where
\begin{align}
    Q_p[F,G;H]:=&\int_{\r^3}\int_{\s^2}q(\o,\abs{v-u})F(u' )G(v')\big(1+\th H(u)\big)\ud\o\ud u,\\
    Q_q[F,G;H]:=&\int_{\r^3}\int_{\s^2}q(\o,\abs{v-u})\Big[-\th F(u' )G(v') H(v)+F(u)H(v)\big(1+\th G(u')+\th G(v')\big)\Big]\ud\o\ud u.
\end{align}
$Q_q[F,F;F]$ contains all terms that depends on $F(v)$, and $Q_p[F,F;F]$ contains all the other terms. 
(Here the $F,G,H$ is the not the same as \eqref{temp 1}).
We may further rewrite
\begin{align}
    Q_q[F,G;H]=H\tilde Q_p[F,G],
\end{align}
with
\begin{align}
    \tilde Q_p[F,G]=\int_{\r^3}\int_{\s^2}q(\o,\abs{v-u})\Big[-\th F(u' )G(v') +F(u)\big(1+\th G(u')+\th G(v')\big)\Big]\ud\o\ud u.
\end{align}
Similar to the decomposition in \eqref{prelim-decompostion: Q}, we may also decompose $\g$ as
\begin{align}
    \g[f,g;h]=\g_p[f,g;h]-\g_q[f,g;h],
\end{align}
where $\g_p$ comes from the linearization of $Q_p$, and $\g_q$ from $Q_p$.

\smallskip
{\it \underline{Step\;1}. Boundedness}.
Define the iteration sequence via
\begin{align}\label{local-equation: iteration.}
    \dt F^{n+1}+v\cdot\nx F^{n+1}+Q_q[F^n,F^n;F^{n+1}]=Q_p[F^n,F^n;F^n],
\end{align}
with
\begin{align}
    F^{n+1}(0,x,v)=F_0(x,v).
\end{align}
This is equivalent to the perturbation form
\begin{align}\label{local-equation: iteration}
    \dt f^{n+1}+v\cdot\nx f^{n+1}+\nu f^{n+1}=K[f^n]+\g_p[f^n,f^n;f^{n}]-\g_q[f^n,f^n;f^{n+1}],
\end{align}
with
\begin{align}
    f^{n+1}(0,x,v)=f_0(x,v).
\end{align}
The iteration starts with $f^0(t,x,v)=f_0(x,v)$. Taking $\p_{\beta}^{\gamma}$ on both sides of \eqref{local-equation: iteration}, we obtain
\begin{align}\label{local-equation: iteration-derivative}
    \dt \Big(\p_{\beta}^{\gamma}f^{n+1}\Big)+v\cdot\nx \Big(\p_{\beta}^{\gamma}f^{n+1}\Big)+\p_{\beta}^{\gamma}\Big(\nu f^{n+1}\Big)=\p_{\beta}^{\gamma}K[f^n]+\p_{\beta}^{\gamma}\g_p[f^n,f^n;f^{n}]-\p_{\beta}^{\gamma}\g_q[f^n,f^n;f^{n+1}].
\end{align}
Multiplying $\p_{\beta}^{\gamma}f^{n+1}$ on both sides of \eqref{local-equation: iteration-derivative} and integrating over $\Omega\times\r^3$, we obtain
\begin{align}\label{local-estimate:temp}
    &\frac{1}{2}\frac{\ud}{\ud t}\tnm{\p_{\beta}^{\gamma}f^{n+1}}{x,v}^2+\iint_{\Omega\times\r^3}\Big(\p_{\beta}^{\gamma}f^{n+1}\Big)\cdot\p_{\beta}^{\gamma}\Big(\nu f^{n+1}\Big)\\
    =&\iint_{\Omega\times\r^3}\Big(\p_{\beta}^{\gamma}f^{n+1}\Big)\cdot\p_{\beta}^{\gamma}K[f^n]+\iint_{\Omega\times\r^3}\Big(\p_{\beta}^{\gamma}f^{n+1}\Big)\cdot\p_{\beta}^{\gamma}\g_p[f^n,f^n;f^{n}]-\iint_{\Omega\times\r^3}\Big(\p_{\beta}^{\gamma}f^{n+1}\Big)\cdot\p_{\beta}^{\gamma}\g_q[f^n,f^n;f^{n+1}].\no
\end{align}
Let $\eta>0$ be sufficiently small and $C_{\eta}>0$. For the second term on the LHS of \eqref{local-estimate:temp}, using Lemma \ref{prelim-lemma: nu-estimate} and Lemma \ref{prelim-lemma: nu-derivative-estimate}, we have
\begin{align}
    &\iint_{\Omega\times\r^3}\Big(\p_{\beta}^{\gamma}f^{n+1}\Big)\cdot\p_{\beta}^{\gamma}\Big(\nu f^{n+1}\Big)\\
    \gs\,&\iint_{\Omega\times\r^3}\Big(\p_{\beta}^{\gamma}f^{n+1}\Big)\cdot\Big(\nu \p_{\beta}^{\gamma}f^{n+1}\Big)-\abs{\sum_{\beta'<\beta}\iint_{\Omega\times\r^3}\Big(\p_{\beta}^{\gamma}f^{n+1}\Big)\cdot\p_{\beta'}^{\gamma}\Big(\nu f^{n+1}\Big)}\no\\
    \gs\,&\vnm{\p_{\beta}^{\gamma}f^{n+1}}{x,v}^2-\eta\tnm{\p_{\beta}^{\gamma}f^{n+1}}{x,v}^2-C_{\eta}\sum_{\beta'<\beta}\tnm{\p_{\beta'}^{\gamma}f^{n+1}}{x,v}^2,\no
\end{align}
For the first term on the RHS of \eqref{local-estimate:temp}, using Lemma \ref{prelim-lemma: K-derivative}, we have
\begin{align}
    \iint_{\Omega\times\r^3}\Big(\p_{\beta}^{\gamma}f^{n+1}\Big)\cdot\p_{\beta}^{\gamma}K[f^n]\,\ls\, \eta\tnm{\p_{\beta}^{\gamma}f^{n+1}}{x,v}^2+C_{\eta}\sum_{\beta'\leq\beta}\tnm{\p_{\beta'}^{\gamma}f^{n}}{x,v}^2.
\end{align}
For the second term on the RHS of \eqref{local-estimate:temp}, using the first inequality in Lemma \ref{prelim-lemma: nonlinear-derivative} with $\sup_v$ falling on the term with the lowest $v$ derivative, combining with Sobolev embedding theorem $H^2_v\hookrightarrow L^{\infty}_v$, we obtain
\begin{align}
    &\iint_{\Omega\times\r^3}\Big(\p_{\beta}^{\gamma}f^{n+1}\Big)\cdot\p_{\beta}^{\gamma}\g_p[f^n,f^n;f^{n}]
    \,\ls\, \vnm{\p_{\beta}^{\gamma}f^{n+1}}{x,v} \sum_{\beta'\leq\beta}\vnm{\p_{\beta'}^{\gamma}f^{n}}{x,v}\cdot\sum_{\beta'\leq\beta}\tnm{\p_{\beta'}^{\gamma}f^{n}}{x,v}^2,
\end{align}
and
\begin{align}
    &\iint_{\Omega\times\r^3}\Big(\p_{\beta}^{\gamma}f^{n+1}\Big)\cdot\p_{\beta}^{\gamma}\g_q[f^n,f^n;f^{n+1}]\\
    \ls&\, \vnm{\p_{\beta}^{\gamma}f^{n+1}}{x,v}\left(\sum_{\beta'\leq\beta}\vnm{\p_{\beta'}^{\gamma}f^{n+1}}{x,v}\cdot\sum_{\beta'\leq\beta}\tnm{\p_{\beta'}^{\gamma}f^{n}}{x,v}^2+\sum_{\beta'\leq\beta}\tnm{\p_{\beta'}^{\gamma}f^{n+1}}{x,v}\cdot\sum_{\beta'\leq\beta}\tnm{\p_{\beta'}^{\gamma}f^{n}}{x,v}\cdot\sum_{\beta'\leq\beta}\vnm{\p_{\beta'}^{\gamma}f^{n}}{x,v}\right).\no
\end{align}
Summarizing all above and running over $\abs{\gamma}+\abs{\beta}\leq N$, we have
\begin{align}\label{temp 6}
    \frac{\ud}{\ud t}\nnm{f^{n+1}}^2+\vnnm{f^{n+1}}^2\,\ls&\, \nnm{f^{n+1}}^2+\vnnm{f^{n}}\nnm{f^{n}}^2\vnnm{f^{n+1}}\\
    &+\nnm{f^{n}}^2+\nnm{f^{n}}^2\vnnm{f^{n+1}}^2+\nnm{f^{n}}^2+\vnnm{f^{n}}\nnm{f^{n}}\vnnm{f^{n+1}}\nnm{f^{n+1}}.\no
\end{align}
Then integrating over $[0,t]$ and using the definition of $\ee$, with the help of Cauchy's inequality, we have
\begin{align}
    \ee[f^{n+1}(t)]\,\ls\,\ee[f_0]+t\sup_{s\in[0,t]}\ee[f^{n+1}(s)]+t\sup_{s\in[0,t]}\ee[f^{n}(s)]+\sup_{s\in[0,t]}\ee[f^{n}(s)]\sup_{s\in[0,t]}\ee[f^{n+1}(s)].
\end{align}
Assume that $\ds\sup_{s\in[0,T]}\ee[f^{n}(s)]\leq M$ and $\ee[f_0]\leq \dfrac{M}{2}$. Then taking supremum over $s\in[0,T]$, we have
\begin{align}
    \sup_{s\in[0,T]}\ee[f^{n+1}(s)]\,\ls\,\dfrac{M}{2}+T\sup_{s\in[0,T]}\ee[f^{n+1}(s)]+TM+M\sup_{s\in[0,T]}\ee[f^{n+1}(s)].
\end{align}
Hence, for $M$ sufficiently small and $T(M)$ sufficiently small, we know 
\begin{align}
    \sup_{s\in[0,T]}\ee[f^{n+1}(s)]\leq M.
\end{align}
Therefore, we know this iteration is uniformly bounded.

\smallskip
{\it \underline{Step\;2}. Contraction}.
On the other hand, taking the difference of the equations for $f^{n+1}$ and $f^n$, we obtain
\begin{align}\label{local-equation: difference}
    &\dt \big(f^{n+1}-f^n\big)+v\cdot\nx \big(f^{n+1}-f^n\big)+\nu \big(f^{n+1}-f^n\big)\\
    =\,&K[f^n-f^{n-1}]+\g_p[f^n-f^{n-1},f^{n};f^{n}]+\g_p[f^{n-1},f^n-f^{n-1};f^{n-1}]+\g_p[f^{n-1},f^{n-1};f^{n}-f^{n-1}]\no\\
    \,&-\g_q[f^n-f^{n-1},f^n;f^{n+1}]-\g_q[f^{n-1},f^n-f^{n-1};f^{n+1}]-\g_q[f^{n-1},f^{n-1};f^{n+1}-f^n],\no
\end{align}
with
\begin{align}
    \big(f^{n+1}-f^n\big)(0,x,v)=0.
\end{align}
Similarly to the above argument, we first take $\p_{\beta}^{\gamma}$, multiply $\p_{\beta}^{\gamma}\big(f^{n+1}-f^n\big)$ on both sides of \eqref{local-equation: difference}, and integrate over $\Omega\times\r^3$. 
Using similar techniques as in proving boundedness, we obtain
\begin{align}
    \sup_{s\in[0,T]}\ee\big[\big(f^{n+1}-f^n\big)\big](s)\,\ls\sup_{s\in[0,T]}\ee\big[\big(f^{n}-f^{n-1}\big)\big](s),
\end{align}
where we use the nonlinear estimates
\begin{align}
    &\iint_{\Omega\times\r^3}\p_{\beta}^{\gamma}\big(f^{n+1}-f^n\big)\cdot\p_{\beta}^{\gamma}\g_p[f^n-f^{n-1},f^n;f^{n}]\\
    \ls\,& \Big(\sup_{s\in[0,T]}\ee\big[\big(f^{n+1}-f^n\big)\big](s)\Big)^{\frac{1}{2}}\Big(\sup_{s\in[0,T]}\ee\big[\big(f^{n}-f^{n-1}\big)\big](s)\Big)^{\frac{1}{2}}\Big(\sup_{s\in[0,T]}\ee[f^n](s)\Big)^{\frac{1}{2}}\Big(\sup_{s\in[0,T]}\ee[f^{n}](s)\Big)^{\frac{1}{2}}\no\\
    \ls\,& M\Big(\sup_{s\in[0,T]}\ee\big[\big(f^{n+1}-f^n\big)\big](s)\Big)^{\frac{1}{2}}\Big(\sup_{s\in[0,T]}\ee\big[\big(f^{n}-f^{n-1}\big)\big](s)\Big)^{\frac{1}{2}},\no
\end{align}
\begin{align}
    &\iint_{\Omega\times\r^3}\p_{\beta}^{\gamma}\big(f^{n+1}-f^n\big)\cdot\p_{\beta}^{\gamma}\g_p[f^{n-1},f^n-f^{n-1};f^{n}]\\
    \ls\,& \Big(\sup_{s\in[0,T]}\ee\big[\big(f^{n+1}-f^n\big)\big](s)\Big)^{\frac{1}{2}}\Big(\sup_{s\in[0,T]}\ee\big[\big(f^{n}-f^{n-1}\big)\big](s)\Big)^{\frac{1}{2}}\Big(\sup_{s\in[0,T]}\ee[f^{n-1}](s)\Big)^{\frac{1}{2}}\Big(\sup_{s\in[0,T]}\ee[f^{n}](s)\Big)^{\frac{1}{2}}\no\\
    \ls\,& M\Big(\sup_{s\in[0,T]}\ee\big[\big(f^{n+1}-f^n\big)\big](s)\Big)^{\frac{1}{2}}\Big(\sup_{s\in[0,T]}\ee\big[\big(f^{n}-f^{n-1}\big)\big](s)\Big)^{\frac{1}{2}},\no
\end{align}
\begin{align}
    &\iint_{\Omega\times\r^3}\p_{\beta}^{\gamma}\big(f^{n+1}-f^n\big)\cdot\p_{\beta}^{\gamma}\g_p[f^{n-1},f^{n-1};f^{n}-f^{n-1}]\\
    \ls\,& \Big(\sup_{s\in[0,T]}\ee\big[\big(f^{n+1}-f^n\big)\big](s)\Big)^{\frac{1}{2}}\Big(\sup_{s\in[0,T]}\ee\big[\big(f^{n}-f^{n-1}\big)\big](s)\Big)^{\frac{1}{2}}\Big(\sup_{s\in[0,T]}\ee[f^{n-1}](s)\Big)\no\\
    \ls\,& M\Big(\sup_{s\in[0,T]}\ee\big[\big(f^{n+1}-f^n\big)\big](s)\Big)^{\frac{1}{2}}\Big(\sup_{s\in[0,T]}\ee\big[\big(f^{n}-f^{n-1}\big)\big](s)\Big)^{\frac{1}{2}},\no
\end{align}
and
\begin{align}
    &\iint_{\Omega\times\r^3}\p_{\beta}^{\gamma}\big(f^{n+1}-f^n\big)\cdot\p_{\beta}^{\gamma}\g_q[f^n-f^{n-1},f^n;f^{n+1}]\\
    \ls\,& \Big(\sup_{s\in[0,T]}\ee\big[\big(f^{n}-f^{n-1}\big)\big](s)\Big)^{\frac{1}{2}}\Big(\sup_{s\in[0,T]}\ee\big[\big(f^{n+1}-f^{n}\big)\big](s)\Big)^{\frac{1}{2}}\Big(\sup_{s\in[0,T]}\ee[f^n](s)\Big)^{\frac{1}{2}}\Big(\sup_{s\in[0,T]}\ee[f^{n+1}](s)\Big)^{\frac{1}{2}}\no\\
    \ls\,& M\Big(\sup_{s\in[0,T]}\ee\big[\big(f^{n}-f^{n-1}\big)\big](s)\Big)^{\frac{1}{2}}\Big(\sup_{s\in[0,T]}\ee\big[\big(f^{n+1}-f^{n}\big)\big](s)\Big)^{\frac{1}{2}},\no
\end{align}
\begin{align}
    &\iint_{\Omega\times\r^3}\p_{\beta}^{\gamma}\big(f^{n+1}-f^n\big)\cdot\p_{\beta}^{\gamma}\g_q[f^{n-1},f^n-f^{n-1};f^{n+1}]\\
    \ls\,& \Big(\sup_{s\in[0,T]}\ee\big[\big(f^{n}-f^{n-1}\big)\big](s)\Big)^{\frac{1}{2}}\Big(\sup_{s\in[0,T]}\ee\big[\big(f^{n+1}-f^{n}\big)\big](s)\Big)^{\frac{1}{2}}\Big(\sup_{s\in[0,T]}\ee[f^{n-1}](s)\Big)^{\frac{1}{2}}\Big(\sup_{s\in[0,T]}\ee[f^{n+1}](s)\Big)^{\frac{1}{2}}\no\\
    \ls\,& M\Big(\sup_{s\in[0,T]}\ee\big[\big(f^{n}-f^{n-1}\big)\big](s)\Big)^{\frac{1}{2}}\Big(\sup_{s\in[0,T]}\ee\big[\big(f^{n+1}-f^{n}\big)\big](s)\Big)^{\frac{1}{2}},\no
\end{align}
\begin{align}
    &\iint_{\Omega\times\r^3}\p_{\beta}^{\gamma}\big(f^{n+1}-f^n\big)\cdot\p_{\beta}^{\gamma}\g_q[f^{n-1},f^{n-1};f^{n+1}-f^n]\\
    \ls\,& \sup_{s\in[0,T]}\ee\big[\big(f^{n+1}-f^{n}\big)\big](s)\cdot\sup_{s\in[0,T]}\ee[f^{n-1}](s)\no\\
    \ls\,& M\sup_{s\in[0,T]}\ee\big[\big(f^{n}-f^{n-1}\big)\big](s).\no
\end{align}
Hence, we know the iteration is a contraction, and thus
defines a (uniform) Cauchy sequence. 
Therefore, we know that there exists a classical solution $f$. The uniqueness follows naturally from the contraction proof and the Gronwall's inequality. The inequality \eqref{temp 6} also justifies the continuity of $\ee[f(t)]$ with respect to $t$. 

The positivity of $F$ follows from a standard induction. Our iteration is actually 
\begin{align}
    \dt F^{n+1}+v\cdot\nx F^{n+1}+\tilde Q_p[F^n,F^n]F^{n+1}=Q_q[F^n,F^n;F^n].
\end{align}
We may verify that for $\th=\pm 1$, $F^n$ satisfies the positivity estimate
\begin{align}
    Q_q[F^n,F^n;F^n]\geq0.
\end{align}
In addition, for $\th=-1$, $F^n$ satisfies the positivity estimate
\begin{align}
\\
    \tilde Q_p[F^n,F^n]=\int_{\r^3}\int_{\s^2}q(\o,\abs{v-u})\Big[-\th F^n(u' )F^n(v')\big(1+\th F^n(u)\big) +F^n(u)\big(1+\th F^n(u')\big)\big(1+\th F^n(v')\big)\Big]\ud\o\ud u\geq0.\no
\end{align}
Hence, by solving the ODE for $F^{n+1}$, we may derive the positivity naturally.
\end{proof}


\smallskip
\subsection{Global Solutions for $\Omega=\t^3$}

In this section, we will justify the global well-posedness and decay for $\Omega=\t^3$ case. 

\subsubsection{Positivity Estimate for $\l$}

\begin{lemma}\label{global-lemma:spatial-derivative}
Assume $f(t,x,v)$ satisfies \eqref{equation: prelim-perturbation} for $t\in[0,T]$ with $T\geq 1$. Also, $f(t,x,v)$ satisfies $\ds\sup_{t\in[0,T]}\ee[f(t)]\leq M$. Then for $0\leq s<t\leq T$, we have
\begin{align}
    \sum_{\abs{\gamma}\leq N}\bigg(\tnm{\p^{\gamma}f(t)}{x,v}^2+\int_s^t\vnm{\p^{\gamma}f(\tau)}{x,v}^2\ud\tau\bigg)\leq \ue^{C(t-s)} \sum_{\abs{\gamma}\leq N}\tnm{\p^{\gamma}f(s)}{x,v}^2,
\end{align}
and
\begin{align}
    \sum_{\abs{\gamma}\leq N}\int_s^t\vnm{\p^{\gamma}f(\tau)}{x,v}^2\ud\tau\geq \Big(1-\ue^{C(t-s)}\Big) \sum_{\abs{\gamma}\leq N}\tnm{\p^{\gamma}f(s)}{x,v}^2.
\end{align}
\end{lemma}

\begin{proof}
Similar to the proof of Theorem \ref{local-theorem}, applying $\p^{\gamma}$ on both sides of the equation \eqref{equation: prelim-perturbation}, multiplying $\p^{\gamma}f$, and integrating over $\Omega\times\r^3$, we have
\begin{align}\label{local-estimate: derivative}
    &\frac{1}{2}\frac{\ud}{\ud t}\tnm{\p^{\gamma}f}{x,v}^2+\iint_{\Omega\times\r^3}\p^{\gamma}f\cdot \big(\nu\p^{\gamma}f\big)
    =\iint_{\Omega\times\r^3}\p^{\gamma}f\cdot K\big[\p^{\gamma}f\big]+\iint_{\Omega\times\r^3}\p^{\gamma}f\cdot\p^{\gamma}\g[f,f;f].
\end{align}
Then using Lemma \ref{prelim-lemma: nu-estimate}, Lemma \ref{prelim-lemma: K-compactness} and Lemma \ref{prelim-lemma: nonlinear-no-derivative}, and summing over $\abs{\gamma}\leq N$, we have
\begin{align}
    \frac{\ud}{\ud t}\bigg(\sum_{\abs{\gamma}\leq N}\tnm{\p^{\gamma}f}{x,v}^2\bigg)+\bigg(\sum_{\abs{\gamma}\leq N}\vnm{\p^{\gamma}f}{x,v}^2\bigg)\ls \bigg(\sum_{\abs{\gamma}\leq N}\tnm{\p^{\gamma}f}{x,v}^2\bigg)+M\bigg(\sum_{\abs{\gamma}\leq N}\tnm{\p^{\gamma}f}{x,v}^2\bigg).
\end{align}
When $M$ is small, we may absorb the last term into the LHS to obtain
\begin{align}
    \frac{\ud}{\ud t}\bigg(\sum_{\abs{\gamma}\leq N}\tnm{\p^{\gamma}f}{x,v}^2\bigg)+\bigg(\sum_{\abs{\gamma}\leq N}\vnm{\p^{\gamma}f}{x,v}^2\bigg)\ls \bigg(\sum_{\abs{\gamma}\leq N}\tnm{\p^{\gamma}f}{x,v}^2\bigg).
\end{align}
Then by Gronwall's inequality, we obtain
\begin{align}
    \sum_{\abs{\gamma}\leq N}\tnm{\p^{\gamma}f(t)}{x,v}^2\leq \ue^{C(t-s)}\sum_{\abs{\gamma}\leq N}\tnm{\p^{\gamma}f(s)}{x,v}^2.
\end{align}
Then we integrate over $\tau\in[s,t]$ to obtain
\begin{align}
    \sum_{\abs{\gamma}\leq N}\int_s^t\vnm{\p^{\gamma}f(\tau)}{x,v}^2\ud\tau\leq \ue^{C(t-s)}\sum_{\abs{\gamma}\leq N}\tnm{\p^{\gamma}f(s)}{x,v}^2.
\end{align}
This justifies the first inequality in the lemma. 

On the other hand, we may rearrange the terms in \eqref{local-estimate: derivative}
\begin{align}
    &\frac{1}{2}\frac{\ud}{\ud t}\tnm{\p^{\gamma}f}{x,v}^2
    =-\iint_{\Omega\times\r^3}\p^{\gamma}f\cdot \big(\nu\p^{\gamma}f\big)+\iint_{\Omega\times\r^3}\p^{\gamma}f\cdot K\big[\p^{\gamma}f\big]+\iint_{\Omega\times\r^3}\p^{\gamma}f\cdot\p^{\gamma}\g[f,f;f].
\end{align}
Similar to the above argument, we have
\begin{align}
    \frac{\ud}{\ud t}\bigg(\sum_{\abs{\gamma}\leq N}\tnm{\p^{\gamma}f}{x,v}^2\bigg)\gs -\bigg(\sum_{\abs{\gamma}\leq N}\vnm{\p^{\gamma}f}{x,v}^2\bigg)-\bigg(\sum_{\abs{\gamma}\leq N}\tnm{\p^{\gamma}f}{x,v}^2\bigg)-M\bigg(\sum_{\abs{\gamma}\leq N}\tnm{\p^{\gamma}f}{x,v}^2\bigg),
\end{align}
which yields
\begin{align}
    \frac{\ud}{\ud t}\bigg(\sum_{\abs{\gamma}\leq N}\tnm{\p^{\gamma}f}{x,v}^2\bigg)\gs -\bigg(\sum_{\abs{\gamma}\leq N}\vnm{\p^{\gamma}f}{x,v}^2\bigg),
\end{align}
Integrating over $\tau\in[s,t]$, we have
\begin{align}
    \sum_{\abs{\gamma}\leq N}\vnm{\p^{\gamma}f(t)}{x,v}^2\gs \sum_{\abs{\gamma}\leq N}\tnm{\p^{\gamma}f(s)}{x,v}^2-\sum_{\abs{\gamma}\leq N}\int_s^t\vnm{\p^{\gamma}f(\tau)}{x,v}^2\ud\tau.
\end{align}
Then by Gronwall's inequality, we obtain
\begin{align}
    \sum_{\abs{\gamma}\leq N}\int_s^t\vnm{\p^{\gamma}f(\tau)}{x,v}^2\ud\tau\geq \Big(1-\ue^{C(t-s)}\Big) \sum_{\abs{\gamma}\leq N}\tnm{\p^{\gamma}f(s)}{x,v}^2.
\end{align}
This justifies the second inequality in the lemma.
\end{proof}

\begin{lemma}\label{global-lemma:positivity}
Assume $f(t,x,v)$ satisfies \eqref{equation: prelim-perturbation} for $t\in[0,T]$ with $T\geq 1$. Assume the initial data $f_0$ satisfies the conservation laws. Also, $f(t,x,v)$ satisfies $\sup_{t\in[0,T]}\ee[f(t)]\leq M$. Then there exists a constant $\d_M\in(0,1)$ such that
\begin{align}
    \sum_{\abs{\gamma}\leq N}\int_0^1\int_{\t^3}\int_{\r^3}\l\big[\p^{\gamma}f(t)\big]\cdot\p^{\gamma}f(t)\,\ud v\ud x\ud t\geq \d_M\sum_{\abs{\gamma}\leq N}\int_0^1\vnm{\p^{\gamma}f(t)}{x,v}^2\ud t.
\end{align}
\end{lemma}

\begin{proof}
We prove by contradiction. If the result is not true, then there exists a sequence of solutions $\big\{f_n(t,x,v)\big\}_{n=1}^{\infty}$ to \eqref{equation: prelim-perturbation} such that
\begin{align}
    \sup_{t\in[0,1]}\sum_{\abs{\gamma}+\abs{\beta}\leq N}\tnm{\p_{\beta}^{\gamma}f_n(t)}{x,v}^2\leq M,
\end{align}
and
\begin{align}
    0\leq \sum_{\abs{\gamma}\leq N}\int_0^1\int_{\t^3}\int_{\r^3}\l\big[\p^{\gamma}f_n(t)\big]\cdot\p^{\gamma}f_n(t)\leq \frac{1}{n}\sum_{\abs{\gamma}\leq N}\int_0^1\vnm{\p^{\gamma}f_n(t)}{x,v}^2.
\end{align}
We normalize
\begin{align}
    Z_n(t,x,v)=\frac{f_n(t,x,v)}{\sqrt{\sum_{\abs{\gamma}\leq N}\int_0^1\vnm{\p^{\gamma}f_n(t)}{x,v}^2\ud t}}.
\end{align}
Then noticing that $\l=\nu I-K$, we know
\begin{align}\label{global-estimate: K-estimate}
    0\leq 1-\sum_{\abs{\gamma}\leq N}\int_0^1\int_{\t^3}\int_{\r^3}K\big[\p^{\gamma}Z_n(t)\big]\cdot\p^{\gamma}Z_n(t)\leq \frac{1}{n}.
\end{align}
On the other hand, by Lemma \ref{global-lemma:spatial-derivative}, we know
\begin{align}
    \sum_{\abs{\gamma}\leq N}\tnm{\p^{\gamma}f_n(t)}{x,v}^2\ls \sum_{\abs{\gamma}\leq N}\tnm{\p^{\gamma}f_n(0)}{x,v}^2,
\end{align}
and
\begin{align}
    \sum_{\abs{\gamma}\leq N}\int_0^1\vnm{\p^{\gamma}f_n(t)}{x,v}^2\gs \sum_{\abs{\gamma}\leq N}\tnm{\p^{\gamma}f_n(0)}{x,v}^2.
\end{align}
Then based on the definition of $Z_n$, we know that
\begin{align}
    \sup_{t\in[0,1]}\sum_{\abs{\gamma}\leq N}\tnm{\p^{\gamma}Z_n(t)}{x,v}^2\ls1.
\end{align}
Based on the equation \eqref{equation: prelim-perturbation}, we know $Z_n$ satisfies
\begin{align}
    \dt Z_n+v\cdot\nx Z_n+\l[Z_n]=\g[f_n,f_n;Z_n].
\end{align}
After applying $\p^{\gamma}$ on both sides, we obtain
\begin{align}\label{global-equation: Zn}
    \dt \big(\p^{\gamma}Z_n\big)+v\cdot\nx \big(\p^{\gamma}Z_n\big)+\l\big[\p^{\gamma}Z_n\big]=\p^{\gamma}\g[f_n,f_n;Z_n].
\end{align}
Also, we have the conservation laws
\begin{align}
    \int_{\Omega}\int_{\r^3} Z_n(t,x,v)\mh(v)\,\ud v\ud x&=0,\quad (\text{Mass})\\
    \int_{\Omega}\int_{\r^3}Z_n(t,x,v)\mh(v)v_i\,\ud v\ud x&=0,\quad (\text{Momentum})\\
    \int_{\Omega}\int_{\r^3}Z_n(t,x,v)\mh(v)\abs{v}^2\ud v\ud x&=0.\quad (\text{Energy})
\end{align}
Due to the boundedness, we can extract weakly convergent subsequence in $L^2([0,1]\times\t^3\times\r^3)$ to get
\begin{align}
    \p^{\gamma}Z_n\rightharpoonup \p^{\gamma}Z,\quad \p^{\gamma}f_n\rightharpoonup \p^{\gamma}f,
\end{align}
where $Z$ and $f$ are the limit functions, respectively.\\

\smallskip
{\it \underline{Step\;1}. $K[\p^{\gamma}Z_n]\rt K[\p^{\gamma}Z]$ in $L^2([0,1]\times\t^3\times\r^3)$ for $\abs{\gamma}\leq N$}.

Based on Lemma \ref{prelim-lemma: K-compactness}, we know $K$ is a bounded operator in $L^2(\t^3\times\r^3)$. Then for any $\e>0$, we have
\begin{align}
    \int_0^{\e}\tnm{K[\p^{\gamma}Z_n]}{x,v}^2+\int_{1-\e}^{1}\tnm{K[\p^{\gamma}Z_n]}{x,v}^2\ls\e.
\end{align}
Hence, there is no time concentration in a neighborhood of $t=0$ or $t=1$. Then it suffices to consider $K[\p^{\gamma}Z_n]\rt K[\p^{\gamma}Z]$ in $L^2([\e,1-\e]\times\t^3\times\r^3)$.

Notice that
\begin{align}
    K[\p^{\gamma}Z_n]=\int_{\r^3}k(u,v)\p^{\gamma}Z_n(u)\ud u.
\end{align}
Due to the proof of Lemma \ref{prelim-lemma: K-compactness}, we know that though we cannot write $k(u,v)$ explicitly, it actually can be pointwise bounded by the corresponding operator $\tilde k(u,v)$ for the classical Boltzmann equation as in \cite[Section 3.2\&3.3]{Glassey1996}. Hence, \cite[Lemma 3.5.1]{Glassey1996} justifies that for any $\e>0$, there exists $m>0$ such that we may split
\begin{align}
    k(u,v)=k_m(u,v)+\big(k(u,v)-k_m(u,v)\big),
\end{align}
where
\begin{align}
    k_m(u,v)=k(u,v)\id_{\{(u,v):\abs{u-v}\geq\frac{1}{m}\ \text{and}\ \abs{v}\leq m\}},
\end{align}
satisfying
\begin{align}
    \int_0^1\tnm{\int_{\r^3}(k-k_m)\p^{\gamma}Z_n(t)\ud u}{x,v}\leq \e\int_0^1\tnm{\p^{\gamma}Z_n(t)}{x,v}\ls\e.
\end{align}
Then naturally $k_m\in L^2(\r^3\times\r^3)$. Based on the density lemma, we may find a smooth function $\k_{\e}(u,v)=\k_1(u)\k_2(v)$ with compact support satisfying
\begin{align}
    \tnm{k_m-\k_{\e}}{u,v}\ls\e.
\end{align}
Hence, we know
\begin{align}
    \int_0^1\tnm{\int_{\r^3}(k_m-\k_{\e})\p^{\gamma}Z_n(t)\ud u}{x,v}\leq \tnm{k_m-\k_{\e}}{u,v}\int_0^1\tnm{\p^{\gamma}Z_n(t)}{x,v}\ls\e.
\end{align}
Then it suffices to justify
\begin{align}
    \int_{\r^3}\k_1(u)\p^{\gamma}Z_n(t)\ud u \rt \int_{\r^3}\k_1(u)\p^{\gamma}Z(t)\ud u
\end{align}
in $L^2([\e,1-\e]\times\t^3)$ since we can later multiply $\k_2(v)$ and integrating over $v\in\r^3$ to complete the proof.

Let $\chi(t,x)$ be a smooth cutoff function in $(0,1)\times\r^3$ such that $\chi=1$ in $[\e,1-\e]\times\t^3$. Multiplying $\k_1(v)\chi(t,x)$ on both sides of \eqref{global-equation: Zn} to obtain
\begin{align}
\\
    \dt \big(\k_1\chi\p^{\gamma}Z_n\big)+v\cdot\nx \big(\k_1\chi\p^{\gamma}Z_n\big)=-\k_1\chi\l\big[\p^{\gamma}Z_n\big]+\k_1\chi\p^{\gamma}\g[f_n,f_n;Z_n]+\p^{\gamma}Z_n\big(\dt+v\cdot\nx\big)(\k_1\chi).\no
\end{align}
Due to compact support in $(t,x,v)$ variables of $\k_1\chi$, we know
\begin{align}
    \int_0^1\tnm{\k_1\chi\l\big[\p^{\gamma}Z_n\big]}{x,v}^2\ls \int_0^1\tnm{\p^{\gamma}Z_n}{x,v}^2\ls 1,
\end{align}
and
\begin{align}
    \int_0^1\tnm{\p^{\gamma}Z_n\big(\dt+v\cdot\nx\big)(\k_1\chi)}{x,v}^2\ls \int_0^1\tnm{\p^{\gamma}Z_n}{x,v}^2\ls 1.
\end{align}
Then we are left with the nonlinear term. Based on the third inequality in Lemma \ref{prelim-lemma: nonlinear-derivative} and our assumption of the lemma, we always put $L^2_{v,x}$ on the term with highest-order derivative and use Sobolev embedding to handle the supremum in $(x,v)$, to obtain
\begin{align}
    \int_0^1\tnm{\k_1\chi\p^{\gamma}\g[f_n,f_n;Z_n]}{x,v}^2\ls M\int_0^1\Big(\tnm{\p^{\gamma}Z_n}{x,v}^2+\tnm{\p^{\gamma}f_n}{x,v}^2\Big)\ls 1.
\end{align}
In total, we know that
\begin{align}
    \dt \big(\k_1\chi\p^{\gamma}Z_n\big)+v\cdot\nx \big(\k_1\chi\p^{\gamma}Z_n\big)\in L^2([0,1]\times\r^3\times\r^3).
\end{align}
Then from the averaging lemma, we obtain
\begin{align}
    \int_{\r^3}\k_1\chi\p^{\gamma}Z_n(t,x,u)\ud u\in H^{\frac{1}{4}}([0,1]\times\r^3).
\end{align}
Then by the compact embedding, we may extract a weakly convergent subsequence in $H^{\frac{1}{4}}$, which is a strongly convergent subsequence in $L^2$ such that
\begin{align}
    \int_{\r^3}\k_1\chi\p^{\gamma}Z_n(t,x,u)\ud u\rt \int_{\r^3}\k_1\chi\p^{\gamma}Z(t,x,u)\ud u.
\end{align}
Hence, our result naturally follows.\\

\smallskip
{\it \underline{Step\;2}. $Z(t,x,v)=a(t,x)\mh+b(t,x)\cdot v\mh+c(t,x)\abs{v}^2\mh$}.

From Step 1, we know that 
\begin{align}
    \int_0^1\int_{\t^3}\int_{\r^3}K\big[\p^{\gamma}Z_n\big]\cdot\p^{\gamma}Z_n \rt \int_0^1\int_{\t^3}\int_{\r^3}K\big[\p^{\gamma}Z\big]\cdot\p^{\gamma}Z.
\end{align}
Hence, taking limit $n\rt\infty$ in \eqref{global-estimate: K-estimate}, we obtain
\begin{align}
0\leq 1-\sum_{\abs{\gamma}\leq N}\int_0^1\int_{\t^3}\int_{\r^3}K\big[\p^{\gamma}Z\big]\cdot\p^{\gamma}Z\leq0.
\end{align}
Hence, we must have
\begin{align}
    \sum_{\abs{\gamma}\leq N}\int_0^1\int_{\t^3}\int_{\r^3}K\big[\p^{\gamma}Z\big]\cdot\p^{\gamma}Z=1.
\end{align}
On the other hand, the lower semi-continuity of $\nu$-norm implies
\begin{align}
    \sum_{\abs{\gamma}\leq N}\int_0^1\vnm{\p^{\gamma}Z(t)}{x,v}^2\leq 1.
\end{align}
Therefore, we know
\begin{align}
    0\leq& \sum_{\abs{\gamma}\leq N}\int_0^1\int_{\t^3}\int_{\r^3}\l\big[\p^{\gamma}Z\big]\cdot\p^{\gamma}Z\\
    =&\sum_{\abs{\gamma}\leq N}\int_0^1\vnm{\p^{\gamma}Z(t)}{x,v}^2-\sum_{\abs{\gamma}\leq N}\int_0^1\int_{\t^3}\int_{\r^3}K\big[\p^{\gamma}Z\big]\cdot\p^{\gamma}Z\no\\
    \leq&\, 1-1=0.\no
\end{align}
Therefore, we have
\begin{align}\label{global-estimate:temp 1}
    \sum_{\abs{\gamma}\leq N}\int_0^1\vnm{\p^{\gamma}Z(t)}{x,v}^2=1,
\end{align}
and
\begin{align}
    \sum_{\abs{\gamma}\leq N}\int_0^1\int_{\t^3}\int_{\r^3}\l\big[\p^{\gamma}Z\big]\cdot\p^{\gamma}Z=0.
\end{align}
In particular, the weak convergence and norm convergence imply strong convergence, i.e. $\p^{\gamma}Z_n\rt \p^{\gamma}Z$ in $L^2([0,1]\times\r^3\times\r^3)$.

Hence, we know $Z$ belongs to the null space of $\l$, i.e. 
\begin{align}\label{global-estimate:temp 2}
    Z(t,x,v)=a(t,x)\mh+b(t,x)\cdot v\mh+c(t,x)\abs{v}^2\mh,
\end{align}
where $a,b,c$ are given by $Z$. In particular, the boundedness of $Z$ implies
\begin{align}
    \sup_{t\in[0,1]}\Big(\tnm{\p^{\gamma}a(t)}{x}^2+\tnm{\p^{\gamma}b(t)}{x}^2+\tnm{\p^{\gamma}c(t)}{x}^2\Big)\ls 1.
\end{align}
Then taking limit $n\rt\infty$ in \eqref{global-equation: Zn}, we know that in the sense of distribution
\begin{align}\label{global-estimate:temp 3}
    \dt\big(\p^{\gamma}Z\big)+v\cdot\nx \big(\p^{\gamma}Z\big)=\p^{\gamma}\g[f,f;Z].
\end{align}
Also, we have the conservation laws
\begin{align}
    \int_{\Omega}\int_{\r^3} Z(t,x,v)\mh(v)\,\ud v\ud x&=0,\quad (\text{Mass})\\
    \int_{\Omega}\int_{\r^3}Z(t,x,v)\mh(v)v_i\,\ud v\ud x&=0,\quad (\text{Momentum})\\
    \int_{\Omega}\int_{\r^3}Z(t,x,v)\mh(v)\abs{v}^2\ud v\ud x&=0.\quad (\text{Energy})
\end{align}

\smallskip
{\it \underline{Step\;3}. $\sum_{\abs{\gamma}\leq N}\int_0^1\vnm{\p^{\gamma}Z(t)}{x,v}^2\ls M$}.

If this is justified, then it contradicts \eqref{global-estimate:temp 1} and we conclude our proof.
Plugging \eqref{global-estimate:temp 2} into \eqref{global-estimate:temp 3}, we obtain that
\begin{align}
    \big(\nx\p^{\gamma}c\big)\cdot v\abs{v}^2\mh+\big(\dt \p^{\gamma}c\abs{v}^2+v\cdot\nx(v\cdot\p^{\gamma})\big)\mh\\
    +\big(\dt\p^{\gamma}b+\nx\p^{\gamma}a\big)\cdot v\mh+\big(\dt\p^{\gamma}a\big)\mh&=\p^{\gamma}\g[f,f;Z].\no
\end{align}
Since 
\begin{align}
    v_i\abs{v}^2\mh,\quad v_iv_j\mh,\quad v_i\mh,\quad \mh
\end{align}
are linearly independent, in the sense of distribution, their coefficients on both sides of the equation should be equal, i.e. the so-called macroscopic equations
\begin{align}
    \p_{x_i}\p^{\gamma}c&=h_{ci}^{\gamma},\label{global-equation:macro1}\\
    \dt \p^{\gamma}c+\p_{x_i}\p^{\gamma}b_i&=h^{\gamma}_i,\label{global-equation:macro2}\\
    \p_{x_i}\p^{\gamma}b_j+\p_{x_j}\p^{\gamma}b_i&=h^{\gamma}_{ij}\ \text{for}\ i\neq j,\label{global-equation:macro3}\\
    \dt\p^{\gamma}b_i+\p_{x_i}\p^{\gamma}a&=h^{\gamma}_{bi},\label{global-equation:macro4}\\
    \dt\p^{\gamma}a&=h^{\gamma}_a,\label{global-equation:macro5}
\end{align}
where $h_{ci}^{\gamma}$, $h^{\gamma}_i$, $h^{\gamma}_{ij}$, $h^{\gamma}_{bi}$ and $h^{\gamma}_a$ are linear combinations of
\begin{align}
\\
    \int_{\r^3}\p^{\gamma}\g[f,f;Z]v_i\abs{v}^2\mh,\quad \int_{\r^3}\p^{\gamma}\g[f,f;Z]v_iv_j\mh,\quad \int_{\r^3}\p^{\gamma}\g[f,f;Z]v_i\mh,\quad \int_{\r^3}\p^{\gamma}\g[f,f;Z]\mh.\no
\end{align}
In particular, based on the second inequality of Lemma \ref{prelim-lemma: nonlinear-derivative}, we know
\begin{align}
    \sup_{t\in[0,1]}\Big(\tnm{h_{ci}^{\gamma}}{x}+\tnm{h^{\gamma}_i}{x}+\tnm{h^{\gamma}_{ij}}{x}+\tnm{h^{\gamma}_{bi}}{x}+\tnm{h^{\gamma}_a}{x}\Big)\ls M.
\end{align}
From \eqref{global-equation:macro1}, we know
\begin{align}
    \tnm{\nx\p^{\gamma}c}{x}\ls \tnm{h_{ci}^{\gamma}}{x}\ls M.
\end{align}
For $b$, we directly compute from \eqref{global-equation:macro2} and \eqref{global-equation:macro3}
\begin{align}
    \Delta_x\p^{\gamma}b_i&=\sum_{i\neq j}\p_{x_jx_j}\p^{\gamma}b_i+\p_{x_ix_i}\p^{\gamma}b_i\\
    &=\sum_{i\neq j}\Big(-\p_{x_ix_j}\p^{\gamma}b_j+\p_{x_j}h^{\gamma}_{ij}\Big)+\Big(-\dt \p_{x_i}\p^{\gamma}c+\p_{x_i}h^{\gamma}_i\Big)\no\\
    &=\sum_{i\neq j}\Big(\dt\p_{x_i}\p^{\gamma}c-\p_{x_i}h^{\gamma}_{j}\Big)+\Big(\sum_{i\neq j}\p_{x_j}h^{\gamma}_{ij}\Big)+\Big(-\dt \p_{x_i}\p^{\gamma}c+\p_{x_i}h^{\gamma}_i\Big)\no\\
    &=\dt\p_{x_i}\p^{\gamma}c+\sum_{i\neq j}\Big(\p_{x_j}h^{\gamma}_{ij}-\p_{x_i}h^{\gamma}_{j}\Big)+\p_{x_i}h^{\gamma}_i\no\\
    &=\Big(-\p_{x_ix_i}\p^{\gamma}b_i+\p_{x_i}h^{\gamma}_i\Big)+\sum_{i\neq j}\Big(\p_{x_j}h^{\gamma}_{ij}-\p_{x_i}h^{\gamma}_{j}\Big)+\p_{x_i}h^{\gamma}_i\no\\
    &=-\p_{x_ix_i}\p^{\gamma}b_i+\sum_{i\neq j}\Big(\p_{x_j}h^{\gamma}_{ij}-\p_{x_i}h^{\gamma}_{j}\Big)+2\p_{x_i}h^{\gamma}_i.\no
\end{align}
Then multiplying $\p^{\gamma}b_i$ in the above equation and integrating by parts, we obtain
\begin{align}
    \tnm{\nx\p^{\gamma}b_i}{x}\ls \tnm{h^{\gamma}_{ij}}{x}+\tnm{h^{\gamma}_{i}}{x}\ls M.
\end{align}
We assume $t>\dfrac{1}{2}$ and focus on $[0,t]$ (otherwise, we can focus on $[t,1]$). We integrate \eqref{global-equation:macro4} over $[0,t]$ to obtain that for $0\leq\abs{\gamma}\leq N-1$
\begin{align}
    \p^{\gamma}b_i(t)-\p^{\gamma}b_i(0)+\int_0^t\p_{x_i}\p^{\gamma}a(s)\ud s=\int_0^t h_{bi}^{\gamma}(s)\ud s.
\end{align}
Then since $\dt\p_{x_i}\p^{\gamma}a=\p_{x_i}h_a^{\gamma}$, we have
\begin{align}
    \p_{x_i}\p^{\gamma}a(s)=\p_{x_i}\p^{\gamma}a(t)+\int_t^s\p_{x_i}h_a^{\gamma}(\tau)\ud\tau.
\end{align}
Then plug this into the above equation, we have
\begin{align}
    \p_{x_i}\p^{\gamma}a(t)&=-\frac{1}{t}\Big(\p^{\gamma}b_i(t)-\p^{\gamma}b_i(0)\Big)-\frac{1}{t}\int_0^t\int_t^s\p_{x_i}h_a^{\gamma}(\tau)\ud\tau\ud s+\frac{1}{t}\int_0^t h_{bi}^{\gamma}(s)\ud s.
\end{align}
Then taking $x_i$ derivative, we obtain
\begin{align}
    \p_{x_ix_i}\p^{\gamma}a(t)&=-\frac{1}{t}\Big(\p_{x_i}\p^{\gamma}b_i(t)-\p_{x_i}\p^{\gamma}b_i(0)\Big)-\frac{1}{t}\int_0^t\int_t^s\p_{x_ix_i}h_a^{\gamma}(\tau)\ud\tau\ud s+\frac{1}{t}\int_0^t \p_{x_i}h_{bi}^{\gamma}(s)\ud s.
\end{align}
Then multiplying $\p^{\gamma}a$ in the above equation and integrating by parts, we obtain
\begin{align}
    \tnm{\nx\p^{\gamma}a}{x}\ls \sup_{t\in[0,1]}\tnm{b}{x}+\tnm{\nx h^{\gamma}_{a}}{x}+\tnm{h_{bi}^{\gamma}}{x}\ls M.
\end{align}
In total, we have proved 
\begin{align}
    \sum_{0<\abs{\alpha}\leq N}\int_0^1\Big(\tnm{\p^{\gamma}a}{x}^2+\tnm{\p^{\gamma}b}{x}^2+\tnm{\p^{\gamma}c}{x}^2\Big)\ls M.
\end{align}
Let $\abs{\gamma}=0$ in \eqref{global-equation:macro1}-\eqref{global-equation:macro5}, we know 
\begin{align}
    \int_0^1\Big(\tnm{\dt a}{x}^2+\tnm{\dt b}{x}^2+\tnm{\dt c}{x}^2\Big)\ls M.
\end{align}
Therefore, applying Poincar\'e's inequality in $[0,t]\times\t^3$, we have
\begin{align}
    &\int_0^1\Big(\tnm{a}{x}^2+\tnm{ b}{x}^2+\tnm{c}{x}^2\Big)\\
    \ls&\int_0^1\Big(\tnm{\nabla_{t,x}a}{x}^2+\tnm{\nabla_{t,x}b}{x}^2+\tnm{\nabla_{t,x}c}{x}^2\Big)+\bigg(\abs{\int_0^t\int_{\r^3}a}+\abs{\int_0^t\int_{\r^3}b}+\abs{\int_0^t\int_{\r^3}c}\bigg)\no\\
    \ls&\, M+\bigg(\abs{\int_0^t\int_{\r^3}a}+\abs{\int_0^t\int_{\r^3}b}+\abs{\int_0^t\int_{\r^3}c}\bigg).\no
\end{align}
The conservation law for $Z$ implies 
\begin{align}
    \int_0^t\int_{\r^3}a=0,\quad \int_0^t\int_{\r^3}b=0,\quad \int_0^t\int_{\r^3}c=0.
\end{align}
Hence, we have
\begin{align}
    \int_0^1\Big(\tnm{a}{x}^2+\tnm{ b}{x}^2+\tnm{c}{x}^2\Big)\ls M.
\end{align}
This concludes our proof.
\end{proof}

\begin{remark}
This proof highly relies on Poincar\'e's inequality, so it cannot be naturally extended to $\Omega=\r^3$ case.
\end{remark}

\begin{lemma}\label{global-lemma:multi-positivity}
Assume $f(t,x,v)$ satisfies \eqref{equation: prelim-perturbation} for $t\in[0,T]$ with $T\geq 1$. Assume the initial data $f_0$ satisfies the conservation laws. Also, $f(t,x,v)$ satisfies $\sup_{t\in[0,T]}\ee[f(t)]\leq M$. Then there exists a constant $\d_M\in(0,1)$ such that for any $t'\geq0$ and a positive integer $n$ with $t'+n\in[0,T]$, 
\begin{align}
    \sum_{\abs{\gamma}\leq N}\int_{t'}^{t'+n}\int_{\t^3}\int_{\r^3}\l\big[\p^{\gamma}f(t)\big]\cdot\p^{\gamma}f(t)\ud v\ud x\ud t\geq \d_M\sum_{\abs{\gamma}\leq N}\int_{t'}^{t'+n}\vnm{\p^{\gamma}f(t)}{x,v}^2\ud t.
\end{align}
\end{lemma}

\begin{proof}
We apply Lemma \ref{global-lemma:positivity} to each of the intervals $[t',t'+1]$, $[t'+1,t'+2]$, $\cdots$, $[t'+(n-1),t'+n]$ and then sum them up.
\end{proof}

\begin{remark}
It is not very easy to further extend the result to intervals with arbitrary length. In particular, if we take $f_0$ satisfying $\nnpk[f_0]=0$, then for a short period of time, we know the results in Lemma \ref{global-lemma:positivity} cannot be true. Hence, the lower bound of interval length is very important.
\end{remark}


\subsubsection{Global Well-Posedness and Time Decay}

\begin{theorem}\label{global-theorem: well-posedness}
There exists $M_0>0$ such that if 
\begin{align}
    \ee[f_0]\leq \frac{M_0}{2},
\end{align}
then there exists a unique solution $f(t,x,v)$ to the quantum Boltzmann equation \eqref{equation: prelim-perturbation} such that
\begin{align}
    \ee[f(t)]\leq M_0,
\end{align}
for any $t\in[0,\infty)$. 
\end{theorem}

\begin{proof}
We first choose the initial data $\ee[f_0]\leq \dfrac{M}{2}$.
Denote
\begin{align}
    T=\sup\Big\{t\geq0: \sup_{s\in[0,t]} \ee[f(s)]\leq M\Big\}.
\end{align}
Based on Theorem \ref{local-theorem} for the local well-posedness, we know $T>0$. For any $t\in[0,T]$, applying $\p^{\gamma}$ on both sides of the equation \eqref{equation: prelim-perturbation}, multiplying $\p^{\gamma}f$, and integrating over $\t^3\times\r^3$, we have
\begin{align}
    &\frac{1}{2}\frac{\ud}{\ud t}\tnm{\p^{\gamma}f}{x,v}^2+\int_{\t^3\times\r^3}\p^{\gamma}f\cdot \l\big[\p^{\gamma}f\big]
    =\int_{\t^3\times\r^3}\p^{\gamma}f\cdot\p^{\gamma}\g[f,f;f].
\end{align}
We further integrate over time to obtain
\begin{align}
    &\tnm{\p^{\gamma}f(t)}{x,v}^2+\int_0^t\int_{\t^3\times\r^3}\p^{\gamma}f\cdot \l\big[\p^{\gamma}f\big]
    =\tnm{\p^{\gamma}f_0}{x,v}^2+\int_0^t\int_{\t^3\times\r^3}\p^{\gamma}f\cdot\p^{\gamma}\g[f,f;f].
\end{align}
For each $t$, we split $t=t'+n$, where $t'\in[0,1)$ and $n$ is a positive integer. Then using Lemma \ref{prelim-theorem: semi-positivity}, Lemma \ref{global-lemma:multi-positivity} and Lemma \ref{prelim-lemma: nonlinear-derivative}, and summing over $\abs{\gamma}\leq N$, we have
\begin{align}
    \bigg(\sum_{\abs{\gamma}\leq N}\tnm{\p^{\gamma}f(t)}{x,v}^2\bigg)+\int_{t'}^t\bigg(\sum_{\abs{\gamma}\leq N}\vnm{\p^{\gamma}f}{x,v}^2\bigg)\ls C_M\ee[f_0]+ C_M\sup_{s\in[0,t]}\ee[f(s)]^2,\no
\end{align}
for some constant $C_M\geq 1$ depending on $M$. However, the second term in LHS still lacks the information on $[0,t']$. We fill this gap by adding the missing piece (the integral over $[0,t']$) on both sides
\begin{align}
\\
    \bigg(\sum_{\abs{\gamma}\leq N}\tnm{\p^{\gamma}f(t)}{x,v}^2\bigg)+\int_{0}^t\bigg(\sum_{\abs{\gamma}\leq N}\vnm{\p^{\gamma}f}{x,v}^2\bigg)
    \ls C_M\ee[f_0]+ C_M\sup_{s\in[0,t]}\ee[f(s)]^2+\int_{0}^{t'}\bigg(\sum_{\abs{\gamma}\leq N}\vnm{\p^{\gamma}f}{x,v}^2\bigg).\no
\end{align}
Then based on Lemma \ref{global-lemma:spatial-derivative}, we know for $t'\in[0,1)$
\begin{align}
    \int_{0}^{t'}\bigg(\sum_{\abs{\gamma}\leq N}\vnm{\p^{\gamma}f}{x,v}^2\bigg)\ls \sum_{\abs{\gamma}\leq N}\tnm{\p^{\gamma}f_0}{x,v}^2.
\end{align}
Hence, in total, we obtain
\begin{align}
    &\sum_{\abs{\gamma}\leq N}\tnm{\p^{\gamma}f(t)}{x,v}^2+\int_{0}^t\bigg(\sum_{\abs{\gamma}\leq N}\vnm{\p^{\gamma}f}{x,v}^2\bigg)
    \leq C_M\ee[f_0]+ C_M\sup_{s\in[0,t]}\ee[f(s)]^2.
\end{align}
Next, we consider the mixed derivative case.
For any $t\in[0,T]$, applying $\p^{\gamma}_{\beta}$ on both sides of the equation \eqref{equation: prelim-perturbation}, multiplying $\p^{\gamma}_{\beta}f$, and integrating over $\t^3\times\r^3$, we have
\begin{align}
\\
    &\frac{1}{2}\frac{\ud}{\ud t}\tnm{\p_{\beta}^{\gamma}f}{x,v}^2+\int_{\t^3\times\r^3}\p_{\beta}^{\gamma}f\cdot \l\big[\p_{\beta}^{\gamma}f\big]
    =-\int_{\t^3\times\r^3}\p_{\beta}^{\gamma}f\cdot \Big(\p_{\beta}\l[\p^{\gamma}f]-\l\big[\p_{\beta}^{\gamma}f\big]\Big)+\int_{\t^3\times\r^3}\p_{\beta}^{\gamma}f\cdot\p_{\beta}^{\gamma}\g[f,f;f].\no
\end{align}
Note that the each term in $\p_{\beta}\l\big[\p^{\gamma}f\big]-\l\big[\p_{\beta}^{\gamma}f\big]$ has $\p^{\gamma}_{\beta'}f$ with $0\leq\abs{\beta'}<\abs{\beta}$. Hence, we may use a simple induction over $\abs{\beta}=0,1,2,\cdots,N$ to obtain
\begin{align}\label{global-estimate:well-posedness}
    &\ee[f(t)]
    \leq C_M\ee[f_0]+ C_M\sup_{s\in[0,t]}\ee[f(s)]^2.
\end{align}
Note that we cannot directly absorb the last term into the LHS since we are not clear whether $C_MM<1$. We further choose $M_0$ such that
\begin{align}
    C_MM_0\leq \frac{1}{2}.
\end{align}
Also, we choose the initial data
\begin{align}
    \ee[f_0]\leq \e_M=\frac{M_0}{4C_M}<\frac{M_0}{2}<\frac{M}{2}.
\end{align}
Denote
\begin{align}
    T_0=\sup\Big\{t\geq0: \sup_{s\in[0,t]} \ee[f(s)]\leq M_0\Big\}.
\end{align}
For $0<t<T_0\leq T$, the bound \eqref{global-estimate:well-posedness} still holds. Hence, we have
\begin{align}
    \ee[f(t)]
    \leq C_M\ee[f_0]+ C_M\sup_{s\in[0,t]}\ee[f(s)]^2\leq \frac{M_0}{4}+\frac{1}{2}\sup_{s\in[0,t]}\ee[f(s)].
\end{align}
Taking supremum over $t\in[0,T_0]$, we have
\begin{align}
    \sup_{t\in[0,T_0]}\ee[f(s)]\leq \frac{M_0}{2}<M_0.
\end{align}
Then by standard continuity argument, we know $T_0=\infty$.
\end{proof}

\begin{theorem}\label{global-theorem: decay}
Under the same assumption as in Theorem \ref{global-theorem: well-posedness}, the global solution $f(t,x,v)$ satisfies
\begin{align}
    \nnm{f(t)}\leq C\ue^{-kt}\nnm{f_0},
\end{align}
for some constant $C,K>0$. 
\end{theorem}

\begin{proof}
We mainly use an argument similar to Had\v{z}i\'c-Guo \cite{Hadzic.Guo2010} and Maslova \cite{Maslova1993}. Based on the proof of Theorem \ref{global-theorem: well-posedness}, we know for any $s<t$ with $\abs{t-s}\geq1$,
\begin{align}
    \nnm{f(t)}^2+\int_s^t\vnnm{f(\tau)}^2\ud\tau\leq  C_0\nnm{f(s)}^2.
\end{align}
Since $\nnm{f(t)}\ls\vnnm{f(t)}$, we know
\begin{align}\label{global-estimate:decay}
    \nnm{f(t)}^2+\int_s^t\nnm{f(\tau)}^2\ud\tau\leq C_0\nnm{f(s)}^2.
\end{align}
Denote
\begin{align}
    V(s)=\int_s^{\infty}\nnm{f(\tau)}^2\ud\tau.
\end{align}
Then naturally 
\begin{align}
    V(s)\leq C_0\nnm{f(s)}^2,
\end{align}
and thus
\begin{align}
    V'(s)=-\nnm{f(s)}^2\leq -\frac{1}{C_0}V(s).
\end{align}
Therefore, we know 
\begin{align}
    V(s)\leq V(0)\ue^{-\frac{1}{C_0}s}.
\end{align}
Then we integrate over $s\in\left[t,2t\right]$ for $t\geq1$ in \eqref{global-estimate:decay}, we have
\begin{align}
    t\nnm{f(t)}^2\leq\int_t^{2t}\nnm{f(\tau)}^2\ud\tau\leq\int_t^{\infty}\nnm{f(\tau)}^2\ud\tau=V(t)
\end{align}
Hence, we have
\begin{align}
    \nnm{f(t)}^2\leq V(0)\ue^{-\frac{1}{C_0}t}.
\end{align}
Since $V(0)\ls M_0$, our result naturally follows.
\end{proof}


\smallskip
\subsection{Global Solutions for $\Omega=\r^3$}

In this section, we will prove the global well-posedness when $\Omega=\r^3$. 

Denote a special dissipation rate
\begin{align}
    \vnnmz{f}=\vnm{\nnpk f}{x,v}+\sum_{0<\abs{\gamma}\leq N}\vnm{\p^{\gamma}_{\beta}f}{x,v}.
\end{align}
Note that this does not include $\vnm{\pk f}{x,v}$, which has not time or spatial derivatives.

\subsubsection{Positivity Estimate for $\l$}

Similar to $\Omega=\t^3$ case, we denote
\begin{align}\label{whole-equation: decomposition}
    f(t,x,v)=&\;\pk[f](t,x,v)+\nnpk[f](t,x,v)\\
    =&\;a(t,x)\mh+b(t,x)\cdot v\mh+c(t,x)\abs{v}^2\mh+\nnpk[f](t,x,v),\no
\end{align}
where $a,b,c$ are given by $f$. 

Plugging \eqref{whole-equation: decomposition} into \eqref{equation: prelim-perturbation} and compare the two sides with the basis
\begin{align}\label{whole-estimate:temp1}
    v_i\abs{v}^2\mh,\quad v_iv_j\mh,\quad v_i\mh,\quad \mh
\end{align}
we obtain the so-called macroscopic equations
\begin{align}
    \p_{x_i}\p^{\gamma}c&=\ell^{\gamma} _{ci}+h_{ci}^{\gamma},\label{whole-equation:macro1}\\
    \dt \p^{\gamma}c+\p_{x_i}\p^{\gamma}b_i&=\ell_i^{\gamma}+h_i^{\gamma},\label{whole-equation:macro2}\\
    \p_{x_i}\p^{\gamma}b_j+\p_{x_j}\p^{\gamma}b_i&=\ell_{ij}^{\gamma}+h_{ij}^{\gamma}\ \text{for}\ i\neq j,\label{whole-equation:macro3}\\
    \dt \p^{\gamma}b_i+\p_{x_i}\p^{\gamma}a&=\ell_{bi}^{\gamma}+h_{bi}^{\gamma},\label{whole-equation:macro4}\\
    \dt \p^{\gamma}a&=\ell_a^{\gamma}+h_a^{\gamma},\label{whole-equation:macro5}
\end{align}
where $\ell_{ci}^{\gamma}$, $\ell_i^{\gamma}$, $\ell_{ij}^{\gamma}$, $\ell_{bi}^{\gamma}$ and $\ell_a^{\gamma}$ are the coefficients corresponding to the above basis for the linear term $-\big(\dt+v\cdot\nx+\big)\nnpk[\p^{\gamma}f]$, and $h_{ci}^{\gamma}$, $h_i^{\gamma}$, $h_{ij}^{\gamma}$, $h_{bi}^{\gamma}$ and $h_a^{\gamma}$ are corresponding to $\p^{\gamma}\g[f,f;f]$.

\begin{lemma}\label{whole-lemma:linear}
We have
\begin{align}
    \sum_{\abs{\gamma}\leq N-1}\Big(\tnm{\ell_{ci}^{\gamma}}{x}+\tnm{\ell_i^{\gamma}}{x}+\tnm{\ell^{\gamma}_{ij}}{x}+\tnm{\ell^{\gamma}_{bi}}{x}+\tnm{\ell^{\gamma}_a}{x}\Big)\ls\sum_{\abs{\gamma}\leq N}\tnm{\nnpk[\p^{\gamma}f]}{x}.
\end{align}
\end{lemma}

\begin{proof}
Assume the basis in \eqref{whole-estimate:temp1} is $\{\e_n(v)\}$. Then the coefficients $\ell_{ci}^{\gamma}$, $\ell_i^{\gamma}$, $\ell_{ij}^{\gamma}$, $\ell_{bi}^{\gamma}$ and $\ell_a^{\gamma}$ are just linear combinations of 
\begin{align}
    \int_{\r^3}\big(\dt+v\cdot\nx+\l\big)\nnpk[\p^{\gamma}f]\cdot\e_n.
\end{align}
Note that with Lemma \ref{prelim-lemma: K-compactness} and Lemma \ref{prelim-theorem: semi-positivity},
\begin{align}
    &\tnm{\int_{\r^3}\big(\dt+v\cdot\nx+\l\big)\nnpk[\p^{\gamma}f]\cdot\e_n}{x}^2\\
    \ls&\int_{\r^3}\abs{\e_n}\cdot\int_{\r^3\times\r^3}\abs{\e_n}\Big(\abs{\nnpk[\dt\p^{\gamma}f]}^2+\abs{v}^2\abs{\nnpk[\nx\p^{\gamma}f]}^2+\abs{\nnpk[\l[\p^{\gamma}f]]}^2\Big)\no\\
    \ls&\tnm{\nnpk[\dt\p^{\gamma}f]}{x}^2+\tnm{\nnpk[\nx\p^{\gamma}f]}{x}^2+\tnm{\nnpk[\p^{\gamma}f]}{x}^2.\no
\end{align}
Hence, our result is obvious.
\end{proof}

\begin{remark}
This lemma indicates that we must include $\dt$ in the definition of $\gamma$.
\end{remark}

\begin{lemma}\label{whole-lemma:nonlinear}
\begin{align}
    \sum_{\abs{\gamma}\leq N}\Big(\tnm{h_{ci}^{\gamma}}{x}+\tnm{h^{\gamma}_i}{x}+\tnm{h^{\gamma}_{ij}}{x}+\tnm{h^{\gamma}_{bi}}{x}+\tnm{h^{\gamma}_a}{x}\Big)\ls \nnm{f}^2\vnnmz{f}.
\end{align}
\end{lemma}

\begin{proof}
Similar to the above lemma, it suffices to bound
\begin{align}
    \tnm{\int_{\r^3}\p^{\gamma}\g[f,f;f]\cdot\e_n}{x}.
\end{align}
This is a bit delicate since $\vnnmz{f}$ does not include the lowest order terms.

For $\abs{\gamma}>0$, the derivative is distributed among the three arguments in $\g$. Based on the second inequality in Lemma \ref{prelim-lemma: nonlinear-no-derivative}, we may assign $L^2_x$ to the term with highest-order derivative to bound it by $\vnnmz{f}$. Then we assign $L^{\infty}_x$ for the other two terms, and the Sobolev embedding helps bound them by $\nnm{f}^2$.

The more delicate case is $\abs{\gamma}=0$. We split $f=\pk[f]+\nnpk[f]$ and get
\begin{align}
    \g[f,f;f]=&\g\Big[f,f;\nnpk[f]\Big]+\g\Big[f,\nnpk[f];\pk[f]\Big]+\g\Big[\nnpk[f],\pk[f];\pk[f]\Big]\\
    &+\g\Big[\pk[f],\pk[f];\pk[f]\Big].\no
\end{align}
Since $\vnnmz{f}$ includes $\tnm{\nnpk[f]}{x}$, so the first three terms are good to go. We just need the estimates as in $\abs{\gamma}>0$ case. The difficult part is the last term
\begin{align}
    \tnm{\int_{\r^3}\g\Big[\pk[f],\pk[f];\pk[f]\Big]\cdot\e_n}{x}.
\end{align}
Since $\pk[f]=a(t,x)\mh+b(t,x)\cdot v\mh+c(t,x)\abs{v}^2\mh$, we have
\begin{align}
    \tnm{\int_{\r^3}\g\Big[\pk[f],\pk[f];\pk[f]\Big]\cdot\e_n}{x}&\ls \tnm{\abs{a}^3+\abs{b}^3+\abs{c}^3}{x}\ls\nm{a}_{L^6_x}^3+\nm{b}_{L^6_x}^3+\nm{c}_{L^6_x}^3.
\end{align}
Due to Sobolev inequality in $\r^3$, we have
\begin{align}
    \nm{a}_{L^6_x}^3+\nm{b}_{L^6_x}^3+\nm{c}_{L^6_x}^3\ls \tnm{\nx a}{x}^3+\tnm{\nx b}{x}^3+\tnm{\nx c}{x}^3\ls \nnm{f}^2\vnnmz{f}.
\end{align}
Hence, our result naturally follows.
\end{proof}

\begin{remark}
The Sobolev inequality in $\r^3$ plays a key role in the proof of this lemma. It does not hold in $\Omega=\t^3$ case.
\end{remark}

\begin{lemma}\label{whole-lemma:positivity}
Assume $f(t,x,v)$ satisfies \eqref{equation: prelim-perturbation} for $t\in[0,T]$ with $T\geq 1$. Assume the initial data $f_0$ satisfies the conservation laws. Also, $f(t,x,v)$ satisfies $\ds\sup_{t\in[0,T]}\nnm{f(t)}\leq M$. Then there exists a constant $\d_M\in(0,1)$ such that
\begin{align}
\\
    \sum_{0<\abs{\gamma}\leq N}\int_{\t^3}\int_{\r^3}\l\big[\p^{\gamma}f(t)\big]\cdot\p^{\gamma}f(t)\ud v\ud x\geq \d_M\sum_{0<\abs{\gamma}\leq N}\vnm{\p^{\gamma}f(t)}{x,v}^2-\frac{\ud}{\ud t}\int_{\r^3}a(\nx\cdot b)-\nnm{f}^4\vnnmz{f}^2.\no
\end{align}
\end{lemma}

\begin{proof}
Due to Lemma \ref{prelim-theorem: semi-positivity}, we know
\begin{align}
    \sum_{0<\abs{\gamma}\leq N}\int_{\t^3}\int_{\r^3}\l\big[\p^{\gamma}f(t)\big]\cdot\p^{\gamma}f(t)\ud v\ud x\geq \d\sum_{0<\abs{\gamma}\leq N}\vnm{\p^{\gamma}\nnpk[f](t)}{x,v}^2.
\end{align}
Hence, it suffices to bound $\vnm{\p^{\gamma}\pk[f](t)}{x,v}$.
Similar the bound of $Z$ in the proof of Lemma \ref{global-lemma:positivity}, and using the proof of Lemma \ref{whole-lemma:linear} and Lemma \ref{whole-lemma:nonlinear}, we know
\begin{align}
    \tnm{\nx\p^{\gamma}c}{x}\ls \tnm{\ell_{ci}^{\gamma}}{x}+\tnm{h_{ci}^{\gamma}}{x}\ls\sum_{0<\abs{\gamma}\leq N}\tnm{\nnpk[\p^{\gamma}f]}{x}+\nnm{f}^2\vnnmz{f},
\end{align}
and
\begin{align}
    \tnm{\nx\p^{\gamma}b_i}{x}&\ls \tnm{\ell^{\gamma}_{ij}}{x}+\tnm{\ell^{\gamma}_{i}}{x}+\tnm{h^{\gamma}_{ij}}{x}+\tnm{h^{\gamma}_{i}}{x}\\
    &\ls\sum_{0<\abs{\gamma}\leq N}\tnm{\nnpk[\p^{\gamma}f]}{x}+\nnm{f}^2\vnnmz{f}.\no
\end{align}
Also, from \eqref{global-equation:macro2}, we have
\begin{align}
    \tnm{\dt\p^{\gamma}c}{x}\ls\tnm{\p_{x_i}\p^{\gamma}b_i}{x}+\tnm{\ell_i^{\gamma}}{x}+\tnm{h_i^{\gamma}}{x}\ls \sum_{0<\abs{\gamma}\leq N}\tnm{\nnpk[\p^{\gamma}f]}{x}+\nnm{f}^2\vnnmz{f}.
\end{align}
The remaining term is for temporal derivative of $b$, which will be discussed later.

For $a$, \eqref{global-equation:macro5} implies
\begin{align}
    \tnm{\dt\p^{\gamma}a}{x}\ls\tnm{\ell_a^{\gamma}}{x}+\tnm{h_a^{\gamma}}{x}\ls \sum_{0<\abs{\gamma}\leq N}\tnm{\nnpk[\p^{\gamma}f]}{x}+\nnm{f}^2\vnnmz{f}.
\end{align}
The remaining term is for purely spatial derivative of $a$. Let $\gamma=[0,\gamma_1,\gamma_2,\gamma_3]$. For $\abs{\gamma}>0$, taking $\p_{x_i}$ in \eqref{global-equation:macro4} yields
\begin{align}
    \p_{x_ix_i}\p^{\gamma}a=-\dt\p_{x_i}\p^{\gamma}b_i+\p_{x_i}(\ell_{bi}^{\gamma}+h_{bi}^{\gamma}).
\end{align}
Multiplying $\p^{\gamma}a$ on both sides, integrating over $\r^3$, and integrating by parts, we have
\begin{align}
    \tnm{\nx\p^{\gamma}a}{x}\ls \tnm{\dt\p^{\gamma}b_i}{x}+\tnm{\ell_{bi}^{\gamma}}{x}+\tnm{h_{bi}^{\gamma}}{x}\ls \sum_{0<\abs{\gamma}\leq N}\tnm{\nnpk[\p^{\gamma}f]}{x}+\nnm{f}^2\vnnmz{f}.
\end{align}
For $\abs{\gamma}=0$, the same procedure implies
\begin{align}
    \tnm{\nx a}{x}^2\ls\int_{\r^3}a(\nx\cdot\dt b)+\tnm{\ell_{bi}}{x}+\tnm{h_{bi}}{x}.
\end{align}
In particular, we know
\begin{align}
    \int_{\r^3}a(\nx\cdot\dt b)=\frac{\ud}{\ud t}\int_{\r^3}a(\nx\cdot b)-\int_{\r^3}\dt a(\nx\cdot b)\leq \frac{\ud}{\ud t}\int_{\r^3}a(\nx\cdot b)+\tnm{\dt a}{x}^2+\tnm{\nx b}{x}^2.
\end{align}
In summary, we have
\begin{align}
    \tnm{\nx a}{x}^2\ls\frac{\ud}{\ud t}\int_{\r^3}a(\nx\cdot b)+\sum_{0<\abs{\gamma}\leq N}\tnm{\nnpk[\p^{\gamma}f]}{x}^2+\nnm{f}^4\vnnmz{f}^2.
\end{align}
Finally, we come to the purely temporal derivative of $b$. For $\gamma=[\gamma_0,0,0,0]$ with $\abs{\gamma}\geq0$ in \eqref{whole-equation:macro4}, we have
\begin{align}
    \tnm{\dt\p^{\gamma}b_i}{x}\ls\tnm{\p_{x_i}\p^{\gamma}a}{x}+\tnm{\ell_{bi}}{x}+\tnm{h_{bi}}{x}.
\end{align}
The RHS has been estimated as above. In particular, for $\abs{\gamma}=0$, we need to introduce $\frac{\ud}{\ud t}\int_{\r^3}a(\nx\cdot b)$.
\end{proof}


\subsubsection{Global Well-Posedness}

Denote 
\begin{align}
    \ee[f(t)]=\nnm{f(t)}^2+\int_0^t\vnnmz{f(s)}^2\ud s,
\end{align}
and
\begin{align}
    \ee[f_0]=\nnm{f_0}^2.
\end{align}
\begin{theorem}\label{whole-theorem: well-posedness}
There exists $M_0>0$ such that if 
\begin{align}
    \ee[f_0]\leq \frac{M_0}{2},
\end{align}
then there exists a unique solution $f(t,x,v)$ to the quantum Boltzmann equation \eqref{equation: prelim-perturbation} such that
\begin{align}
    \ee[f(t)]\leq M_0,
\end{align}
for any $t\in[0,\infty)$. 
\end{theorem}

\begin{proof}
Applying $\p^{\gamma}$ with $\abs{\gamma}>0$ to \eqref{equation: prelim-perturbation}, multiplying $\p^{\gamma}f$ on both sides and integrating over $\r^3\times\r^3$, we get
\begin{align}
    \frac{1}{2}\sum_{0<\abs{\gamma}\leq N}\tnm{\p^{\gamma}f}{x,v}^2+\sum_{0<\abs{\gamma}\leq N}\int_{\r^3\times\r^3}\l\big[\p^{\gamma}f\big]\cdot\p^{\gamma}f=\sum_{0<\abs{\gamma}\leq N}\int_{\r^3\times\r^3}\p^{\gamma}\g[f,f;f]\cdot \p^{\gamma}f.
\end{align}
Using Lemma \ref{whole-lemma:positivity} and similar techniques as in the proof of Lemma \ref{whole-lemma:nonlinear}, we have
\begin{align}
    \frac{\ud}{\ud t}\bigg(\sum_{0<\abs{\gamma}\leq N}\tnm{\p^{\gamma}f}{x,v}^2-\int_{\r^3}a(\nx\cdot b)\bigg)+\sum_{0<\abs{\gamma}\leq N}\vnm{\p^{\gamma}f}{x,v}^2\ls \nnm{f}^4\vnnmz{f}^2.
\end{align}
For $\abs{\gamma}=0$, we have
\begin{align}
    \frac{1}{2}\tnm{f}{x}^2+\int_{\r^3\times\r^3}\l[f]\cdot f=\int_{\r^3\times\r^3}\g[f,f;f]\cdot f.
\end{align}
Note that 
\begin{align}
    \int_{\r^3\times\r^3}\g[f,f;f]\cdot f=\int_{\r^3\times\r^3}\g[f,f;f]\cdot \nnpk[f].
\end{align}
Using similar techniques as in the proof of Lemma \ref{whole-lemma:nonlinear}, we have
\begin{align}
    \tnm{f}{x,v}^2+\tnm{\nnpk[f]}{x,v}^2\ls \nnm{f}^2\vnnmz{f}^2.
\end{align}
In total, we have
\begin{align}
    \frac{\ud}{\ud t}\bigg(C\tnm{f}{x,v}+\sum_{0<\abs{\gamma}\leq N}\tnm{\p^{\gamma}f}{x,v}^2-\int_{\r^3}a(\nx\cdot b)\bigg)\\
    +\bigg(\sum_{0<\abs{\gamma}\leq N}\vnm{\p^{\gamma}f}{x,v}^2+\tnm{\nnpk[f]}{x,v}^2\bigg)&\ls \nnm{f}^2\vnnmz{f}^2.\no
\end{align}
In particular, we may choose $C$ sufficiently large to kill $\int_{\r^3}a(\nx\cdot b)$. Hence, we have
\begin{align}
    \frac{\ud}{\ud t}\bigg(\tnm{f}{x,v}+\sum_{0<\abs{\gamma}\leq N}\tnm{\p^{\gamma}f}{x,v}^2\bigg)
    +\bigg(\sum_{0<\abs{\gamma}\leq N}\vnm{\p^{\gamma}f}{x,v}^2+\tnm{\nnpk[f]}{x,v}^2\bigg)&\ls \nnm{f}^2\vnnmz{f}^2.
\end{align}
This is for $\abs{\beta}=0$ case. When $\abs{\beta}>0$, we use similar induction as in the proof of Theorem \ref{global-theorem: well-posedness} to obtain
\begin{align}
    \frac{\ud}{\ud t}\nnm{f}^2+\vnnmz{f}\ls \nnm{f}^2\vnnmz{f}^2.
\end{align}
Here note the key fact that
\begin{align}
    \int_{\r^3}\l[f]\cdot g=\int_{\r^3}\l\big[\nnpk[f]\big]\cdot\nnpk[g],
\end{align}
and
\begin{align}
    \int_{\r^3}\g[f_1,f_2;f_3]\cdot g=\int_{\r^3}\g[f_1,f_2;f_3]\cdot \nnpk[g].
\end{align}
This helps handle the case when the velocity derivative hits $\nu$, $K$ or $\g$. Since $\npk$ part is included in the dissipation, we are good to go. Then by a similar argument as in the proof of Theorem \ref{global-theorem: well-posedness}, we obtain the global well-posedness.
\end{proof}

\begin{remark}
This provide a different framework to justify global well-posedness. It also works for $\Omega=\t^3$ case. However, note that Theorem \ref{global-theorem: well-posedness} is slightly better since there we do not need to take temporal derivatives.
\end{remark}




\bigskip
\section{Global Stability of the Vacuum} \label{Sec:Vacuum-GlbSol}

In this section, we focus on the global well-posedness and positivity of the mild solution near the vacuum. 


\smallskip
\subsection{Mild Formulation}

As in the classical Boltzmann equation, we decompose the collision term
\begin{align}
    Q[F,F;F]=&Q_{\text{gain}}[F,F;F]-Q_{\text{loss}}[F,F;F],
\end{align}
where
\begin{align}
    Q_{\text{gain}}[F,F;F]:=&\int_{\r^3}\int_{\s^2}q(\omega,\abs{v-u})F(u' )F(v')\big(1+\th F(u)+\th F(v)\big)\ud\o\ud u,\\
    Q_{\text{loss}}[F,F;F]:=&\int_{\r^3}\int_{\s^2}q(\omega,\abs{v-u})F(u)F(v)\big(1+\th F(u')+\th F(v')\big)\ud\o\ud u.
\end{align}
In particular, we might write
\begin{align}
    Q_{\text{loss}}[F,F;F](v)=F(v)\cdot R[F,F](v),
\end{align}
where
\begin{align}
    R[F,F]:=\int_{\r^3}\int_{\s^2}q(\o,\abs{v-u})F(u)\big(1+\th F(u')+\th F(v')\big)\ud\o\ud u.
\end{align}

Given $\beta>0$, define 
\begin{align}
    S=\Big\{F\in C^0(\rp\times\r^3\times\r^3): \text{there exists}\ c>0\ \text{such that}\ \abs{F(t,x,v)}\leq c\ue^{-\beta(\abs{x}^2+\abs{v}^2)}\Big\},
\end{align}
equipped with norm
\begin{align}
    \nnm{F}:=\sup_{t,x,v}\Big(\ue^{\beta(\abs{x}^2+\abs{v}^2)}\abs{F(t,x,v)}\Big).
\end{align}
We name the weight function
\begin{align}
    w(x,v):=\ue^{\beta(\abs{x}^2+\abs{v}^2)}.
\end{align}
We introduce the transported solution
\begin{align}
    \fs(t,x,v)=F(t,x+tv,v).
\end{align}
Then the quantum Boltzmann equation can be rewritten as
\begin{align}\label{equation: vacuum}
    \dt\fs=\qs[F,F;F]=\qs_{\text{gain}}[F,F;F]-\qs_{\text{loss}}[F,F;F],
\end{align}
where 
\begin{align}
    &\qs_{\text{gain}}[F,F;F](t,x,v)=\qs_{\text{gain}}[F,F;F](t,x+tv,v)\\
    =&\int_{\r^3}\int_{\s^2}q(\omega,\abs{v-u})F(t,x+tv,u')F(t,x+tv,v')\big(1+\th F(t,x+tv,u)+\th F(t,x+tv,v)\big)\ud\o\ud u,\no\\
    =&\int_{\r^3}\int_{\s^2}q(\omega,\abs{v-u})\fs\big(t,x+t(v-u'),u'\big)\fs\big(t,x+t(v-v'),v'\big)\Big(1+\th \fs\big(t,x+t(v-u),u\big)+\th \fs\big(t,x,v\big)\Big)\ud\o\ud u,\no\\
    &\qs_{\text{loss}}[F,F;F](t,x,v)=\qs_{\text{loss}}[F,F;F](t,x+tv,v)\\
    =&\int_{\r^3}\int_{\s^2}q(\omega,\abs{v-u})F(t,x+tv,u)F(t,x+tv,v)\big(1+\th F(t,x+tv,u')+\th F(t,x+tv,v')\big)\ud\o\ud u,\no\\
    =&\int_{\r^3}\int_{\s^2}q(\omega,\abs{v-u})\fs\big(t,x+t(v-u),u\big)\fs\big(t,x,v\big)\Big(1+\th \fs\big(t,x+t(v-u'),u'\big)+\th \fs\big(t,x+t(v-v'),v'\big)\Big)\ud\o\ud u.\no
\end{align}


\smallskip
\subsection{Global Well-Posedness}

The equation \eqref{equation: vacuum} can be written in the mild formulation
\begin{align}
    \fs(t,x,v)=F_0(x,v)+\int_0^t\qs[F,F;F](\tau,x,v)\ud\tau.
\end{align}
We call the function $F\in S$ satisfying the above a mild solution to \eqref{equation: vacuum}. Hence, the key is to bound $\ds\int_0^t\qs_{\text{gain}}[F,F;F](\tau,x,v)\ud\tau$ and $\ds\int_0^t\qs_{\text{loss}}[F,F;F](\tau,x,v)\ud\tau$.

\begin{lemma}\label{vacuum-lemma: prelim 1}
We have
\begin{align}
    \int_0^{\infty}\ue^{-\beta\abs{x+\tau(v-u)}^2}\ud\tau\leq \sqrt{\frac{\beta}{\pi}}\frac{1}{\abs{v-u}}.
\end{align}
\end{lemma}

\begin{proof}
This is \cite[Lemma 2.1.1]{Glassey1996}.
\end{proof}

\begin{lemma}\label{vacuum-lemma: prelim 2}
We have
\begin{align}
    \abs{x+\tau(u-v')}^2+\abs{x+\tau(v-v')^2}=\abs{x}^2+\abs{x+\tau(v-u)}^2.
\end{align}
\end{lemma}

\begin{proof}
This is \cite[(2.19)]{Glassey1996}. Note that the conservation laws of the classical and quantum Boltzmann equations are the same, i.e. $u+v=u'+v'$ and $\abs{u}^2+\abs{v}^2=\abs{u'}^2+\abs{v'}^2$.
\end{proof}

\begin{lemma}\label{vacuum-lemma: loss-estimate}
For $\fs\in S$ and any $t\geq0$, we have
\begin{align}
    \abs{\int_0^t\qs_{\text{loss}}[F,F;F](\tau,x,v)\ud\tau}\ls \beta^{-2}w^{-1}(x,v)\Big(\nnm{\fs}^2+\nnm{\fs}^3\Big).
\end{align}
\end{lemma}

\begin{proof}
We may further decompose
\begin{align}
    &\qs_{\text{loss}}[F,F;F](t,x,v)=\qs_{\text{loss},1}[F,F](t,x,v)+\th\qs_{\text{loss},2}[F,F;F](t,x,v)+\th\qs_{\text{loss,3}}[F,F;F](t,x,v),
\end{align}
where
\begin{align}
    \qs_{\text{loss},1}[F,F](t,x,v)=&\int_{\r^3}\int_{\s^2}q(\omega,\abs{v-u})\fs\big(t,x+t(v-u),u\big)\fs\big(t,x,v\big)\ud\o\ud u,\no\\
    \qs_{\text{loss},2}[F,F;F](t,x,v)=&\int_{\r^3}\int_{\s^2}q(\omega,\abs{v-u})\fs\big(t,x+t(v-u),u\big)\fs\big(t,x,v\big) \fs\big(t,x+t(v-u'),u'\big)\ud\o\ud u,\no\\
    \qs_{\text{loss,3}}[F,F;F](t,x,v)=&\int_{\r^3}\int_{\s^2}q(\omega,\abs{v-u})\fs\big(t,x+t(v-u),u\big)\fs\big(t,x,v\big) \fs\big(t,x+t(v-v'),v'\big)\ud\o\ud u.\no
\end{align}
We first consider $\qs_{\text{loss},1}$. Direct computation reveals
\begin{align}
    \abs{\int_0^t\qs_{\text{loss},1}[F,F](\tau,x,v)\ud\tau}=&\abs{\int_0^t\fs\big(\tau,x,v\big)\ud\tau\int_{\r^3}\abs{v-u}\fs\big(\tau,x+\tau(v-u),u\big)\ud u}\\
    \ls&\;w^{-1}(x,v)\nnm{\fs}\abs{\int_0^t\int_{\r^3}\abs{v-u}\fs\big(\tau,x+\tau(v-u),u\big)\ud u\ud\tau}\no\\
    \ls&\;w^{-1}(x,v)\nnm{\fs}^2\abs{\int_0^t\int_{\r^3}\abs{v-u}\ue^{-\beta\abs{u}^2}\ue^{-\beta\abs{x+\tau(v-u)}^2}\ud u\ud\tau}\no\\
    =&\;w^{-1}(x,v)\nnm{\fs}^2\abs{\int_{\r^3}\abs{v-u}\ue^{-\beta\abs{u}^2}\ud u\int_0^t\ue^{-\beta\abs{x+\tau(v-u)}^2}\ud\tau}.\no
\end{align}
Based on Lemma \ref{vacuum-lemma: prelim 1}, we know
\begin{align}
    \int_0^t\ue^{-\abs{x+\tau(v-u)}^2}\ud\tau\leq \int_0^{\infty}\ue^{-\abs{x+\tau(v-u)}^2}\ud\tau\leq \sqrt{\frac{\beta}{\pi}}\frac{1}{\abs{v-u}}.
\end{align}
Hence, we have
\begin{align}
    \abs{\int_0^t\qs_{\text{loss},1}[F,F](\tau,x,v)\ud\tau}\ls&\;\beta^{-\frac{1}{2}}w^{-1}(x,v)\nnm{\fs}^2\int_{\r^3}\ue^{-\beta\abs{u}^2}\ud u\ls \beta^{-2}w^{-1}(x,v)\nnm{\fs}^2.
\end{align}
Next, we turn to $\qs_{\text{loss},2}$. We have
\begin{align}
    &\abs{\int_0^t\qs_{\text{loss},2}[F,F;F](\tau,x,v)\ud\tau}\\
    =&\,\abs{\int_0^t\fs\big(\tau,x,v\big)\ud\tau\int_{\r^3}\int_{\s^2}\abs{\o\cdot(v-u)}\fs\big(\tau,x+\tau(v-u),u\big)\fs\big(\tau,x+\tau(v-u'),u'\big)\ud\o\ud u}\no\\
    \ls&\;w^{-1}(x,v)\nnm{\fs}^3\abs{\int_0^t\int_{\r^3}\int_{\s^2}\abs{\o\cdot(v-u)}\ue^{-\beta\abs{u}^2}\ue^{-\beta\abs{u'}^2}\ue^{-\beta\abs{x+\tau(v-u)}^2}\ue^{-\beta\abs{x+\tau(v-u')}^2}\ud\o\ud u\ud\tau}.\no
\end{align}
Since $\ue^{-\beta\abs{u'}^2}\leq 1$ and $\ue^{-\beta\abs{x+\tau(v-u')}^2}\leq1$, it reduces to $\qs_{\text{loss},1}$ case. Hence, we have
\begin{align}
    \abs{\int_0^t\qs_{\text{loss},2}[F,F;F](\tau,x,v)\ud\tau}\ls \beta^{-2}w^{-1}(x,v)\nnm{\fs}^3.
\end{align}
Similarly, we know
\begin{align}
    \abs{\int_0^t\qs_{\text{loss},3}[F,F;F](\tau,x,v)\ud\tau}\ls \beta^{-2}w^{-1}(x,v)\nnm{\fs}^3.
\end{align}
\end{proof}

\begin{lemma}\label{vacuum-lemma: gain-estimate}
For $\fs\in S$ and any $t\geq0$, we have
\begin{align}
    \abs{\int_0^t\qs_{\text{gain}}[F,F;F](\tau,x,v)\ud\tau}\ls \beta^{-2}w^{-1}(x,v)\Big(\nnm{\fs}^2+\nnm{\fs}^3\Big).
\end{align}
\end{lemma}

\begin{proof}
We may further decompose
\begin{align}
    &\qs_{\text{gain}}[F,F;F](t,x,v)=\qs_{\text{gain},1}[F,F](t,x,v)+\th\qs_{\text{gain},2}[F,F;F](t,x,v)+\th\qs_{\text{gain,3}}[F,F;F](t,x,v),
\end{align}
where
\begin{align}
    \qs_{\text{gain},1}[F,F](t,x,v)=&\int_{\r^3}\int_{\s^2}q(\omega,\abs{v-u})\fs\big(t,x+t(v-u'),u'\big)\fs\big(t,x+t(v-v'),v'\big)\ud\o\ud u,\no\\
    \qs_{\text{gain},2}[F,F;F](t,x,v)=&\int_{\r^3}\int_{\s^2}q(\omega,\abs{v-u})\fs\big(t,x+t(v-u'),u'\big)\fs\big(t,x+t(v-v'),v'\big)\fs\big(t,x+t(v-u),u\big)\ud\o\ud u,\no\\
    \qs_{\text{gain,3}}[F,F;F](t,x,v)=&\int_{\r^3}\int_{\s^2}q(\omega,\abs{v-u})\fs\big(t,x+t(v-u'),u'\big)\fs\big(t,x+t(v-v'),v'\big) \fs\big(t,x,v\big)\ud\o\ud u.\no
\end{align}
We first consider $\qs_{\text{gain},1}$. Direct computation reveals
\begin{align}
    &\abs{\int_0^t\qs_{\text{gain},1}[F,F](\tau,x,v)\ud\tau}\\
    =&\,\abs{\int_0^t\int_{\r^3}\int_{\s^2}\abs{v-u}\fs\big(\tau,x+\tau(v-u'),u'\big)\fs\big(\tau,x+\tau(v-v'),v'\big)\ud\o\ud u\ud\tau}\no\\
    \ls&\;\nnm{\fs}^2\abs{\int_0^t\int_{\r^3}\int_{\s^2}\abs{v-u}\ue^{-\beta\abs{u'}^2}\ue^{-\beta\abs{v'}^2}\ue^{-\beta\abs{x+\tau(v-u')}^2}\ue^{-\beta\abs{x+\tau(v-v')}^2}\ud\o\ud u\ud\tau}.\no
\end{align}
Using Lemma \ref{vacuum-lemma: prelim 2} and the conservation laws for $(u,v)$ and $(u',v')$, we have
\begin{align}
    \abs{\int_0^t\qs_{\text{gain},1}[F,F](\tau,x,v)\ud\tau}
    \ls&\;\nnm{\fs}^2\abs{\int_0^t\int_{\r^3}\abs{v-u}\ue^{-\beta\abs{u}^2}\ue^{-\beta\abs{v}^2}\ue^{-\beta\abs{x}^2}\ue^{-\beta\abs{x+\tau(v-u)}^2}\ud u\ud\tau}\\
    \leq&\;w^{-1}(x,v)\nnm{\fs}^2\abs{\int_0^t\int_{\r^3}\abs{v-u}\ue^{-\beta\abs{u}^2}\ue^{-\beta\abs{x+\tau(v-u)}^2}\ud u\ud\tau}.\no
\end{align}
Then it reduces to the $\qs_{\text{loss},1}$ case in Lemma \ref{vacuum-lemma: loss-estimate}, so we know
\begin{align}
    \abs{\int_0^t\qs_{\text{gain},1}[F,F](\tau,x,v)\ud\tau}
    \ls\beta^{-2}w^{-1}(x,v)\nnm{\fs}^2.
\end{align}
Next, we turn to $\qs_{\text{gain},2}$. We have
\begin{align}
    &\abs{\int_0^t\qs_{\text{gain},2}[F,F;F](\tau,x,v)\ud\tau}\\
    \ls\;&\nnm{\fs}^3\abs{\int_0^t\int_{\r^3}\int_{\s^2}\abs{\o\cdot(v-u)}\ue^{-\beta\abs{u}^2}\ue^{-\beta\abs{u'}^2}\ue^{-\beta\abs{v'}^2}\ue^{-\beta\abs{x+\tau(v-u)}^2}\ue^{-\beta\abs{x+\tau(v-u')}^2}\ue^{-\beta\abs{x+\tau(v-v')}^2}\ud\o\ud u\ud\tau}.\no
\end{align}
Using the same technique as in $\qs_{\text{gain},1}$ case, we have
\begin{align}
    \abs{\int_0^t\qs_{\text{gain},2}[F,F;F](\tau,x,v)\ud\tau}
    \ls&\;\nnm{\fs}^3\abs{\int_0^t\int_{\r^3}\abs{v-u}\ue^{-2\beta\abs{u}^2}\ue^{-\beta\abs{v}^2}\ue^{-\beta\abs{x}^2}\ue^{-2\beta\abs{x+\tau(v-u)}^2}\ud u\ud\tau}\\
    \ls&\; \beta^{-2}w^{-1}(x,v)\nnm{\fs}^3.\no
\end{align}
For $\qs_{\text{gain},3}$, we directly get
\begin{align}
    \abs{\int_0^t\qs_{\text{gain},3}[F,F;F](\tau,x,v)\ud\tau}\ls&\;w^{-1}(x,v)\nnm{\fs}\abs{\int_0^t\qs_{\text{gain},1}[F,F](\tau,x,v)\ud\tau}\\
    \ls&\;\beta^{-2}w^{-2}(x,v)\nnm{\fs}^3.\no
\end{align}
\end{proof}

Define the operator
\begin{align}
    \mathcal{F}[\fs](x,v)=F_0(x,v)+\int_0^t\qs[F,F;F](\tau,x,v)\ud\tau.
\end{align}
Define the solution set
\begin{align}
    S_R=\{F\in S: \nnm{F}\leq R\}.
\end{align}
\begin{theorem}\label{vacuum-theorem: well-posedness}
There exists a constant $R_0$ such that if $\nnm{F_0}<\dfrac{R_0}{2}$, then the equation \eqref{equation: vacuum} has a unique mild solution $F\in S_{R_0}$.
\end{theorem}

\begin{proof}
Using Lemma \ref{vacuum-lemma: gain-estimate} and Lemma \ref{vacuum-lemma: loss-estimate}, we know for $\nnm{F_0}\leq \dfrac{R_0}{2}$, we have
\begin{align}
    \abs{\mathcal{F}[\fs]}\leq&\, \abs{F_0}+\abs{\int_0^t\qs_{\text{gain}}[F,F;F](\tau,x,v)\ud\tau}+\abs{\int_0^t\qs_{\text{loss}}[F,F;F](\tau,x,v)\ud\tau}\\
    \leq&\;w^{-1}(x,v)\nnm{F_0}+\beta^{-2}w^{-1}(x,v)\Big(\nnm{\fs}^2+\nnm{\fs}^3\Big).\no
\end{align}
Therefore, we know
\begin{align}
    \nnm{\mathcal{F}[\fs]}\ls\;& \nnm{F_0}+\beta^{-2}\Big(\nnm{\fs}^2+\nnm{\fs}^3\Big)\\
    \ls\;&\frac{R_0}{2}+\beta^{-2}(R_0^2+R_0^3)\leq R_0.\no
\end{align}
Hence, $\mathcal{F}$ is a mapping from $S_{R_0}$ to $S_{R_0}$. A similar argument justifies that this is a contraction. Hence, the solution exists uniquely.
\end{proof}


\begin{remark}
This theorem justifies that for both fermions and bosons, in the space $S$, the solution remains small and smooth. Hence, if the initial data is in $S$ and is sufficiently small, there is no possibility of Bose-Einstein condensation. This is significantly different from the result of homogeneous equation.

In the homogeneous equation, the lack of transport operator means that we lose the dispersion and cannot handle the time integral. This is exactly the key in non-homogeneous case.
\end{remark}


\smallskip
\subsection{Positivity of $F$ for Bosons $\th=1$}

Recall \begin{align}
    Q_{\text{gain}}[F,G;H]=&\int_{\r^3}\int_{\s^2}q(\omega,\abs{v-u})F(u' )G(v')\big(1+\th H(u)+\th H(v)\big)\ud\o\ud u,
\end{align}
and
\begin{align}
    Q_{\text{loss}}[F,G;H](v)=F(v)\cdot R[G,H](v),
\end{align}
where
\begin{align}
    R[G,H]=\int_{\r^3}\int_{\s^2}q(\o,\abs{v-u})G(u)\big(1+\th H(u')+\th H(v')\big)\ud\o\ud u.
\end{align}

Suppose $S_T$ is the restriction of element $F\in S$ to $[0,T]\times\r^3\times\r^3$. Assume $\ell_0(t,x,v)\leq u_0(t,x,v)$ with $\ell_0,u_0\in S$. Define a sequence
\begin{align}
    \dt\lls_{k+1}+\lls_{k+1}\rs[u_k,u_k]&=\qs_{\text{gain}}[\ell_k,\ell_k;\ell_k],\\
    \dt\uus_{k+1}+\uus_{k+1}\rs[\ell_k,\ell_k]&=\qs_{\text{gain}}[u_k,u_k;u_k],
\end{align}
with initial data $\ell_{k+1}=F_0$ and $u_{k+1}=F_0$.

We would like to select some special starting point $(\ell_0,u_0)$ and study the convergence property of $(\lls_k,\uus_k)$.

\begin{lemma}\label{vacuum-lemma: positivity 1}
If $\nnm{F_0}$ and $\beta^{-2}(R_0+R_0^2)$ are sufficiently small, then there exists $\ell_0,u_0\in S_T$ such that the Beginning Condition (BC)
\begin{align}
    0\leq \ell_0\leq \ell_1\leq u_1\leq u_0
\end{align}
holds for $t\in[0,T]$.
\end{lemma}

\begin{proof}
We take $\ell_0=0$, which implies $\rs[\ell_0,\ell_0]=0$ and $\qs_{\text{gain}}[\ell_0,\ell_0;\ell_0]=0$. Hence, we have
\begin{align}
    \dt\lls_{1}+\lls_{1}\rs[u_0,u_0]&=0,\\
    \dt\uus_{1}&=\qs_{\text{gain}}[u_0,u_0;u_0].
\end{align}
Due to the positivity of $\rs[u_0,u_0]$ and $\qs_{\text{gain}}[u_0,u_0;u_0]$, this naturally implies
\begin{align}
    0\leq \lls_1\leq F_0\leq \uus_1.
\end{align}
Hence, it remains to show $\uus_1\leq \uus_0$. This does not hold for arbitrary $u_0$, so we need a delicate construction.

Let 
\begin{align}
    \psi(v)=\sup_x\ue^{\beta\abs{x}^2}\abs{F_0(x,v)}.
\end{align}
Then we know
\begin{align}
    \psi(v)\ls \ue^{-\beta\abs{v}^2}.
\end{align}
We know
\begin{align}
    \uus_1(t,x,v)=F_0+\int_0^t\qs_{\text{gain}}[u_0,u_0;u_0](\tau,x,v)\ud\tau,
\end{align}
or equivalently
\begin{align}
    u_1(t,x+tv,v)=\,&F_0+\int_0^t\int_{\r^3}\int_{\s^2}\abs{\o\cdot(v-u)}u_0(\tau,x+\tau v,u')u_0(\tau,x+\tau v,v')\\
    \,&\times\big(1+\th u_0(\tau,x+\tau v,u)+\th u_0(\tau,x+\tau v,v)\big)\ud\o\ud u\ud\tau.\no
\end{align}
We will look for $u_0=\tilde v(x-tv,v)$, Therefore, $u_1\leq u_0$ if and only if 
\begin{align}
    \int_0^t\int_{\r^3}\int_{\s^2}\abs{\o\cdot(v-u)}\tilde v(x+\tau (v-u'),u')\tilde v(x+\tau (v-v'),v')\\
    \times\big(1+\th \tilde v(x+\tau (v-u),u)+\th \tilde v(x,v)\big)\ud\o\ud u\ud\tau&\,\leq\, \tilde v(x,v)-F_0(x,v).\no
\end{align}
Then we further require $\tilde v(x,v)=\ue^{-\beta\abs{x}^2}w(v)$. Then we know
\begin{align}
    \tilde v(x+\tau (v-u'),u')\tilde v(x+\tau (v-v'),v')
    =&\;w(u')w(v')\ue^{-\beta\abs{x+\tau (v-u')}^2}\ue^{-\beta\abs{x+\tau (v-v')}^2}\\
    =&\;w(u')w(v')\ue^{-\beta\abs{x}^2}\ue^{-\beta\abs{x+\tau (v-u)}^2}.\no
\end{align}
Hence, $u_1\leq u_0$ if and only if
\begin{align}
\\
    \int_0^t\int_{\r^3}\int_{\s^2}\abs{\o\cdot(v-u)}w(u')w(v')\ue^{-\beta\abs{x+\tau (v-u)}^2}\times\big(1+\th \ue^{-\beta\abs{x+\tau (v-u)}^2}w(u)+\th w(v)\big)\ud\o\ud u\ud\tau\leq w(v)-\psi(v).\no
\end{align}
Using Lemma \ref{vacuum-lemma: prelim 1}, we know
\begin{align}
    \sqrt{\frac{\pi}{\beta}}\big(1+\th w(v)\big)\int_{\r^3}\int_{\s^2}w(u')w(v')\ud\o\ud u+\th\sqrt{\frac{\pi}{2\beta}}\int_{\r^3}\int_{\s^2}w(u')w(v')w(u)\ud\o\ud u\leq w(v)-\psi(v).
\end{align}
To prove the existence of solution $w(v)\geq0$, we introduce the space
\begin{align}
    W=\bigg\{w\in C(\r^3): \text{there exists}\ c>0\ \text{such that}\ \abs{w(v)}\leq c\ue^{-\beta\abs{v}^2} \bigg\},
\end{align}
equipped with norm
\begin{align}
    \nnm{w}_W=\sup_v\ue^{\beta\abs{v}^2}\abs{w(v)}.
\end{align}
Define operator
\begin{align}
    \mathcal{T}[w]=\psi(v)+\sqrt{\frac{\pi}{\beta}}\big(1+\th w(v)\big)\int_{\r^3}\int_{\s^2}w(u')w(v')\ud\o\ud u+\th\sqrt{\frac{\pi}{2\beta}}\int_{\r^3}\int_{\s^2}w(u')w(v')w(u)\ud\o\ud u.
\end{align}
If $w\geq0$, then naturally $T[w]\geq0$. Then similar to the proof of Lemma \ref{vacuum-lemma: gain-estimate} and Lemma \ref{vacuum-lemma: loss-estimate}, we have
\begin{align}
    \nnm{T[w]}_W\ls \nnm{\psi}+\beta^{-2}\Big(\nnm{w}_W^2+\nnm{w}_W^3\Big).
\end{align}
Hence, when $\nnm{F_0}$ and $\beta^{-2}(R_0+R_0^2)$ are sufficiently small, we can easily justify that $\mathcal{T}$ maps a small ball under $W$ norm into the same ball and it is a contraction. Therefore, such $w$ must exist.
\end{proof}

\begin{lemma}\label{vacuum-lemma: positivity 2}
If $\ell_0,u_0\in S_T$ such that the Beginning Condition (BC)
\begin{align}
    0\leq \ell_0\leq \ell_1\leq u_1\leq u_0
\end{align}
for $t\in[0,T]$, then the iterative sequence $(\ell_k,u_k)$ are always well-defined for $t\in[0,T]$ and satisfies 
\begin{align}
    \ell_{k}\leq \ell_{k+1}\leq u_{k+1}\leq u_{k}.
\end{align}
\end{lemma}

\begin{proof}
The sequence is naturally well-defined due to basic ODE theory. We will focus on the inequality. Rewrite the iteration into mild formulation
\begin{align}
    \lls_{k+1}(t)&=F_0\ue^{-\int_0^t\rs[u_k,u_k]}+\int_0^t\ue^{-\int_{\tau}^t\rs[u_k,u_k]}\qs_{\text{gain}}[\ell_k,\ell_k;\ell_k]\ud\tau,
\end{align}
and
\begin{align}
    \lls_{k}(t)&=F_0\ue^{-\int_0^t\rs[u_{k-1},u_{k-1}]}+\int_0^t\ue^{-\int_{\tau}^t\rs[u_{k-1},u_{k-1}]}\qs_{\text{gain}}[\ell_{k-1},\ell_{k-1};\ell_{k-1}]\ud\tau.
\end{align}
Also, we assume 
\begin{align}
    \ell_{k-1}\leq \ell_k\leq u_k\leq u_{k-1}.
\end{align}
Then
\begin{align}
    \lls_{k+1}(t)-\lls_{k}(t)=\,&F_0\bigg(\ue^{-\int_0^t\rs[u_k,u_k]}-\ue^{-\int_0^t\rs[u_{k-1},u_{k-1}]}\bigg)\\
    \,&+\int_0^t\bigg(\ue^{-\int_{\tau}^t\rs[u_k,u_k]}-\ue^{-\int_{\tau}^t\rs[u_{k-1},u_{k-1}]}\bigg)\qs_{\text{gain}}[\ell_k,\ell_k;\ell_k]\ud\tau\no\\
    \,&+\int_0^t\ue^{-\int_{\tau}^t\rs[u_{k-1},u_{k-1}]}\bigg(\qs_{\text{gain}}[\ell_k,\ell_k;\ell_k]-\qs_{\text{gain}}[\ell_{k-1},\ell_{k-1};\ell_{k-1}]\bigg)\ud\tau.\no
\end{align}
Due to the monotonicity of $R$ and $Q_{\text{gain}}$, we know all the three terms on the RHS are nonnegative. Hence, we know $\lls_{k+1}\geq\lls_{k}$. Similarly, we have $\uus_{k+1}\leq\uus_{k}$. By induction, we know the desired inequality holds.
\end{proof}

\begin{theorem}\label{vacuum-theorem: positivity}
There exists a constant $R_0$ such that if $\nnm{F_0}$ and $\beta^{-2}(R_0+R_0^2)$ are sufficiently small with $F_0\geq0$, then the equation \eqref{equation: vacuum} for bosons has a unique mild solution $F\in S_{R_0}$ with $F\geq0$.
\end{theorem}

\begin{proof}
Let $k\rt\infty$ in the iteration, since we know $\ell_k$ and $u_k$ are pointwise monotone with proper upper and lower bounds, dominated convergence theorem implies $\ell_k\rt\ell$ and $u_k\rt u$ satisfying
\begin{align}
    \lls-F_0+\int_0^t\lls \rs[u,u]&=\int_0^t\qs_{\text{gain}}[\ell,\ell;\ell],\\
    \uus-F_0+\int_0^t\uus \rs[\ell,\ell]&=\int_0^t\qs_{\text{gain}}[u,u;u].
\end{align}
Taking the difference, we have
\begin{align}
    \uus-\lls=&\bigg(\int_0^t\lls \rs[u,u]-\int_0^t\uus \rs[\ell,\ell]\bigg)+\bigg(\int_0^t\qs_{\text{gain}}[u,u;u]-\int_0^t\qs_{\text{gain}}[\ell,\ell;\ell]\bigg).
\end{align}
Hence, we have
\begin{align}
    \nnm{\uus-\lls}\ls \beta^{-2}\Big(R_0^2+R_0^3\Big)\nnm{\uus-\lls},
\end{align}
which implies $\uus=\lls$. They both converge to the solution $F$ to the equation \eqref{equation: vacuum}. Based on our construction, we know $F$ is nonnegative.
\end{proof}


\smallskip
\subsection{Positivity of $F$ for Fermions $\th=-1$}

Recall \begin{align}
    Q_{\text{gain}}[F,G;H]=&\int_{\r^3}\int_{\s^2}q(\omega,\abs{v-u})F(u' )G(v')\big(1+\th H(u)+\th H(v)\big)\ud\o\ud u,
\end{align}
and
\begin{align}
    Q_{\text{loss}}[F,G;H](v)=F(v)\cdot R[G,H](v),
\end{align}
where
\begin{align}
    R[G,H]=\int_{\r^3}\int_{\s^2}q(\o,\abs{v-u})G(u)\big(1+\th H(u')+\th H(v')\big)\ud\o\ud u.
\end{align}

Since $\th=-1$, we have to define the iterative sequence in a different fashion
\begin{align}
    \dt\lls_{k+1}+\lls_{k+1}\rs[u_k,\ell_k]&=\qs_{\text{gain}}[\ell_k,\ell_k;u_k],\\
    \dt\uus_{k+1}+\uus_{k+1}\rs[\ell_k,u_k]&=\qs_{\text{gain}}[u_k,u_k;\ell_k],
\end{align}
with initial data $\ell_{k+1}=F_0$ and $u_{k+1}=F_0$.

We would likee to select some special starting point $(\ell_0,u_0)$ and study the convergence property of $(\lls_k,\uus_k)$.

\begin{lemma}\label{vacuum-lemma: positivity 1.}
If $\nnm{F_0}$ and $\beta^{-2}(R_0+R_0^2)$ are sufficiently small, then there exists $\ell_0,u_0\in S_T$ such that the Beginning Condition (BC)
\begin{align}
    0\leq \ell_0\leq \ell_1\leq u_1\leq u_0
\end{align}
holds for $t\in[0,T]$.
\end{lemma}

\begin{proof}
We take $\ell_0=0$, which implies $\rs[\ell_0,\ell_0]=0$ and $\qs_{\text{gain}}[\ell_0,\ell_0;\ell_0]=0$. Hence, we have
\begin{align}
    \dt\lls_{1}+\lls_{1}\rs[u_0,0]&=0,\\
    \dt\uus_{1}&=\qs_{\text{gain}}[u_0,u_0;0].
\end{align}
Due to the positivity of $\rs[u_0,0]$ and $\qs_{\text{gain}}[u_0,u_0;0]$, this naturally implies
\begin{align}
    0\leq \lls_1\leq F_0\leq \uus_1.
\end{align}
The rest of the proof follows from that of Lemma \ref{vacuum-lemma: positivity 1}.
\end{proof}

\begin{lemma}\label{vacuum-lemma: positivity 2.}
If $\ell_0,u_0\in S_T$ such that the Beginning Condition (BC)
\begin{align}
    0\leq \ell_0\leq \ell_1\leq u_1\leq u_0
\end{align}
for $t\in[0,T]$, then the iterative sequence $(\ell_k,u_k)$ are always well-defined for $t\in[0,T]$ and satisfies 
\begin{align}
    \ell_{k}\leq \ell_{k+1}\leq u_{k+1}\leq u_{k}.
\end{align}
\end{lemma}

\begin{proof}
This follows naturally from that of Lemma \ref{vacuum-lemma: positivity 2} based on the monotonicity of $R$ and $Q_{\text{gain}}$.
\end{proof}

\begin{theorem}\label{vacuum-theorem: positivity.}
There exists a constant $R_0$ such that if $\nnm{F_0}$ and $\beta^{-2}(R_0+R_0^2)$ are sufficiently small with $F_0\geq0$, then the equation \eqref{equation: vacuum} for fermions has a unique mild solution $F\in S_{R_0}$ with $F\geq0$.
\end{theorem}

\begin{proof}
This follows from that of Theorem \ref{vacuum-theorem: positivity}.
\end{proof}

\begin{remark}
Since we consider the near vacuum case, $F\leq 1$ is naturally true when $R$ is small.
\end{remark}

\bigskip
\subsection*{Acknowledgements}

We would like to thank Ning Jiang for pointing out a mistake in the initial version of this paper. Also, we would like to thank Lingbing He and Xuguang Lu for suggesting some references. Lei Wu's research is supported in part by NSF grant DMS-1853002.

\bigskip



\phantomsection

\addcontentsline{toc}{section}{References}

\bibliographystyle{siam}
\bibliography{Reference}

\begin{thebibliography}{10}

\bibitem{Alexandre.Desvillettes.Villani.Wennberg2000}
{\sc R.~Alexandre, L.~Desvillettes, C.~Villani, and B.~Wennberg}, {\em Entropy
  dissipation and long-range interactions}, Arch. Ration. Mech. Anal., 152
  (2000), p.~327–355.

\bibitem{Arkeryd.Nouri2012}
{\sc L.~Arkeryd and A.~Nouri}, {\em Bose condensates in interaction with
  excitations: a kinetic model}, Comm. Math. Phys., 310 (2012), p.~765–788.

\bibitem{Arkeryd.Nouri2013}
\leavevmode\vrule height 2pt depth -1.6pt width 23pt, {\em A {Milne} problem
  from a {Bose} condensate with excitations}, Kinet. Relat. Models, 6 (2013),
  pp.~671--686.

\bibitem{Arkeryd.Nouri2015}
\leavevmode\vrule height 2pt depth -1.6pt width 23pt, {\em Bose condensates in
  interaction with excitations: a two-component space-dependent model close to
  equilibrium}, J. Stat. Phys., 160 (2015), p.~209–238.

\bibitem{Arkeryd.Nouri2017}
\leavevmode\vrule height 2pt depth -1.6pt width 23pt, {\em On the {Cauchy}
  problem with large data for a space-dependent {Boltzmann-Nordheim} boson
  equation}, Commun. Math. Sci., 15 (2017), p.~1247–1264.

\bibitem{Bae.Jang.Yun2021}
{\sc G.-C. Bae, J.~W. Jang, and S.-B. Yun}, {\em The relativistc quantum
  {Boltzmann} equation near equilibrium}, Arxiv: 2012.14213,  (2021).

\bibitem{Bandyopadhyay.Velazquez2015}
{\sc J.~Bandyopadhyay and J.~J.~L. Vel\'azquez}, {\em Blow-up rate estimates
  for the solutions of the bosonic {Boltzmann-Nordheim} equation}, J. Math.
  Phys., 56 (2015), p.~063302.

\bibitem{Benedetto.Castella.Esposito.Pulvirenti2004}
{\sc D.~Benedetto, F.~Castella, R.~Esposito, and M.~Pulvirenti}, {\em Some
  considerations on the derivation of the nonlinear quantum {Boltzmann}
  equation}, J. Statist. Phys., 116 (2004), p.~381–410.

\bibitem{Benedetto.Castella.Esposito.Pulvirenti2005}
\leavevmode\vrule height 2pt depth -1.6pt width 23pt, {\em On the weak-coupling
  limit for bosons and fermions}, Math. Models Methods Appl. Sci., 15 (2005),
  p.~1811–1843.

\bibitem{Benedetto.Castella.Esposito.Pulvirenti2006}
\leavevmode\vrule height 2pt depth -1.6pt width 23pt, {\em Some considerations
  on the derivation of the nonlinear quantum {Boltzmann} equation. {II}. the
  low density regime}, J. Stat. Phys., 124 (2006), p.~951–996.

\bibitem{Benedetto.Castella.Esposito.Pulvirenti2007}
\leavevmode\vrule height 2pt depth -1.6pt width 23pt, {\em A short review on
  the derivation of the nonlinear quantum {Boltzmann} equations}, Commun. Math.
  Sci., suppl. 1 (2007), p.~55–71.

\bibitem{Benedetto.Castella.Esposito.Pulvirenti2008}
\leavevmode\vrule height 2pt depth -1.6pt width 23pt, {\em From the {N}-body
  {Schr\"odinger} equation to the quantum {Boltzmann} equation: a term-by-term
  convergence result in the weak coupling regime}, Comm. Math. Phys., 277
  (2008), p.~1–44.

\bibitem{Briant2015}
{\sc M.~Briant}, {\em Instantaneous filling of the vacuum for the full
  {Boltzmann} equation in convex domains}, Arch. Ration. Mech. Anal., 218
  (2015), p.~985–1041.

\bibitem{Briant.Einav2016}
{\sc M.~Briant and A.~Einav}, {\em On the cauchy problem for the homogeneous
  {Boltzmann-Nordheim} equation for bosons: local existence, uniqueness and
  creation of moments}, J. Stat. Phys., 163 (2016), p.~1108–1156.

\bibitem{Cai.Lu2019}
{\sc S.~Cai and X.~Lu}, {\em The spatially homogeneous {Boltzmann} equation for
  {Bose-Einstein} particles: rate of strong convergence to equilibrium}, J.
  Stat. Phys., 175 (2019), p.~289–350.

\bibitem{Chen.Guo2015}
{\sc X.~Chen and Y.~Guo}, {\em On the weak coupling limit of quantum many-body
  dynamics and the quantum {Boltzmann} equation}, Kinet. Relat. Models, 8
  (2015), p.~443–465.

\bibitem{Colangeli.Pezzotti.Pulvirenti2015}
{\sc M.~Colangeli, F.~Pezzotti, and M.~Pulvirenti}, {\em A {Kac} model for
  fermions}, Arch. Ration. Mech. Anal., 216 (2015), p.~359–413.

\bibitem{Dolbeault1994}
{\sc J.~Dolbeault}, {\em Kinetic models and quantum effects: a modified
  {Boltzmann} equation for {Fermi-Dirac} particles}, Arch. Rational Mech.
  Anal., 127 (1994), p.~101–131.

\bibitem{Duan.Strain2011(=)}
{\sc R.~Duan and R.~M. Strain}, {\em Optimal large-time behavior of the
  {Vlasov-Maxwell-Boltzmann} system in the whole space}, Comm. Pure Appl.
  Math., 64 (2011), p.~1497–1546.

\bibitem{Duan.Strain2011}
\leavevmode\vrule height 2pt depth -1.6pt width 23pt, {\em Optimal time decay
  of the {Vlasov-Poisson-Boltzmann} system in $\mathbb{R}^3$}, Arch. Ration.
  Mech. Anal., 199 (2011), p.~291–328.

\bibitem{Duan.Yang.Zhu2005}
{\sc R.~Duan, T.~Yang, and C.~Zhu}, {\em Global existence to {Boltzmann}
  equation with external force in infinite vacuum}, J. Math. Phys., 46 (2005),
  p.~053307.

\bibitem{Erdos.Salmhofer.Yau2004}
{\sc L.~Erd\H{o}s, M.~Salmhofer, and H.-T. Yau}, {\em On the quantum
  {Boltzmann} equation}, J. Statist. Phys., 116 (2004), p.~367–380.

\bibitem{Escobedo.Mischler.Valle2005}
{\sc M.~Escobedo, S.~Mischler, and M.~A. Valle}, {\em Entropy maximisation
  problem for quantum relativistic particles}, Bull. Soc. Math. France, 133
  (2005), p.~87–120.

\bibitem{Escobedo.Mischler.Velazquez2007}
{\sc M.~Escobedo, S.~Mischler, and J.~J.~L. Vel\'azquez}, {\em On the
  fundamental solution of a linearized {Uehling-Uhlenbeck} equation}, Arch.
  Ration. Mech. Anal., 186 (2007), p.~309–349.

\bibitem{Escobedo.Mischler.Velazquez2008}
\leavevmode\vrule height 2pt depth -1.6pt width 23pt, {\em Singular solutions
  for the {Uehling-Uhlenbeck} equation}, Proc. Roy. Soc. Edinburgh Sect. A, 138
  (2008), p.~67–107.

\bibitem{Escobedo.Velazquez2008}
{\sc M.~Escobedo and J.~J.~L. Vel\'azquez}, {\em A derivation of a new set of
  equations at the onset of the {Bose-Einstein} condensation}, J. Phys. A, 41
  (2008), p.~395208.

\bibitem{Escobedo.Velazquez2015}
\leavevmode\vrule height 2pt depth -1.6pt width 23pt, {\em Finite time blow-up
  and condensation for the bosonic {Nordheim} equation}, Invent. Math., 200
  (2015), p.~761–847.

\bibitem{Glassey1996}
{\sc R.~T. Glassey}, {\em The {Cauchy} problem in kinetic theory}, Society for
  Industrial and Applied Mathematics (SIAM), Philadelphia, PA, 1996.

\bibitem{Gressman.Strain2011}
{\sc P.~T. Gressman and R.~M. Strain}, {\em Global classical solutions of the
  {Boltzmann} equation without angular cut-off}, J. Amer. Math. Soc., 24
  (2011), pp.~771--847.

\bibitem{Guo2001}
{\sc Y.~Guo}, {\em The {Vlasov-Poisson-Boltzmann} system near vacuum}, Comm.
  Math. Phys., 218 (2001), p.~293–313.

\bibitem{Guo2002(=)}
\leavevmode\vrule height 2pt depth -1.6pt width 23pt, {\em The {Landau}
  equation in a periodic box}, Comm. Math. Phys., 231 (2002), pp.~391--434.

\bibitem{Guo2002}
\leavevmode\vrule height 2pt depth -1.6pt width 23pt, {\em The
  {Vlasov-Poisson-Boltzmann} system near {Maxwellians}}, Comm. Pure Appl.
  Math., 55 (2002), pp.~1104--1135.

\bibitem{Guo2003(=)}
\leavevmode\vrule height 2pt depth -1.6pt width 23pt, {\em Classical solutions
  to the {Boltzmann} equation for molecules with an angular cutoff}, Arch.
  Ration. Mech. Anal., 169 (2003), p.~305–353.

\bibitem{Guo2003}
\leavevmode\vrule height 2pt depth -1.6pt width 23pt, {\em The
  {Vlasov-Maxwell-Boltzmann} system near {Maxwellians}}, Invent. Math., 153
  (2003), p.~593–630.

\bibitem{Guo2004}
\leavevmode\vrule height 2pt depth -1.6pt width 23pt, {\em The {Boltzmann}
  equation in the whole space}, Indiana Univ. Math. J., 53 (2004),
  p.~1081–1094.

\bibitem{Guo2010}
\leavevmode\vrule height 2pt depth -1.6pt width 23pt, {\em Decay and continuity
  of the {Boltzmann} equation in bounded domains}, Arch. Ration. Mech. Anal.,
  197 (2010), pp.~713--809.

\bibitem{Guo2012}
\leavevmode\vrule height 2pt depth -1.6pt width 23pt, {\em The
  {Vlasov-Poisson-Landau} system in a periodic box}, J. Amer. Math. Soc., 25
  (2012), pp.~759--812.

\bibitem{Hadzic.Guo2010}
{\sc M.~Had\v{z}i\'c and Y.~Guo}, {\em Stability in the {Stefan} problem with
  surface tension {(I)}}, Comm. Partial Differential Equations, 35 (2010),
  p.~201–244.

\bibitem{He.Lu.Pulvirenti2020}
{\sc L.-B. He, X.~Lu, and M.~Pulvirenti}, {\em On semi-classical limit of
  spatially homogeneous quantum {Boltzmann} equation: weak convergence}, To
  appear in Comm. Math. Phy.,  (2020).

\bibitem{Illner.Shinbrot1984}
{\sc R.~Illner and M.~Shinbrot}, {\em The {Boltzmann} equation: global
  existence for a rare gas in an infinite vacuum}, Comm. Math. Phys., 95
  (1984), p.~217–226.

\bibitem{Kim.Guo.Hwang2020}
{\sc J.~Kim, Y.~Guo, and H.~J. Hwang}, {\em An {$L^2$} to {$L^{\infty}$}
  framework for the landau equation}, Peking Math. J., 3 (2020), p.~131–202.

\bibitem{Lemou2000}
{\sc M.~Lemou}, {\em Linearized quantum and relativistic {Fokker-Planck-Landau}
  equations}, Math. Methods Appl. Sci., 23 (2000), p.~1093–1119.

\bibitem{Li.Lu2019}
{\sc W.~Li and X.~Lu}, {\em Global existence of solutions of the {Boltzmann}
  equation for {Bose-Einstein} particles with anisotropic initial data}, J.
  Funct. Anal., 276 (2019), p.~231–283.

\bibitem{Lu2000}
{\sc X.~Lu}, {\em A modified {Boltzmann} equation for {Bose-Einstein}
  particles: isotropic solutions and long-time behavior}, J. Statist. Phys., 98
  (2000), p.~1335–1394.

\bibitem{Lu2001}
\leavevmode\vrule height 2pt depth -1.6pt width 23pt, {\em On spatially
  homogeneous solutions of a modified {Boltzmann} equation for {Fermi-Dirac}
  particles}, J. Statist. Phys., 105 (2001), pp.~353--388.

\bibitem{Lu2004}
\leavevmode\vrule height 2pt depth -1.6pt width 23pt, {\em On isotropic
  distributional solutions to the {Boltzmann} equation for {Bose-Einstein}
  particles}, J. Statist. Phys., 116 (2004), p.~1597–1649.

\bibitem{Lu2005}
\leavevmode\vrule height 2pt depth -1.6pt width 23pt, {\em The {Boltzmann}
  equation for {Bose-Einstein} particles: velocity concentration and
  convergence to equilibrium}, J. Statist. Phys., 119 (2005), p.~1027–1067.

\bibitem{Lu2006}
\leavevmode\vrule height 2pt depth -1.6pt width 23pt, {\em On the {Boltzmann}
  equation for {Fermi-Dirac} particles with very soft potentials: averaging
  compactness of weak solutions}, J. Statist. Phys., 124 (2006), p.~517–547.

\bibitem{Lu2008}
\leavevmode\vrule height 2pt depth -1.6pt width 23pt, {\em On the {Boltzmann}
  equation for {Fermi-Dirac} particles with very soft potentials: global
  existence of weak solutions}, J. Differential Equations, 245 (2008),
  p.~1705–1761.

\bibitem{Lu2013}
\leavevmode\vrule height 2pt depth -1.6pt width 23pt, {\em The {Boltzmann}
  equation for {Bose-Einstein} particles: condensation in finite time}, J.
  Statist. Phys., 150 (2013), p.~1138–1176.

\bibitem{Lu2014}
\leavevmode\vrule height 2pt depth -1.6pt width 23pt, {\em The {Boltzmann}
  equation for {Bose-Einstein} particles: egularity and condensation}, J.
  Statist. Phys., 156 (2014), p.~493–545.

\bibitem{Lu2016}
\leavevmode\vrule height 2pt depth -1.6pt width 23pt, {\em Long time
  convergence of the {Bose-Einstein} condensation}, J. Statist. Phys., 162
  (2016), p.~652–670.

\bibitem{Lu2018}
\leavevmode\vrule height 2pt depth -1.6pt width 23pt, {\em Long time strong
  convergence to {Bose-Einstein} distribution for low temperature}, Kinet.
  Relat. Models, 11 (2018), p.~715–734.

\bibitem{Lu.Wennberg2003}
{\sc X.~Lu and B.~Wennberg}, {\em On stability and strong convergence for the
  spatially homogeneous {Boltzmann} equation for {Fermi-Dirac} particles},
  Arch. Ration. Mech. Anal., 168 (2003), pp.~1--34.

\bibitem{Lu.Zhang2011}
{\sc X.~Lu and X.~Zhang}, {\em On the {Boltzmann} equation for {2D
  Bose-Einstein} particles}, J. Statist. Phys., 143 (2011), p.~990–1019.

\bibitem{Maslova1993}
{\sc N.~B. Maslova}, {\em Nonlinear evolution equations. Kinetic approach},
  World Scientific Publishing Co., Inc., River Edge, NJ, USA, 1993.

\bibitem{Mouhot2005}
{\sc C.~Mouhot}, {\em Quantitative lower bounds for the full {Boltzmann}
  equation. {I.} periodic boundary conditions}, Comm. Partial Differential
  Equations, 30 (2005), p.~881–917.

\bibitem{Nguyen.Tran2019}
{\sc T.~T. Nguyen and M.-B. Tran}, {\em Uniform in time lower bound for
  solutions to a quantum {Boltzmann} equation of bosons}, Arch. Ration. Mech.
  Anal., 231 (2019), p.~63–89.

\bibitem{Pulvirenti2006}
{\sc M.~Pulvirenti}, {\em The weak-coupling limit of large classical and
  quantum systems}, International Congress of Mathematicians, {Vol. III}
  (2006), p.~229–256.

\bibitem{Silvestre2016}
{\sc L.~Silvestre}, {\em A new regularization mechanism for the {Boltzmann}
  equation without cut-off}, Comm. Math. Phys., 348 (2016), p.~69–100.

\bibitem{Spohn1994}
{\sc H.~Spohn}, {\em Quantum kinetic equations}, On Three Levels, Plenum Press
  (1994), pp.~1--10.

\bibitem{Spohn2010}
\leavevmode\vrule height 2pt depth -1.6pt width 23pt, {\em Kinetics of the
  {Bose-Einstein} condensation}, Phys. D, 239 (2010), p.~627–634.

\bibitem{Strain.Guo2004}
{\sc R.~M. Strain and Y.~Guo}, {\em Stability of the relativistic {Maxwellian}
  in a collisional plasma}, Comm. Math. Phys., 251 (2004), pp.~263--320.

\bibitem{Strain.Guo2006}
\leavevmode\vrule height 2pt depth -1.6pt width 23pt, {\em Almost exponential
  decay near {Maxwellian}}, Comm. Partial Differential Equations, 31 (2006),
  pp.~417--429.

\bibitem{Strain.Guo2008}
\leavevmode\vrule height 2pt depth -1.6pt width 23pt, {\em Exponential decay
  for soft potentials near {Maxwellian}}, Arch. Ration. Mech. Anal., 187
  (2008), pp.~287--339.

\bibitem{Uehling.Uhlenbeck1933}
{\sc E.~A. Uehling and G.~E. Uhlenbeck}, {\em Transport phenomena in
  {Einstein-Bose} and {Fermi-Dirac} gases. {I}}, Phys. Rev., 43 (1933), p.~552.

\bibitem{Villani2002}
{\sc C.~Villani}, {\em A review of mathematical topics in collisional kinetic
  theory}, Handbook of Mathematical Fluid Dynamics, Vol.{I} (2002),
  p.~71–305.

\bibitem{Zhang.Lu2002}
{\sc Y.~Zhang and X.~Lu}, {\em {Boltzmann} equations with quantum effects. {I.}
  long time behavior of spatial decay solutions}, Tsinghua Sci. Technol., 7
  (2002), p.~215–218.

\bibitem{Zhang.Lu2002(=)}
\leavevmode\vrule height 2pt depth -1.6pt width 23pt, {\em {Boltzmann}
  equations with quantum effects. {II.} entropy identity, existence and
  uniqueness of spatial decay solutions}, Tsinghua Sci. Technol., 7 (2002),
  p.~219–222.

\end{thebibliography}

\end{document}